\documentclass[12pt]{amsart}

\usepackage{verbatim}
\usepackage{fullpage}
\usepackage{amsfonts}
\usepackage{enumerate}
\usepackage{graphicx}
\usepackage{amsmath, amsthm, amssymb, bbm}
\usepackage{fancybox}
\usepackage{xcolor}

\usepackage{tikz,varwidth}
\usetikzlibrary{shapes.geometric, arrows}
\tikzstyle{stop} = [rectangle, rounded corners, minimum width=3cm, minimum height=1cm,text centered,  text width=4.5cm, draw=black, fill=green!30]
\tikzstyle{no} = [rectangle, rounded corners, minimum width=3cm, minimum height=1cm,text centered,  text width=4.5cm, draw=black, fill=red!30]
\tikzstyle{kernel} = [rectangle, rounded corners, minimum width=3cm, minimum height=1cm,text centered,  text width=4.5cm, draw=black, fill=yellow!30]
\tikzstyle{io} = [trapezium, trapezium left angle=70, trapezium right angle=110, minimum width=1cm, minimum height=1cm, text width=10cm, text centered, draw=black]
\tikzstyle{process} = [rectangle, minimum width=3cm, minimum height=1cm, text centered, text width = 4.5cm, draw=black]
\tikzstyle{decision} = [diamond, draw=black, text width=6em, text badly centered, inner sep=1pt]
\tikzstyle{empty} = []
\tikzstyle{arrow} = [thick,->]

\newcommand\iprec{\mathrel{\ooalign{$\prec$\cr
 \,\raise0.85ex\hbox{\scriptsize$\circ$}\cr}}}

\renewcommand{\tilde}{\widetilde}

\newcommand{\R}{\mathbb{R}}

\newcommand{\N}{\mathbb{N}}
\newcommand{\Z}{\mathbb{Z}}

\newcommand{\Hil}{\mathcal{H}}

\newcommand{\eps}{\varepsilon}

\DeclareMathOperator{\lspan}{span}
\DeclareMathOperator{\tr}{tr}
\DeclareMathOperator{\ran}{ran}

\newcounter{Theorem}

\numberwithin{equation}{section}
\numberwithin{Theorem}{section}

\theoremstyle{plain} 
\newtheorem{thm}[Theorem]{Theorem}
\newtheorem{cor}[Theorem]{Corollary}
\newtheorem{lem}[Theorem]{Lemma}
\newtheorem{prop}[Theorem]{Proposition}

\theoremstyle{definition}
\newtheorem{defn}[Theorem]{Definition}
\newtheorem{problem}[Theorem]{Problem}

\theoremstyle{remark}
\newtheorem{remark}[Theorem]{Remark}
\newtheorem{ex}[Theorem]{Example}
\newtheorem{nota}[Theorem]{Notation}

\begin{document}
\title[Diagonals of self-adjoint operators]{Diagonals of self-adjoint operators II: 
\\
Non-compact operators}

\begin{abstract}
Given a self-adjoint operator $T$ on a separable infinite-dimensional Hilbert space we study the problem of characterizing the set $\mathcal D(T)$ of all possible diagonals of $T$. For operators $T$ with at least two points in their essential spectrum $\sigma_{ess}(T)$, we give a complete characterization of $\mathcal D(T)$ for the class of self-adjoint operators sharing the same spectral measure as $T$ with a possible exception of multiplicities of eigenvalues at the extreme points of $\sigma_{ess}(T)$. 
We also give a more precise description of $\mathcal D(T)$ for a fixed self-adjoint operator $T$, albeit modulo the kernel problem for special classes of operators. These classes consist of operators $T$ for which an extreme point of the essential spectrum $\sigma_{ess}(T)$ is also an extreme point of the spectrum $\sigma(T)$.

Our results generalize a characterization of diagonals of orthogonal projections by Kadison \cite{k1, k2}, Blaschke-type  results of M\"uller and Tomilov \cite{mt} and Loreaux and  Weiss \cite{lw3}, and a characterization of diagonals of operators with finite spectrum by the authors \cite{npt}.
\end{abstract}

\author{Marcin Bownik}
\author{John Jasper}

\address{Department of Mathematics, University of Oregon, Eugene, OR 97403--1222, USA}

\email{mbownik@uoregon.edu}

\address{Department of Mathematics \& Statistics, Air Force Institute of Technology, Wright Patterson AFB, OH 45433, USA}

\email{john.jasper@afit.edu}




\thanks{The first author was partially supported by NSF grant DMS-1956395. The views expressed in this article are those of the authors and do not reflect the official policy or position of the United States Air Force, Department of Defense, or the U.S. Government.}

\maketitle


\section{Introduction}

The study of diagonals of bounded linear operators acting on infinite-dimensional Hilbert spaces has a long history. A systematic study of the set of diagonals of an operator was pioneered by Herrero \cite{herr}. In particular, he gave sufficient conditions for a sequence to be a diagonal in terms of the essential numerical range $W_e(T)$ of an operator $T$, see \cite{fsw}. Even earlier, Fan \cite{fan} jointly with Fong and Herrero \cite{ff, ffh} studied the existence of zero diagonals of bounded linear operators. More recently in this vein, Loreaux and Weiss \cite{lw2} have shown that an infinite-rank idempotent admits a zero diagonal if and only if it is not a Hilbert-Schmidt perturbation of an orthogonal projection.

M\"uller and Tomilov \cite{mt} proved a far reaching extension of Herrero's result \cite{herr}. Any sequence $(d_i)_{i\in \N}$ contained in the interior of $W_e(T)$ and satisfying the non-Blaschke condition
\begin{equation}\label{nbla}
\sum_{i\in \N} \operatorname{dist}(d_i, \R \setminus W_e(T)) =\infty
\end{equation}
is a diagonal of a self-adjoint operator $T$. Moreover, their result extends to tuples of self-adjoint operators. M\"uller and Tomilov \cite{mt2} also proved that the set of all possible constant diagonals of a bounded Hilbert space operator $T$ is always convex. In a related paper \cite{mt3} they studied matrix representations for $T$ having certain specified algebraic or asymptotic structure. 

Among other results on diagonals of operators, Fong \cite{fong} proved that any bounded sequence of complex numbers is a diagonal of a nilpotent operator. Jasper, Loreaux, and Weiss \cite{jlw} proved an infinite-dimensional extension of Thompson's theorem \cite{th} for compact operators and they characterized diagonals of unitary operators.  A complex-valued sequence $(d_i)_{i\in\N}$ is a diagonal of some unitary operator on $\mathcal H$ if and only if $\sup |d_i| \le 1$ and
\[
2(1-\inf_{i\in \N} |d_i|) \le \sum_{i\in\N} (1-|d_i|).
\]

Bourin and Lee \cite{bou, bou2} investigated a general setting of operator-valued diagonals defined as compressions $(P_{M_n}TP_{M_n})_{n\in \N}$ of an operator $T$ with respect to a collection of mutually orthogonal subspaces $( M_n)_{n\in\N}$ of a Hilbert space $\mathcal H$ such that $\bigoplus_{n\in N} M_n = \mathcal H$, where $P_{M_n}$ denotes an orthogonal projection onto $M_n$. In the case when all subspaces $M_n$ are $1$-dimensional such compressions correspond to a diagonal of $T$.


For finite-dimensional spaces, the classical Schur-Horn theorem \cite{horn, moa, schur} characterizes diagonals of hermitian matrices in terms of their eigenvalues. A sequence $(d_1,\ldots,d_N)$ is a diagonal of a hermitian matrix with eigenvalues $(\lambda_1,\ldots,\lambda_N)$ if and only if 
\begin{equation}\label{horn2}
(d_1,\ldots,d_N) \in \operatorname{conv} \{ (\lambda_{\sigma(1)},\ldots,\lambda_{\sigma(N)}): \sigma \in S_N\}.
\end{equation}
This characterization has attracted  significant interest and has been generalized in many remarkable ways. Some major milestones are the Kostant convexity theorem \cite{ko} and the convexity of moment mappings in symplectic geometry \cite{at, gs1, gs2}. 

An infinite-dimensional extension of the Schur-Horn theorem has been a subject of intensive study in recent years. In particular, we are interested in the following problem generalizing the Schur-Horn theorem.

\begin{problem}\label{DT}
Given a bounded linear operator $T$ on a separable Hilbert space $\mathcal H$,  characterize the set of all diagonals 
\begin{equation}\label{dt}
\mathcal D(T) = \{ (\langle T e_i, e_i \rangle)_{i\in \N}: (e_i)_{i\in\N} \text { is an orthonormal basis of } \mathcal H \} \subset \ell^{\infty}(\N).
\end{equation}
\end{problem}

Neumann \cite{neu} gave an initial, albeit approximate, solution to this problem for self-adjoint operators $T$ by identifying the $\ell^\infty$-closure of $\mathcal D(T)$ with a convex set generalizing condition \eqref{horn2}. Antezana, Massey, Ruiz, and Stojanoff \cite{amrs}
obtained a refinement of results of Neumann \cite{neu}.
However, the first fully satisfactory result in this direction was shown by Kadison in his influential work \cite{k1, k2}. Kadison discovered a characterization of diagonals of orthogonal projections on separable Hilbert spaces. A sequence
$(d_i)_{i\in \N}$ in $[0,1]$ is a diagonal of an orthogonal projection if and only if the sums 
\begin{equation}\label{kab}
a:=\sum_{d_i <1/2} d_i \qquad\text{and}\qquad b:=\sum_{d_i \ge 1/2} (1-d_i)
\end{equation}
satisfy either $a+b=\infty$, or $a+b<\infty$ and $a-b\in\Z$. Constructive proofs of the sufficiency part of Kadison's theorem, also known as the Carpenter's theorem, were given by the authors \cite{mbjj2} and Argerami \cite{ar}. Kaftal and Loreaux \cite{kl} gave an insightful proof of the necessity part of Kadison's theorem by identifying the integer $a-b$ with the so-called essential codimension of a pair of projections. Bownik and Szyszkowski \cite{bsz} proved a measurable extension of Kadison's theorem.

The natural extension of this problem to characterizing $\mathcal D(T)$ for normal operators $T$ remains mostly intractable. In the infinite-dimensional case Arveson \cite{a} gave a necessary condition on diagonals of operators with finite spectrum. In the finite-dimensional case the diagonals of $3\times 3$ normal matrices were completely characterized by Williams \cite{will}, but the problem remains open in dimensions $>3$.
Several authors \cite{am, am2, am3, rara, dfhs, ks, mr, ravi} have studied an extension of the Schur-Horn problem in von Neumann algebras, which was originally proposed by Arveson and Kadison \cite{ak}.

Arveson and Kadison \cite{ak} and Kaftal and Weiss \cite{kw} extended the Schur-Horn theorem to positive trace class operators and compact positive operators, respectively. These results are stated in terms of majorization inequalities. 
Kaftal and Weiss \cite{kw} showed that a nonincreasing sequence of positive numbers $d_1\ge d_2 \ge \ldots>0$ converging to $0$ is a diagonal of a positive compact operator with trivial kernel and positive eigenvalues $\lambda_1 \ge \lambda_2 \ge \ldots>0$, listed with multiplicity, if and only if
\begin{equation}\label{horn1}
\sum_{i=1}^{\infty}d_{i}= \sum_{i=1}^{\infty}\lambda_{i}
\qquad\text{and}\qquad
\sum_{i=1}^{n}d_{i}\leq \sum_{i=1}^{n}\lambda_{i}\quad\text{for all }n\in\N.
\end{equation}
Loreaux and Weiss \cite{lw} gave necessary conditions and sufficient conditions on $\mathcal D(T)$ for compact positive operators $T$ with nontrivial kernel.

Beyond positive compact operators, the authors \cite{mbjj} characterized the set $\mathcal D(T)$ for a class of locally invertible positive operators $T$. This result was used to characterize sequences of norms of a frame with prescribed lower and upper frame bounds, extending Kadison's theorem, which characterizes sequences of norms of Parseval frames. The connection between frame theory \cite{ch, hl} and the study of diagonals was investigated by Antezana, Massey, Ruiz, and Stojanoff \cite{amrs}. For more results in this direction see \cite{mbjj3, mbjj6, dfklow, klar}.

Beyond orthogonal projections, the second author \cite{jj} extended Kadison's result \cite{k1,k2} to characterize the set of diagonals $\mathcal D(T)$ of a self-adjoint operator $T$ with three points in the spectrum. Subsequently, the authors have characterized the set of diagonals $\mathcal D(T)$ of self-adjoint operators with finite spectrum \cite{npt, mbjj5} by introducing Lebesgue type majorization extending the Riemann type majorization  \eqref{horn1}. 

Siudeja and the authors \cite{unbound} proved a version of the Schur-Horn theorem for unbounded self-adjoint operators with discrete spectrum. All previous results only considered bounded operators. Another important development in this paper is a new type of result which we refer to as ``diagonal-to-diagonal". For example,  Proposition \ref{kwd2d}, which is the extension of the result of Kaftal and Weiss \cite{kw} mentioned above, is an example of such a diagonal-to-diagonal result. That is, given positive nonincreasing sequences \((d_{i})_{i\in\N}\) and \((\lambda_{i})_{i\in\N}\) which converge to zero and satisfy \eqref{horn1}, if \((\lambda_{i})_{i\in\N}\) is a \textit{diagonal} of some (not necessarily bounded) symmetric operator \(T\), then \((d_{i})_{i\in\N}\) is also a diagonal of \(T\).

The most recent progress on Problem~\ref{DT} was shown by Loreaux and Weiss \cite{lw3}. They resolved Blaschke's enigma, introduced by M\"uller and Tomilov \cite{mt}, for self-adjoint operators whose essential spectrum contains at least three points including the extreme points of the spectrum. If $T$ is a self-adjoint operator with essential spectrum satisfying $\#|\sigma_{ess}(T)|\ge 3$ and the spectrum of $T$ is contained in its essential numerical range $\sigma(T) \subset W_{e}(T)$, then the sufficient condition \eqref{nbla} is also necessary, see \cite[Theorem 4.6]{lw3}.
In summary, a complete answer to Problem~\ref{DT} was previously known only for compact positive operators with either trivial or infinite-dimensional kernel \cite{kw, lw}, for operators with finite spectrum \cite{npt, k1, k2}, or for the above class of operators with $\ge 3$ points in their essential spectrum \cite{lw3}.

The goal of this paper is to characterize the set of diagonals $\mathcal D(T)$ of an arbitrary self-adjoint operator $T$. In our previous paper \cite{mbjj7} we gave a characterization of $\mathcal D(T)$ for compact self-adjoint operators, which extended earlier results for positive compact operators in \cite{kw, lw}. In the current paper we characterize the set of diagonals $\mathcal D(T)$ for self-adjoint operators with at least 2 points in their essential spectrum, extending earlier results of the authors \cite{npt} for operators with finite spectrum.

In the case when $T$ has at least two points in its essential spectrum $\sigma_{ess}(T)$, by rescaling we can assume that the smallest and the largest point of $\sigma_{ess}(T)$ are $0$ and $1$, respectively. Our main result is Theorem \ref{mainthm}, which yields a characterization of $\mathcal D(T)$ for the class of self-adjoint operators sharing the same spectral measure as $T$ with the possible exception of the multiplicities of the eigenvalues $0$ and $1$. 

We also give a more precise description of $\mathcal D(T)$ for a fixed self-adjoint operator $T$, albeit modulo the kernel problem for three classes of operators for which an extreme point of the essential spectrum of $T$ is also an extreme point of the spectrum of $T$. The first class consists of positive compact operators, where the kernel problem makes its initial appearance, see Theorem \ref{kp}. The second class consists of diagonalizable operators whose spectrum is contained in $[0,1]$ and its eigenvalues sequence $(\lambda_i)$ satisfies the Blaschke condition: $\sum \min(\lambda_i,1-\lambda_i)<\infty$. The third class is an extension of the second class allowing the spectrum to be contained either in $(-\infty,1]$ or $[0,\infty)$, and requiring a similar Blaschke condition for eigenvalues in $(0,1)$.

Our description is given by a series of algorithms, which are grouped according to the cardinality of $\sigma_{ess}(T)$, that determine whether a numerical sequence is a diagonal of $T$, or not. The algorithm for operators with $\ge 5$ points in their essential spectrum uses a recursive pruning procedure and is always conclusive unless it is reduced to $4$-point algorithm. The algorithm for an operator $T$ with $4$-point essential spectrum is conclusive unless $T$ decouples into a pair of a compact operators and an operator with $3$-point essential spectrum, which reduces the outcome to the kernel problem (for positive compact operators) and $3$-point algorithm. A similar scenario plays out for an operator $T$ with $3$-point essential spectrum. The most interesting, but also the most complicated, algorithm for an operator $T$ with $2$-point essential spectrum is conclusive unless the question is reduced to one of the three above mentioned classes of the kernel problem or an application of the $1$-point algorithm \cite[Section 11]{mbjj7}. In turn, $1$-point algorithm for determining diagonals might be inconclusive only when it reduces to the kernel problem for positive compact operators. In summary, a conglomeration of our algorithms is inconclusive only if it reduces to the kernel problem for the three specific classes of operators described in the previous paragraph.

Moreover, we characterize diagonals of self-adjoint operators with uncountable spectrum, which includes all non-diagonalizable self-adjoint operators, see Theorem \ref{noncount}. This characterization does not take the form of a finite list of conditions, but rather we show that the algorithm for operators with \(\geq 5\)-point essential spectrum is always conclusive, though possibly after an infinite number of steps.

Even for nonnegative nonincreasing sequences in \(c_{0}\), the concept of majorization for infinite sequences can be quite subtle, see \cite{kw,kw1,lw}. Since the sequences in this paper can not, in general, be rearranged into nonincreasing order, we need to develop a more general concept of majorization which uses the following majorization function. This is distinct from the Hardy-Littlewood-P\'olya majorization of $L^1$ functions used in the Schur-Horn problem in von Neumann algebras \cite{am, rara, dhs, dfhs, mr, ravi}.

\begin{defn}\label{d81} Given a sequence $\boldsymbol{d}=(d_{i})_{i\in I}$, for each $\alpha\in\R$ define
\[C_{\boldsymbol{d}}(\alpha) = \sum_{0\leq d_{i}<\alpha}d_{i}\quad\text{and}\quad D_{\boldsymbol{d}}(\alpha) = \sum_{\alpha\leq d_{i}\leq 1}(1-d_{i}).\]
Define the \textit{majorization function} $f_{\boldsymbol{d}}: \R \to [0,\infty]$ by
\[f_{\boldsymbol{d}}(\alpha) =\begin{cases}
\sum_{d_{i}\leq \alpha}(\alpha-d_{i}) & \alpha \le0\\
(1-\alpha)C_{\boldsymbol{d}}(\alpha) + \alpha D_{\boldsymbol{d}}(\alpha) & \alpha\in(0,1)\\
\sum_{d_{i}\geq \alpha}(d_{i}-\alpha) & \alpha \ge 1.
\end{cases}
\]
\end{defn}

We will say that $x$ is an \textit{accumulation point} of a sequence $\boldsymbol\lambda=(\lambda_{i})_{i\in I}$ if for every $\eps>0$ there are infinitely many $i\in I$ such that $|x-\lambda_{i}|<\eps$. Note that a point $x$ is an accumulation point of $\boldsymbol\lambda$ if and only if it is a limit point of the set of values which $\boldsymbol\lambda$ takes, or $\{i\in I\colon \lambda_{i}=x\}$ is infinite.

\begin{defn}
We say that two operators $T$ and $T'$ are {\it unitarily equivalent modulo $\{0,1\}$} if $T\oplus P$ and $T' \oplus P$ are unitarily equivalent, where $P$ denotes an orthogonal projection on a Hilbert space with infinite-dimensional kernel and infinite-dimensional range.
\end{defn}

Note that in the case $T$ and $T'$ are unitarily equivalent modulo $\{0,1\}$ and $T$ is diagonalizable, then $T'$ is diagonalizable and the multiplicities of all  eigenvalues $\lambda \not = 0,1$ of $T$ and $T'$ coincide. In general, $T$ and $T'$ are unitarily equivalent modulo $\{0,1\}$ if and only if the restrictions of the spectral measures of $T$ and $T'$ to $\R\setminus \{0,1\}$ are unitarily equivalent. Our main result characterizing diagonals of non-compact operators, modulo two extreme points of the essential spectrum, takes the following form.

\begin{thm}\label{mainthm}
Let $T$ be a self-adjoint operator on $\mathcal H$ such that $0$ and $1$ are two extreme points of the essential spectrum $\sigma_{ess}(T)$ of $T$, that is,
\begin{equation*}
\{0,1 \} \subset \sigma_{ess}(T) \subset [0,1].
\end{equation*}
Let $\boldsymbol\lambda=(\lambda_{j})_{j\in J}$ be the list of all eigenvalues of $T$ with multiplicity (which is possibly an empty list). 
Let $\boldsymbol{d} = (d_{i})_{i\in \N}$. Define
\[\delta^{L} = \liminf_{\beta\nearrow 0}(f_{\boldsymbol\lambda}(\beta)-f_{\boldsymbol{d}}(\beta))\quad\text{and}\quad \delta^{U} = \liminf_{\beta\searrow 1}(f_{\boldsymbol\lambda}(\beta)-f_{\boldsymbol{d}}(\beta)).\]
Then, $\boldsymbol{d}$  is a diagonal of an operator which is unitarily equivalent to $T$ modulo $\{0,1\}$  if and only if 
\begin{equation}
\label{mainthm1}
 f_{\boldsymbol{d}}(\alpha)  \leq f_{\boldsymbol\lambda}(\alpha)  
\qquad\text{for all }\alpha\in\R\setminus[0,1],
\end{equation}
and one of the following five conditions holds:
\begin{enumerate}
\item
$\delta^L=\delta^U =\infty$,
\item
$\delta^L<\infty$, $\delta^U=\infty$, and $\sum_{d_i>0} d_i=\infty$,
\item
$\delta^U<\infty$, $\delta^L=\infty$, and $\sum_{d_i<1} (1-d_i)=\infty$,
\item 
$\delta^L+\delta^U <\infty$ and $f_{\boldsymbol d}(1/2) =\infty$, or
\item
$\delta^L+\delta^U <\infty$, $f_{\boldsymbol \lambda}(1/2) <\infty$, $f_{\boldsymbol d}(1/2) <\infty$, $\boldsymbol d$ has accumulation points at $0$ and $1$, and there exists $\delta_0 \in [0,\delta^L]$ and $\delta_1 \in [0,\delta^U]$ such that
\begin{align}
\label{mainthm2}
f_{\boldsymbol\lambda}(\alpha)  &\leq f_{\boldsymbol{d}}(\alpha) + (1-\alpha)\delta_{0} + \alpha\delta_{1} & \text{for all }\alpha\in(0,1),\\
\label{mainthm3}f_{\boldsymbol\lambda}'(\alpha)-f_{\boldsymbol{d}}'(\alpha)  & \equiv \delta_1-\delta_0 \mod 1 & \text{for some }\alpha\in(0,1),
\\
\label{mainthm4}
\sum_{\lambda_i<0} |\lambda_i|<\infty &\implies
\delta^L-\delta_0 \le \delta^U- \delta_1,
\\
\label{mainthm5}
\sum_{\lambda_i>1} (\lambda_i-1)<\infty
&\implies \delta^U- \delta_1 \le \delta^L-\delta_0.
\end{align}
\end{enumerate}
\end{thm}

The conditions (i)--(iv) have a qualitative character, which lacks any quantitative conditions when at least one the excesses $\delta^L$ or $\delta^U$, evaluated between eigenvalue list $\boldsymbol \lambda$ and diagonal $\boldsymbol d$ and beyond the extreme points of essential spectrum, are infinite. In the case when both $\delta^L$ and $\delta^U$ are finite, we either have a qualitative non-Blaschke condition on diagonal terms between $(0,1)$,
\[
\sum_{i: 0<d_i<1} \min(d_i,1-d_i) =\infty,
\]
or we have the extremely elaborate quantitative conditions in (v). 

It turns out that any operator $E$ with diagonal $\boldsymbol d$ satisfying $\delta^L+\delta^U<\infty$ and the Blaschke condition
\begin{equation}\label{blad}
\sum_{i: 0<d_i<1} \min(d_i,1-d_i) <\infty,
\end{equation}
is actually diagonalizable, see Theorem \ref{nec}. The eigenvalue list $\boldsymbol \lambda$ of $E$ must satisfy not only the Blaschke condition
\begin{equation}\label{blal}
\sum_{i: 0<\lambda_i<1} \min(\lambda_i,1-\lambda_i) <\infty,
\end{equation}
but also the Lebesgue majorization inequality \eqref{mainthm2} expressed in terms of the majorization functions $f_{\boldsymbol{\lambda}}$ and $f_{\boldsymbol{d}}$. This is an extension of the Lebesgue interior majorization inequalities already present for operators with finite spectrum \cite[Theorem 1.3]{npt}. The condition \eqref{mainthm3} is an analogue of the trace condition $a-b\in \Z$ for finite sums \eqref{kab} in Kadison's theorem. In contrast, the last two conditions \eqref{mainthm4} and \eqref{mainthm5} do not have an obvious analogue for operators with finite spectrum as they collapse to the equality $\delta^L-\delta_0=\delta^U-\delta_1=0$. 
Indeed, the quantities $\delta_0$ and $\delta_1$ measure how much of mass of the eigenvalue sequence $\boldsymbol \lambda$ has crossed over extreme points of the essential spectrum, $0$ and $1$ respectively, in the process of producing the diagonal $\boldsymbol d$. In other words, the quantities $\delta^L-\delta_0$ and $\delta^U-\delta_1$ represent how much of the mass of $\boldsymbol \lambda$ has disappeared within the exterior majorization inequality \eqref{mainthm1}. This is similar to the case of non-trace class compact positive operators where the excess $\sigma_+$ could be positive.
In light of \eqref{mainthm1}, we could replace the summability of \(\boldsymbol{\lambda}_{-}\) and \((\lambda_{i}-1)_{\lambda_{i}>1}\) in \eqref{mainthm4} and \eqref{mainthm5}, respectively, with summability of \(\boldsymbol{d}_{-}\) and \((d_{i}-1)_{d_{i}>1}\). 
Hence, together \eqref{mainthm3}, \eqref{mainthm4}, and \eqref{mainthm5} can be viewed as the trace condition on diagonals of self-adjoint operators with two points in their essential spectrum.

To establish the sufficiency part in Theorem \ref{mainthm} we show the equivalence of Riemann and Lebesgue interior majorization for nondecreasing sequences $\boldsymbol \lambda$ and $\boldsymbol d$ in $[0,1]$, which are indexed by the integers, and satisfy the Blaschke condition \eqref{blad} and \eqref{blal}. As a consequence, we establish diagonal-to-diagonal result for sequences in $(0,1)$ satisfying the majorization \eqref{mainthm2} and trace  \eqref{mainthm3} conditions. This is then used to prove the main diagonal-to-diagonal result for sequences $\boldsymbol \lambda$ and $\boldsymbol d$ satisfying (v), see Theorem~\ref{suff}. In addition, we need to impose strictness of the inequality \eqref{mainthm2}, which is imposed by the  decoupling phenomenon, see Theorem \ref{nec}. In a nutshell, decoupling at a point $\alpha \in [0,1]$ breaks the operator $E$ into two parts with spectrum in $(-\infty,\alpha]$ and $[\alpha,\infty)$, and splits the diagonal sequence $\boldsymbol d$ at the same point $\alpha$ into two subsequences, which become diagonals of two parts of $E$, respectively. Once decoupling happens, more necessary conditions arise, which imposes a limitation on the generality of sufficiency results.  

The proof of the sufficiency of the remaining conditions (i)--(iv) in Theorem \ref{mainthm} is exceedingly concise given how large of a class of self-adjoint operators it covers. It reflects a general observation that sufficiency proofs are hardest and most convoluted for simple operators such as projections, compact operators, and diagonalizable operators satisfying the Blaschke condition \eqref{blal}, while they become simpler for more complicated operators, such as non-diagonalizable operators. This is due to a theorem of M\"uller and Tomilov \cite{mt}, which guarantees that any sequence in the interior of the essential numerical range $W_e(E)$ and satisfying \eqref{nbla} is a diagonal of $E$. Hence, the sufficiency proof is mostly reduced to considerations involving exterior majorization \eqref{mainthm1}, see Theorem \ref{nbd}. This result is exceptional, compared with the rest of the paper, as all earlier sufficiency theorems are diagonal-to-diagonal. The proof of Theorem \ref{mainthm} concludes in Section \ref{S8}, which combines earlier necessity and sufficiency results. In addition, we obtain a stronger variant of Theorem \ref{mainthm}, which describes the set of diagonals $\mathcal D(T)$ for a single self-adjoint operator $T$, albeit with a gap between necessary and sufficient conditions.

The rest of the paper is devoted to algorithms for determining whether a given sequence is a diagonal of a self-adjoint operator. In particular, we answer Problem~\ref{DT} for all non-diagonalizable self-adjoint operators, see Remark \ref{concl5}.

\begin{thm}\label{noncount}
Let $E$ be a self-adjoint operator  such that its spectrum $\sigma(E)$ is uncountable. Let $\boldsymbol d$ be a bounded sequence of real numbers. Then, the algorithm for determining diagonals of self-adjoint operators 
with $\geq 5$-point essential spectrum in Section \ref{S21} is conclusive for the pair $(E,\boldsymbol d)$.
\end{thm}

We conclude the paper by discussing the kernel problem for two classes of non-compact operators, for which our algorithms are inconclusive. We show that operators in these classes need to satisfy additional necessary conditions beyond those arising from Theorem \ref{nec}. These are analogues of necessary conditions for compact positive operators with nontrivial kernel shown by Kaftal and Loreaux \cite{kl}. Due to the gap between necessary and sufficient conditions, Problem~\ref{DT} remains open only for these very special classes of operators.

\section{Preliminaries}

In this section we state previously known results about diagonals of operators which will be used in subsequent sections. To describe our results in detail, we need to set some notation.
Let \(I\) be a countably infinite set.
Let $c_0(I)$ be the collection of real-valued sequences converging to $0$ and indexed by the set \(I\). That is, the set \(\{i\in I: |d_{i}|>\eps\}\) is finite for every \(\eps>0\). Let \(c_{0}^{+}(I)\) be the set of nonnegative-valued sequences in \(c_{0}(I)\). We will write \(c_{0}\) and \(c_{0}^{+}\) when the indexing set of the sequence is obvious.

\begin{defn}
Let $\boldsymbol\lambda =(\lambda_i)_{i\in I}$ be a real-valued sequence. Define its {\it positive part} $\boldsymbol\lambda_+ =(\lambda^{+}_i)_{i\in I}$ by $\lambda^+_i=\max(\lambda_i,0)$. The {\it negative part} is defined as $\boldsymbol\lambda_-=(-\boldsymbol \lambda)_+$.
If $\boldsymbol\lambda \in c_0^+$, then define its {\it decreasing rearrangement} $\boldsymbol\lambda^{\downarrow} =(\lambda^{\downarrow}_i)_{i\in \N}$ by taking $\lambda^{\downarrow}_i$ to be the $i$th largest term of $\boldsymbol \lambda$. For the sake of brevity, we will denote the $i$th term of $(\boldsymbol\lambda_{+})^{\downarrow}$ by $\lambda_{i}^{+\downarrow}$, and similarly for $(\boldsymbol\lambda_{-})^{\downarrow}$.
\end{defn}

\begin{defn} Given two sequences \(\boldsymbol\lambda = (\lambda_{j})_{j\in J}\) and \(\boldsymbol{d}=(d_{i})_{i\in I}\), the \textit{concatenation} of \(\boldsymbol{\lambda}\) and \(\boldsymbol{d}\), denoted \(\boldsymbol\lambda\oplus\boldsymbol{d}\), is the sequence \((c_{k})_{k\in I\sqcup J}\) where \(I\sqcup J\) is the disjoint union of \(I\) and \(J\), and \[c_{k} = \begin{cases} \lambda_{k} & k\in J,\\ d_{k} & k\in I.\end{cases}\]
Note, if \(I\cap J=\varnothing\), then \(I\sqcup J:=I\cup J\), if \(I\cap J\neq\varnothing\), then \(I\sqcup J= (I\times\{1\})\cup (J\times\{2\})\). In the latter case \(k\in I\) is interpreted to mean \((k,1)\in I\times\{1\}\), and similarly for \(k\in J\).
\end{defn}

\subsection{Majorization of sequences} There are two ways in which we can define majorization of sequences.

\begin{defn}\label{rmajdef}
Let \(I,J\) be countable infinite sets, and let \(\boldsymbol\lambda = (\lambda_{i})_{i\in I}\) and \(\boldsymbol{d}=(d_{i})_{i\in J}\) be sequences in \(c_{0}^{+}\). We say that \(\boldsymbol\lambda\) {\it majorizes} \(\boldsymbol d\) and write \(\boldsymbol d \prec \boldsymbol \lambda\) if
\begin{equation}\label{rmajdef1}\sum_{i=1}^{n}d_{i}^{\downarrow}\leq \sum_{i=1}^{n}\lambda_{i}^{\downarrow}\quad\text{for all }n\in\N.\end{equation}
If, in addition, we have
\[\liminf_{n\to\infty}\sum_{i=1}^{n}(\lambda_{i}^{\downarrow}-d_{i}^{\downarrow})=0\]
then we say that \(\boldsymbol\lambda\) {\it strongly majorizes} \(\boldsymbol d\), and we write \(\boldsymbol d\preccurlyeq\boldsymbol\lambda\).
\end{defn}

\begin{defn}\label{deltadef}
Given two sequences $\boldsymbol{d}\in c_{0}(I)$ and $\boldsymbol\lambda\in c_{0}(J)$, for $\alpha\in\R\setminus\{0\}$ we define the function
\begin{equation}\label{delta}\delta(\alpha,\boldsymbol\lambda,\boldsymbol{d}) = \begin{cases} \displaystyle\sum_{\lambda_{i}\geq\alpha}(\lambda_{i}-\alpha)-\sum_{d_{i}\geq\alpha}(d_{i}-\alpha) & \alpha>0,\\
\displaystyle\sum_{\lambda_{i}\leq\alpha}(\alpha-\lambda_{i})-\sum_{d_{i}\leq\alpha}(\alpha-d_{i}) & \alpha<0.
\end{cases}
\end{equation}
If $\delta(\alpha,\boldsymbol\lambda,\boldsymbol{d})\geq 0$ for all $\alpha\neq 0$ then we say that $\boldsymbol\lambda$ {\it majorizes} $\boldsymbol{d}$.
\end{defn}

The following lemma can be found in \cite[Lemma 2.5]{mbjj7}.

\begin{lem}\label{dclem}  If $\mathbf c\in c_{0}^{+}(I)$, then
\[g(\alpha) = \sum_{i\colon c_{i}\geq\alpha}(c_{i}-\alpha).\]
is piecewise linear, continuous, and decreasing on $(0,\infty)$, and
\begin{equation*}\lim_{\alpha\searrow0}g(\alpha) = \sum_{i\in I}c_{i}.\end{equation*}
Moreover, if $\boldsymbol{d},\boldsymbol\lambda\in c_{0}(I)$, then $\delta(\alpha,\boldsymbol\lambda,\boldsymbol{d})$ is piecewise linear and continuous on $\R\setminus\{0\}$.
\end{lem}

The following result shows that two concepts of majorization are equivalent \cite[Proposition 2.7]{mbjj7}.

\begin{prop}\label{LR}
 Let $\boldsymbol d=(d_{i})_{i=1}^{\infty}, \boldsymbol\lambda=(\lambda_{i})_{i=1}^{\infty}$ be sequences in $c_0^+$. Then, the following two conditions are equivalent:
\begin{enumerate}
\item \(\boldsymbol d\prec \boldsymbol \lambda\),
\item $\delta(\alpha,\boldsymbol\lambda,\boldsymbol d)\geq 0$ for all $\alpha> 0$.
\end{enumerate}
In this case 
\begin{equation}\label{rtelt}\liminf_{\alpha\searrow 0} \delta(\alpha,\boldsymbol\lambda,\boldsymbol d) = \liminf_{k\to\infty}\sum_{i=1}^{k}(\lambda^\downarrow_{i}-d^\downarrow_{i}).\end{equation}
\end{prop}

In the case $\boldsymbol \lambda$ is summable we obtain a stronger version of the identity \eqref{rtelt}, see \cite[Proposition 2.8]{mbjj7}.

\begin{prop}\label{LRS} Let \(\boldsymbol d = (d_{i})_{i\in I}\) and \(\boldsymbol\lambda = (\lambda_{i})_{i\in J}\) be nonnegative sequences. If \(\boldsymbol\lambda\) is summable and \(\delta(\alpha,\boldsymbol\lambda,\boldsymbol d)\geq 0\) for all \(\alpha>0\), then \(\boldsymbol{d}\) is summable, and
\[\liminf_{\alpha\searrow 0}\delta(\alpha,\boldsymbol\lambda,\boldsymbol d) = \sum_{j\in J}\lambda_{j} - \sum_{i\in I}d_{i}.\]
\end{prop}

\subsection{Necessary conditions for one-sided compact operators} In our earlier paper \cite{mbjj7} we proved the following necessary conditions on diagonals, which only assume that either positive or negative part of the operator is compact. Thus, they will also be useful for non-compact operators.

The following result is a generalization of Schur's Theorem \cite[Theorem 3.2]{mbjj7}.

\begin{thm}\label{cptschurv2} Let \(E\) be a self-adjoint operator on a Hilbert space \(\Hil\) with compact negative part, and let \((e_{i})_{i\in I}\) be an orthonormal basis for \(\Hil\). Set \(\boldsymbol{d} = (\langle Ee_{i},e_{i}\rangle)_{i\in I}\), and let \(\boldsymbol\lambda\) be the sequence of strictly negative eigenvalues of \(E\), counted with multiplicity. Then \(\delta(\alpha,\boldsymbol\lambda,\boldsymbol{d})\geq 0\) for all \(\alpha<0\).\end{thm}

\begin{defn}\label{decouple}
Let $E$ be a self-adjoint operator on a Hilbert space $\mathcal H$. Let $(d_{i})$ be a diagonal of $E$ with respect to an orthonormal basis $(e_{i})$ of $\mathcal H$.
We say that the operator $E$ {\it decouples at $\alpha\in \R$} with respect to $(d_{i})$ if 
\[
\mathcal H_0=\overline{\lspan }\{ e_i: d_i < \alpha \}\qquad\text{and}\qquad \mathcal H_1=\overline{\lspan} \{ e_i: d_i \ge \alpha \},
\]
are invariant subspaces of $E$ and
\[
\sigma(E|_{\mathcal H_0}) \subset (-\infty,\alpha]\qquad\text{and}\qquad\sigma(E|_{\mathcal H_1}) \subset [\alpha,\infty).
\]
\end{defn}

The following result shows that the operator decouples when the excess $\displaystyle
\liminf_{\alpha\nearrow 0}\delta(\alpha,\boldsymbol\lambda,\boldsymbol d)$
is zero \cite[Proposition 3.5]{mbjj7}.

\begin{prop}\label{p211}
Let $E$ be a self-adjoint operator on $\mathcal H$ with the eigenvalue list (with multiplicity) $\boldsymbol \lambda$, which is possibly an empty list. Let $\boldsymbol d$ be a diagonal of $E$ with respect to some orthonormal basis $(e_i)_{i\in I}$. Assume that either:
\begin{itemize}
\item the positive part $E_{+}$ is compact and 
$\displaystyle
\liminf_{\alpha\searrow 0}\delta(\alpha,\boldsymbol\lambda,\boldsymbol d) =0
$,
or 
\item the negative part $E_-$ is compact and
$\displaystyle
\liminf_{\alpha\nearrow 0}\delta(\alpha,\boldsymbol\lambda,\boldsymbol d) =0$.
\end{itemize}
Then,
the operator $E$ decouples at the point $0$. 
\end{prop}

The following theorem, whose proof depends on two technical lemmas, is a generalization of the trace condition on diagonals of trace class operators to one-sided compact operators. These three results were proven in previous paper \cite[Section 3]{mbjj7} in far greater generality than was needed for compact operators. But we will make use of the generality in this paper.

\begin{lem}\label{nel}
Let $E$ be a self-adjoint operator on $\mathcal H$ such that its positive part is a compact non-trace class operator with positive eigenvalues $\lambda_1 \ge \lambda_2\geq \ldots>0$, listed with multiplicity. Let $(f_j)_{j\in\N}$ be the corresponding orthonormal sequence of eigenvectors, that is, $Ef_j = \lambda_j f_j$, $j\in \N$.
Let  $(e_{i})_{i\in \N}$ be an orthonormal sequence in $\Hil$ and let $d_{i} = \langle Ee_{i},e_{i}\rangle$, $i\in\N$. 
Then,
\begin{equation}\label{nel1}
\liminf_{M\to\infty}\sum_{i=1}^{M}(\lambda_{i} - d_{i}) 
\geq \sum_{j=1}^{\infty}
\lambda_{j}\bigg( 1-\sum_{i\in \N} |\langle e_i, f_j \rangle|^2 \bigg) - \sum_{i\in \N} \langle EP e_i,e_i \rangle,
\end{equation}
where $P$ is an orthogonal projection of $\mathcal H$ onto $(\operatorname{span} \{f_j: j\in \N \})^\perp$.
\end{lem}

\begin{lem}\label{nelt}
Let $E$ be a self-adjoint operator on $\mathcal H$ such that its positive part is a compact trace class operator with positive eigenvalues $(\lambda_{j})_{j=1}^{M}$, where \(M\in\N\cup\{0,\infty\}\). Let $(f_j)_{j=1}^{M}$ be the corresponding orthonormal sequence of eigenvectors, that is, $Ef_j = \lambda_j f_j$.
Let  $(e_{i})_{i\in I}$ be an orthonormal basis in $\Hil$ and let $d_{i} = \langle Ee_{i},e_{i}\rangle$, $i\in I$. 
Then,
\begin{equation}\label{nelt1}
\sum_{i=1}^{M}\lambda_{i} - \sum_{i:d_{i}>0}d_{i} 
=\sum_{j=1}^{M}
\lambda_{j}\bigg( 1-\sum_{i:d_{i}>0} |\langle e_i, f_j \rangle|^2 \bigg) - \sum_{i:d_{i}>0} \langle EP e_i,e_i \rangle,
\end{equation}
where $P$ is an orthogonal projection of $\mathcal H$ onto $(\overline{\lspan}(f_j)_{j=1}^{M})^\perp$.
\end{lem}

\begin{thm}\label{cpttrace} Let $E$ be a self-adjoint operator on $\mathcal H$ such that its positive part $E_{+}$ is a compact operator with the eigenvalue list (with multiplicity) $\boldsymbol \lambda$. Let $\boldsymbol d\in c_{0}$ be a diagonal of $E$ such that
\[\sum_{d_{i}<0}|d_{i}|<\infty,
\]
If the positive excess $\sigma_{+}=\liminf_{\alpha\searrow 0}\delta(\alpha,\boldsymbol\lambda,\boldsymbol{d})<\infty$, then the negative part of $E$ is trace class. Moreover, 
\begin{equation}\label{cpttrace0}
\tr(E_-) \le \sum_{d_{i}<0}|d_{i}| + \sigma_{+},
\end{equation}
with the equality when  \(\sum_{\lambda_{i}>0}\lambda_{i} < \infty\).
\end{thm}

We also need cardinality inequalities which are implied by decoupling \cite[Lemma 3.11]{mbjj7}.

\begin{lem}\label{card}
Let $E$ be a self-adjoint operator on $\mathcal H$ with the eigenvalue list (with multiplicity) $\boldsymbol \lambda$, which is possibly an empty list. Let $\boldsymbol d$ be a diagonal of $E$ with respect to some orthonormal basis $(e_i)_{i\in I}$. Assume that $E$ decouples at $0$. If the positive part $E_{+}$ is compact, then
\begin{equation}\label{card1}
\#|\{i:\lambda_{i}=0\}|\geq \#|\{i:d_{i}=0\}|
\text{\, and \,}
\#|\{i:\lambda_{i} \ge 0\}|\geq \#|\{i:d_{i} \ge0\}|.
\end{equation}
If the negative part $E_-$ is compact, then
\begin{equation}\label{card2}
\#|\{i:\lambda_{i}=0\}|\geq \#|\{i:d_{i}=0\}|
\text{\, and \,}
\#|\{i:\lambda_{i} \le 0\}|\geq \#|\{i:d_{i} \le0\}|.
\end{equation}
\end{lem}

\subsection{Diagonal-to-diagonal results} The following diagonal-two-diagonal result was shown by Siudeja and the authors \cite[Theorem 3.6]{unbound}.

\begin{prop}\label{posSH} Let $\boldsymbol\lambda=(\lambda_{i})_{i=1}^{\infty}$ and $\boldsymbol d = (d_{i})_{i=1}^{\infty}$ be nonincreasing sequences and define
\[\delta_{n} = \sum_{i=1}^{n}(\lambda_{i}-d_{i}).\]
If $\boldsymbol\lambda$ is a diagonal of a self-adjoint operator $E$, $\delta_{n}\geq 0$ for all $n\in\N$, and
\[\liminf_{n\to\infty}\delta_{n} = 0,\]
then $\boldsymbol d$ is also a diagonal of $E$.
\end{prop}

The following result is a diagonal-to-diagonal generalization of the Kaftal-Weiss theorem \cite[Corollary 6.1]{kw}, which was shown in \cite[Propostion 8.3]{mbjj7}.

\begin{prop}\label{kwd2d}
Let $\boldsymbol\lambda=(\lambda_{i})_{i=1}^{\infty}$ and $\boldsymbol d = (d_{i})_{i=1}^{\infty}$ be positive sequences in \(c_{0}\) such that \(\boldsymbol d\prec\boldsymbol\lambda\) and 
\[
\sum_{i=1}^\infty \lambda_i = \sum_{i=1}^\infty d_i.
\]
If $\boldsymbol\lambda$ is a diagonal of a self-adjoint operator $E$,
then $\boldsymbol d$ is also a diagonal of $E$.
\end{prop}

The following elementary lemma \cite[Lemma 4.5]{mbjj7} shows that diagonal-to-diagonal results extend from diagonal subsequences.

\begin{lem}\label{subdiag} Let $(\lambda_{i})_{i\in I}$ and $(d_{i})_{i\in J}$ be two sequences of real numbers. Suppose that: 
\begin{enumerate}
\item there is a set $K$ and partitions $(I_{k})_{k\in K}$ and $(J_{k})_{k\in K}$ of $I$ and $J$, respectively,
\item for every $k\in K$, if $E_k$ is any self-adjoint operator with diagonal $(\lambda_{i})_{i\in I_k}$, then $(d_{i})_{i\in J_k}$ is also a diagonal of $E_k$,
\end{enumerate}
Then, if $E$ is any self-adjoint operator with diagonal $(\lambda_{i})_{i\in I}$, then $(d_{i})_{i\in J}$ is also a diagonal of $E$.
\end{lem}

The following result on diagonals of compact operators was shown by the authors \cite[Theorems 10.1 and 10.2]{mbjj7}.

\begin{thm}\label{intropv3}
 Let $\boldsymbol\lambda , \boldsymbol d \in c_0$. Define
\[\sigma_{+} = \liminf_{n\to\infty} \sum_{i=1}^n (\lambda_i^{+ \downarrow} - d_i^{+ \downarrow})\quad\text{and}\quad \sigma_{-} = \liminf_{n\to\infty} \sum_{i=1}^n (\lambda_i^{- \downarrow} - d_i^{- \downarrow}).\]
\noindent {\rm (Necessity)} Let $E$ be a compact operator  with the eigenvalue list $\boldsymbol\lambda$.
If $\boldsymbol d$ is a diagonal of $E$, then
\begin{equation}\label{p5v3}
\sigma_{+}=0 \implies \#|\{i:\lambda_{i}=0\}|\geq \#|\{i:d_{i}=0\}|
\text{\, and \,}
\#|\{i:\lambda_{i} \ge 0\}|\geq \#|\{i:d_{i} \ge0\}|,
\end{equation}
\begin{equation}\label{p6v3}
\sigma_{-}=0 \implies \#|\{i:\lambda_{i}=0\}|\geq \#|\{i:d_{i}=0\}|
\text{\, and \,}
\#|\{i:\lambda_{i} \le 0\}|\geq \#|\{i:d_{i} \le0\}|,
\end{equation}
\begin{equation}\label{p1v3}
\sum_{i=1}^n \lambda_i^{+ \downarrow} \ge 
\sum_{i=1}^n d_i^{+ \downarrow}
\qquad\text{for all }n\in \N,
\end{equation}
\begin{equation}\label{p2v3}
\sum_{i=1}^n \lambda_i^{- \downarrow} \ge 
\sum_{i=1}^n d_i^{- \downarrow}
\qquad\text{for all }n\in \N,
\end{equation}
\begin{equation}\label{p3v3}
\boldsymbol d_+ \in \ell^1
\quad \implies\quad
\sigma_{-}\geq \sigma_{+},
\end{equation}
\begin{equation}\label{p4v3}
\boldsymbol d_- \in \ell^1 
\quad \implies\quad 
\sigma_{+}\geq \sigma_{-}.
\end{equation}

\noindent {\rm (Sufficiency)} Conversely, suppose that \eqref{p5v3}--\eqref{p4v3} hold and
\begin{equation}\label{p0v3}\sigma_{+}+\sigma_{-}>0.\end{equation}
If a self-adjoint operator $E$ has diagonal $\boldsymbol\lambda$, 
then $\boldsymbol d$ is also a diagonal of $E$.
\end{thm}

Using Proposition \ref{LR}  we have the following special case of the sufficiency part of Theorem \ref{intropv3} 
when excesses are equal $\sigma_-=\sigma_+>0$.


\begin{thm}\label{step5} Let $\boldsymbol\lambda, \boldsymbol d  \in c_{0}$.  If there is a self-adjoint operator $E$ with diagonal $\boldsymbol\lambda$, 
\begin{equation}\label{pos5.0}\delta(\alpha,\boldsymbol\lambda, \boldsymbol d)\geq 0\quad\text{for all }\alpha\neq 0,\end{equation}
and
\begin{equation}\label{pos5}
\liminf_{\alpha\searrow 0}\delta(\alpha,\boldsymbol\lambda, \boldsymbol d) = \liminf_{\alpha\nearrow 0}\delta(\alpha,\boldsymbol\lambda, \boldsymbol d)>0,\end{equation}
then $\boldsymbol d$ is also a diagonal of $E$.
\end{thm}

\subsection{Diagonals of operators}

Loreaux and Weiss \cite{lw} gave necessary conditions and sufficient conditions on $\mathcal D(T)$ for compact positive operators $T$ with nontrivial kernel. When $\ker(T)$ is infinite dimensional, the necessary and sufficient conditions coincide yielding a complete characterization. However, when $\ker T$ is nontrivial and finite dimensional, the full characterization of $\mathcal D(T)$ remains elusive. This is known as the kernel problem.
The following theorem is a convenient reformulation of their two results. The sufficiency part of Theorem \ref{kp} is \cite[Theorem 2.4]{lw}, whereas the necessity is \cite[Theorem 3.4]{lw}.

\begin{thm}\label{kp}
Let $E$ be a compact operator on $\mathcal H$ with the eigenvalue list $\boldsymbol\lambda \in c^+_0$. 
Let 
$\boldsymbol d \in c^+_0$. If  $\#|\{i:d_{i}=0\}|<\infty$, then we set
$z=   \#|\{ i : \lambda_i =0 \}| - \#|\{i:d_{i}=0\}|$; otherwise set $z=0$.

\noindent {\rm (Necessity)} If $\boldsymbol d$ is a diagonal of $E$, then
\begin{align}\label{kp1}
 \#|\{i:\lambda_{i}=0\}| &\geq \#|\{i:d_{i}=0\}|,
\\
\label{kp2}
\sum_{i=1}^n \lambda_i^{+ \downarrow} &\ge 
\sum_{i=1}^n d_i^{+ \downarrow}
\qquad\text{for all }n\in \N,
\\
\label{kp3}
\sum_{i=1}^\infty \lambda_i &= 
\sum_{i=1}^\infty d_i,
\end{align}
and for any $p\in \N$, $p\le z$, and for every $\epsilon>0$, there exists $N=N_{p,\epsilon}>0$, such that
\begin{equation}\label{kp4a}
\sum_{i=1}^n \lambda_i^{+ \downarrow} +  \epsilon \lambda_{n+1}^{+ \downarrow}
\ge 
\sum_{i=1}^{n+p} d_i^{+ \downarrow} 
\qquad\text{for all }n\ge N.
\end{equation}

\noindent {\rm (Sufficiency)} Conversely, if \eqref{kp1}--\eqref{kp3} hold and
for any $p\in \N$, $p\le z$, there exists $N=N_p$, such that
\begin{equation}\label{kp4b}
\sum_{i=1}^n \lambda_i^{+ \downarrow} 
\ge 
\sum_{i=1}^{n+p} d_i^{+ \downarrow} 
\qquad\text{for all }n\ge N,
\end{equation}
then $\boldsymbol d$ is a diagonal of $E$. 
\end{thm}

To describe our algorithms for determining diagonals of self-adjoint operators we shall use the concept of splitting of an operator.

\begin{defn}\label{splitting} Let $E$ be a self-adjoint operator. Let $\alpha \in \R$. We say that a pair of operators $E_1$ and $E_2$ is a {\it splitting} of $E$  at $\alpha$ if there exist $z_0,z_1,z_2 \in \N \cup \{0,\infty\}$ such that:
\begin{enumerate}
\item $E_1$ is a self-adjoint operator such that $\sigma(E_1) \subset (-\infty,\alpha]$ and $z_1=\dim \ker (E_1-\alpha \bf I)$,
\item $E_2$ is a self-adjoint operator such that $\sigma(E_2) \subset [\alpha,\infty)$ and $z_2=\dim \ker (E_2-\alpha \bf I)$,
 and
\item $E$ is unitarily equivalent to $\alpha \mathbf I_{z_0} \oplus E_1 \oplus E_2$, and hence $z_0+z_1+z_2= \dim \ker (E-\alpha \bf I)$.
\end{enumerate}
\end{defn}

The following elementary fact \cite[Theorem 11.4]{mbjj7} bridges the concepts of decoupling and splitting.

\begin{thm}\label{isplit}
Let $E$ be a self-adjoint operator.
Let $\boldsymbol d$ be a bounded sequence. 
Suppose that $\boldsymbol d$ is a diagonal of $E$ and the operator $E$ decouples at $\alpha \in \R$. Then, there exists a splitting $E_1$ and $E_2$ of $E$ at $\alpha$ such that:
\begin{enumerate}[(i)]
\item the sequence $(d_{i})_{d_i<\alpha}$ is a diagonal of $E_1$,
\item the sequence $(d_{i})_{d_i>\alpha}$ is a diagonal of $E_2$, and
\item the number of $\alpha$'s in $\boldsymbol d$ satisfies
\begin{equation}\label{z102}
\dim \ker (\alpha \mathbf I - E) =  \dim \ker (\alpha \mathbf I -E_1) + \dim \ker (\alpha \mathbf I - E_2) + \#|\{i: d_i=\alpha\}|.
\end{equation}
\end{enumerate}
Conversely, if there exists a splitting of $E$ such that (i)--(iii) hold, then $\boldsymbol d$ is a diagonal of $E$ and the operator $E$ decouples at $\alpha$.
\end{thm}

Finally, we will also need the characterization of diagonals of compact self-adjoint operators modulo the dimension of the kernel \cite[Theorem 1.3]{mbjj7}.

\begin{thm}\label{intropv2}
Let $\boldsymbol\lambda, \boldsymbol d \in c_0$. Set
\[\sigma_{+} = \liminf_{n\to\infty} \sum_{i=1}^n (\lambda_i^{+ \downarrow} - d_i^{+ \downarrow})\quad\text{and}\quad \sigma_{-} = \liminf_{n\to\infty} \sum_{i=1}^n (\lambda_i^{- \downarrow} - d_i^{- \downarrow})\]
Let $T$ be a compact self-adjoint operator with eigenvalue list $\boldsymbol\lambda$.
Then the following are equivalent:
\begin{enumerate}
\item
$\boldsymbol d\in \mathcal D(T')$ for some operator $T'$ in the operator norm closure of unitary orbit of $T$,
\[ T' \in \overline
{\{ \text{$UTU^*$}: U \text{ is unitary}\}}^{||\cdot||},
\]
\item $\boldsymbol d$ is a diagonal of an operator $T'$ such that $T\oplus \boldsymbol 0$ and $T' \oplus \boldsymbol 0$
are unitarily equivalent, where $\boldsymbol 0$ denotes the zero operator on an infinite-dimensional Hilbert space,
\item
$\boldsymbol\lambda$ and $\boldsymbol d$ satisfy the following four conditions 
\begin{equation}\label{p1v2}
\sum_{i=1}^n \lambda_i^{+ \downarrow} \ge 
\sum_{i=1}^n d_i^{+ \downarrow}
\qquad\text{for all }k\in \N,
\end{equation}
\begin{equation}\label{p2v2}
\sum_{i=1}^n \lambda_i^{- \downarrow} \ge 
\sum_{i=1}^n d_i^{- \downarrow}
\qquad\text{for all }k\in \N,
\end{equation}
\begin{equation}\label{p3v2}
\boldsymbol d_+ \in \ell^1
\quad \implies\quad
\sigma_{-}\geq \sigma_{+}
\end{equation}
\begin{equation}\label{p4v2}
\boldsymbol d_- \in \ell^1 
\quad \implies\quad 
\sigma_{+}\geq \sigma_{-}
\end{equation}
\end{enumerate}
\end{thm}

\section{The majorization function $f_{\boldsymbol{d}}$}

In this section we prove several useful properties of the majorization function $f_{\boldsymbol{d}}$ given by Definition \ref{d81}. 
The following proposition is a convenient reformulation of \cite[Lemma 3.9]{mbjj3}. The formula \eqref{f1} follows directly from Definition \ref{d81}.

\begin{prop}\label{f} Let $\boldsymbol{d}=(d_{i})_{i\in I}$ be a sequence in $[0,1]$ with $f_{\boldsymbol{d}}(\alpha)<\infty$ for all $\alpha\in(0,1)$. The function $f_{\boldsymbol{d}}$ is piecewise linear, continuous, concave, satisfies
\begin{equation}\label{f0}\lim_{\alpha\to 0^{+}}f_{\boldsymbol{d}}(\alpha) = \lim_{\alpha\to 1^{-}}f_{\boldsymbol{d}}(\alpha)=0,\end{equation}
and
\[f_{\boldsymbol{d}}'(\alpha)=D_{\boldsymbol{d}}(\alpha)-C_{\boldsymbol{d}}(\alpha)\quad\text{ for every }\alpha\in (0,1)\setminus\{d_{i}\colon i\in I\}.\]
Moreover, for \(1>\alpha\geq\beta>0\)
\begin{equation}\label{f1}
C_{\boldsymbol{d}}(\alpha) - D_{\boldsymbol{d}}(\alpha) = C_{\boldsymbol{d}}(\beta) - D_{\boldsymbol{d}}(\beta) + \#|\{i : d_{i}\in[\beta,\alpha)\}|.
\end{equation}
In particular, there exists a number $\eta\in[0,1)$ such that 
\[D_{\boldsymbol{d}}(\alpha)-C_{\boldsymbol{d}}(\alpha)\equiv \eta\mod 1\quad\text{for all }\alpha\in(0,1).\]
\end{prop}

Kadison's theorem \cite{k1,k2} takes the following form using the majorization function $f_{\boldsymbol{d}}$.

\begin{thm}[Kadison]
\label{Kadison} Let $(d_{i})_{i\in I}$ be a sequence in $[0,1]$ and $\alpha\in(0,1)$. 
There exists an orthogonal projection $P$ on $\ell^2(I)$ with diagonal $(d_{i})_{i\in I}$ if and only if either:
\begin{enumerate}
\item $f_{\boldsymbol{d}}(\alpha)=\infty$ for some (and hence all) $\alpha\in(0,1)$,
\item $f_{\boldsymbol{d}}(\alpha)<\infty$ for some (and hence all) $\alpha\in(0,1)$, and
\begin{equation}
\label{kadcond} 
f_{\boldsymbol{d}}'(\alpha)\in\Z
\qquad\text{for some (and hence for almost every) }\alpha\in(0,1).
\end{equation}
\end{enumerate}
\end{thm}

\begin{prop}\label{9.2} Let $\boldsymbol{d}=(d_{i})_{i\in I}$ and $\boldsymbol\lambda=(\lambda_{i})_{i\in J}$ be sequences in $[0,1]$ with $f_{\boldsymbol{d}}(\alpha)<\infty$ and $f_{\boldsymbol\lambda}(\alpha)<\infty$ for all $\alpha\in(0,1)$. The following are equivalent:
\begin{enumerate}
\item $f_{\boldsymbol{d}}(\alpha)\geq f_{\boldsymbol\lambda}(\alpha)$ for all $\alpha\in (0,1)$
\item $f_{\boldsymbol{d}}(\lambda_{i})\geq f_{\boldsymbol\lambda}(\lambda_{i})$ for all $i\in J\setminus\{j\in J\colon \lambda_{j}=0\text{ or }1\}$.
\end{enumerate}
\end{prop}

\begin{proof} (i)\(\Rightarrow\) (ii) is obvious. For the other direction we simply note that \(f_{\boldsymbol{d}}\) is concave, and \(f_{\boldsymbol\lambda}\) is piecewise linear, with knots at each \(\lambda_{i}\in(0,1)\). 
\end{proof}

\begin{prop}\label{fd} Let $\boldsymbol{d}=(d_{i})_{i\in I}$ be a sequence in $[0,1]$ with $f_{\boldsymbol{d}}(\alpha)<\infty$. Then, for any set $J\subset I$ and $\alpha \in (0,1)$ we have
\begin{equation}\label{fd1}
f_{\boldsymbol{d}}(\alpha) \le  (1-\alpha)\sum_{i\in J}d_i + \alpha\sum_{i\in I \setminus J }( 1 - d_i).
\end{equation}
\end{prop}

\begin{proof}
We can assume that the right hand side \eqref{fd1} is finite. Suppose that there exists $i_0 \in J$ such that $d_{i_0} > \alpha$. Then the right hand side \eqref{fd1} corresponding to the set $J \setminus \{i_0\}$ decreases by $d_{i_0}-\alpha$. Likewise, if there exists $i_0 \in I \setminus J$ such that $d_{i_0} < \alpha$, then the right hand side \eqref{fd1} corresponding to the set $J \cup \{i_0\}$ decreases by $\alpha - d_{i_0}$. This implies that the right hand side \eqref{fd1} is minimized if and only if
\[
\{i\in I: d_i < \alpha \} \subset J \subset \{i\in I: d_i \le \alpha \}.
\]
In that case \eqref{fd1} becomes an equality.
\end{proof}

\begin{lem}\label{intToComp} Suppose \(\boldsymbol\lambda = (\lambda_{i})_{i\in\N}\) and \(\boldsymbol{d} = (d_{i})_{i\in\N}\) are sequences such that \(f_{\boldsymbol\lambda}(\alpha)+f_{\boldsymbol{d}}(\alpha)<\infty\) for all \(\alpha\in(0,1)\) and \eqref{mainthm2} holds. Fix \(\alpha_{0}\in(0,1)\) such that \(\alpha_{0}\) is not a term of \(\boldsymbol\lambda\) or \(\boldsymbol{d}\). Define \(\boldsymbol\lambda_{0} := (\lambda_{i})_{0\leq \lambda_{i}<\alpha_{0}}\), \(\boldsymbol{d}_{0} := (d_{i})_{0\leq d_{i}<\alpha_{0}}\),
\[\kappa:=f_{\boldsymbol{d}}'(\alpha_{0}) - f_{\boldsymbol\lambda}'(\alpha_0)-\delta_{0}+\delta_{1},\]
and
\[\beta:=f_{\boldsymbol{d}}(\alpha_{0})-f_{\boldsymbol\lambda}(\alpha_{0})+(1-\alpha_{0})\delta_{0} + \alpha_{0}\delta_{1}-\alpha_{0}\kappa.\]
Then,
\[\delta(\alpha,\boldsymbol\lambda_{0},\boldsymbol d_{0})\geq -\kappa\alpha-\beta\quad\text{for all }\alpha\neq 0,\]
and
\[\liminf_{\alpha\searrow0}\delta(\alpha,\boldsymbol\lambda_{0},\boldsymbol d_{0}) = \delta_{0}-\beta.\]
\end{lem}

\begin{proof} For \(\alpha\leq\alpha_{0}\), by \eqref{f1} we have
\begin{align*}
f_{\boldsymbol\lambda}(\alpha) & = C_{\boldsymbol\lambda}(\alpha)+\alpha (D_{\boldsymbol\lambda}(\alpha)-C_{\boldsymbol\lambda}(\alpha))\\
 & = C_{\boldsymbol\lambda}(\alpha)+\alpha(D_{\boldsymbol\lambda}(\alpha_{0}) - C_{\boldsymbol\lambda}(\alpha_{0}) - \#|\{i : \lambda_{i}\in[\alpha,\alpha_{0})\}|)\\
 & = C_{\boldsymbol\lambda}(\alpha)+\alpha f_{\boldsymbol\lambda}'(\alpha_{0}) - \alpha\#|\{i : \lambda_{i}\in[\alpha,\alpha_{0})\}|.
\end{align*}
The same reasoning for the sequence \(\boldsymbol{d}\) yields
\[f_{\boldsymbol{d}}(\alpha) = C_{\boldsymbol{d}}(\alpha)+\alpha f_{\boldsymbol{d}}'(\alpha_{0}) - \alpha\#|\{i : d_{i}\in[\alpha,\alpha_{0})\}|.\]
Now, the inequality \eqref{mainthm2} becomes
\begin{align*}
C_{\boldsymbol\lambda}(\alpha) &-C_{\boldsymbol{d}}(\alpha) 
\\
& \leq \alpha\big(f_{\boldsymbol{d}}'(\alpha_{0})-f_{\boldsymbol\lambda}'(\alpha_{0})+\#|\{i : d_{i}\in[\alpha,\alpha_{0})\}|-\#|\{i : \lambda_{i}\in[\alpha,\alpha_{0})\}|+\delta_{1}-\delta_{0}\big) + \delta_{0}\\
 & = \alpha\big(\#|\{i : d_{i}\in[\alpha,\alpha_{0})\}|-\#|\{i : \lambda_{i}\in[\alpha,\alpha_{0})\}| + \kappa\big) + \delta_{0}\\
 & = \sum_{d_{i}\in[\alpha,\alpha_{0})}\alpha - \sum_{\lambda_{i}\in[\alpha,\alpha_{0})}\alpha + \alpha\kappa+\delta_{0},
\end{align*}
and
\begin{equation}\label{dd}
\begin{split}
\sum_{0\leq\lambda_{i}<\alpha_{0}}\lambda_{i} - \sum_{0\leq d_{i}<\alpha_{0}}d_{i} & = C_{\boldsymbol\lambda}(\alpha_{0})-C_{\boldsymbol{d}}(\alpha_{0})\\
 & = f_{\boldsymbol\lambda}(\alpha_{0}) - f_{\boldsymbol{d}}(\alpha_{0}) + \alpha_{0}\big(f_{\boldsymbol{d}}'(\alpha_{0})-f_{\boldsymbol\lambda}'(\alpha_{0})\big)\\
 & = (1-\alpha_{0})\delta_{0} + \alpha_{0}\delta_{1}-\alpha_{0}\kappa-\beta +\alpha_{0}(\kappa+\delta_{0}-\delta_{1}) = \delta_{0}-\beta.
\end{split}\end{equation}
Putting these together we have
\[\delta_{0}-\beta \leq \sum_{\lambda_{i}\in[\alpha,\alpha_{0})}(\lambda_{i}-\alpha)-\sum_{d_{i}\in[\alpha,\alpha_{0})}(d_{i}-\alpha)+ \alpha\kappa + \delta_{0}\]
Thus,
\[-\kappa\alpha-\beta\leq \sum_{\lambda_{i}\in[\alpha,\alpha_{0})}(\lambda_{i}-\alpha)-\sum_{d_{i}\in[\alpha,\alpha_{0})}(d_{i}-\alpha) = \delta(\alpha,\boldsymbol\lambda_{0},\boldsymbol d_{0}).\]
Where \(\boldsymbol\lambda_{0} = (\lambda_{i})_{\lambda_{i}\in[0,\alpha_{0})}\) and \(\boldsymbol{d}_{0} = (d_{i})_{d_{i}\in[0,\alpha_{0})}\). By \eqref{dd} and Lemma \ref{dclem} we see that
\[\liminf_{\alpha\searrow0}\delta(\alpha,\boldsymbol\lambda_{0},\boldsymbol d_{0}) = \delta_{0}-\beta.\qedhere\]
\end{proof}

\begin{cor}\label{intToComp1} Suppose \(\boldsymbol\lambda = (\lambda_{i})_{i\in\N}\) and \(\boldsymbol{d} = (d_{i})_{i\in\N}\) are sequences such that \(f_{\boldsymbol\lambda}(\alpha)+f_{\boldsymbol{d}}(\alpha)<\infty\) for all \(\alpha\in(0,1)\). If \eqref{mainthm2} holds with equality at \(\alpha=\alpha_{0}\in(0,1)\), then \(\boldsymbol\lambda_{0} := (\lambda_{i})_{0\leq \lambda_{i}<\alpha_{0}}\) majorizes \(\boldsymbol{d}_{0} := (d_{i})_{0\leq d_{i}<\alpha_{0}}\) in the sense of Definition \ref{deltadef}, that is,
\[\delta(\alpha,\boldsymbol\lambda_{0},\boldsymbol d_{0})\geq 0\quad\text{for all }\alpha\neq 0.\]
Moreover,
\[\liminf_{\alpha\searrow0}\delta(\alpha,\boldsymbol\lambda_{0},\boldsymbol d_{0}) = \delta_{0}.\]
\end{cor}

\begin{proof} There is a nonempty interval \(I_{0}=(\alpha_{0}-\eps,\alpha_{0})\) which contains no terms of \(\boldsymbol{\lambda}\) or \(\boldsymbol{d}\). Note that \(C_{\boldsymbol\lambda},D_{\boldsymbol\lambda}\), \(C_{\boldsymbol{d}}\), and \(D_{\boldsymbol{d}}\) are all constant on \(I_{0}\), and hence we have equality in \eqref{mainthm2} for all \(\alpha\in I_{0}\). We can decrease \(\alpha_{0}\) to guarantee it is not equal to any term of \(\boldsymbol{\lambda}\) or \(\boldsymbol{d}\), and without changing the definitions of \(\boldsymbol{\lambda}_{0}\) and \(\boldsymbol{d}_{0}\).

Define \(\kappa\) and \(\beta\) as in Lemma \ref{intToComp}. Note that \eqref{mainthm2} is equivalent to
\[f_{\boldsymbol{d}}(\alpha)-f_{\boldsymbol\lambda}(\alpha)+(1-\alpha)\delta_{0} + \alpha\delta_{1}\geq 0\quad\text{for all }\alpha\in(0,1).\]
Equality in \eqref{mainthm2}  at \(\alpha=\alpha_{0}\) is equivalent to \(\beta=-\alpha_{0}\kappa\).
 Moreover, we see that the left-hand side is differentiable at \(\alpha_{0}\). We deduce that the derivative of the left-hand size is \(0\) at \(\alpha=\alpha_{0}\), that is, \(\kappa=0\). Hence, we also have \(\beta=-\alpha_{0}\kappa=0\). Hence applying Lemma \ref{intToComp} to the sequences \(\boldsymbol{\lambda}\) and \(\boldsymbol{d}\) gives the desired conclusions.
\end{proof}

\begin{cor}\label{intToComp2} Suppose \(\boldsymbol\lambda = (\lambda_{i})_{i\in\N}\) and \(\boldsymbol{d} = (d_{i})_{i\in\N}\) are sequences such that \(f_{\boldsymbol\lambda}(\alpha)+f_{\boldsymbol{d}}(\alpha)<\infty\) for all \(\alpha\in(0,1)\). If \eqref{mainthm2} holds with equality at \(\alpha=\alpha_{0}\in(0,1)\), then \(\boldsymbol\lambda_{1} := (1-\lambda_{i})_{\alpha_0 \le \lambda_{i} \le 1}\) majorizes \(\boldsymbol{d}_{1} := (1-d_{i})_{\alpha_0 \leq d_{i}\le 1}\) in the sense of Definition \ref{deltadef}, that is,
\[\delta(\alpha,\boldsymbol\lambda_{1},\boldsymbol d_{1})\geq 0\quad\text{for all }\alpha\neq 0.\]
Moreover,
\[\liminf_{\alpha\searrow0}\delta(\alpha,\boldsymbol\lambda_{1},\boldsymbol d_{1}) = \delta_{1}.\]
\end{cor}

\begin{proof} As in the proof of Corollary \ref{intToComp1} we may assume \(\alpha_{0}\) is in the interior of and interval \(I_{0}\) such that \eqref{mainthm2} holds for all \(\alpha\in I_{0}\) and both \(d_{i},\lambda_{i}\notin I_{0}\) for all \(i\in I\). This implies that \(\boldsymbol d_{1} = (1-d_{i})_{\alpha_{0}<d_{i}\leq 1}\) and \(\boldsymbol\lambda_{1} = (1-\lambda_{i})_{\alpha_{0}<\lambda_{i}\leq 1}\).

Set \(\tilde{\boldsymbol\lambda} = (1-\lambda_{i})_{i\in I}\) and \(\tilde{\boldsymbol{d}} = (1-d_{i})_{i\in I}\). Note that \eqref{mainthm2} holds for these sequences with \(\delta_{0}\) and \(\delta_{1}\) swapped. Hence, the conclusions hold by applying Corollary \ref{intToComp1} to \(\tilde{\boldsymbol\lambda}\) and \(\tilde{\boldsymbol{d}}\).
\end{proof}

\section{Necessity proof of Theorem \ref{mainthm}}\label{necProof}

The goal of this section is to prove the necessary conditions on diagonals of operators with $\ge 2$ points in their essential spectrum. The bulk of necessity direction of Theorem \ref{mainthm} follows from the following theorem. In essence, Theorem \ref{nec} says that if a diagonal sequence $\boldsymbol d$ satisfies the Blaschke condition and the lower and upper excesses are finite, then the operator $E$ is necessarily diagonalizable with eigenvalues satisfying the Blaschke condition. A similar result, albeit with the stronger assumption on $E$ that $E_-$ and $(E-\mathbf I)_+$ are trace class, was shown was Loreaux and Weiss \cite[Corollary 4.5]{lw3}.

\begin{thm}\label{nec}
Let $E$ be a self-adjoint operator on $\mathcal H$ such that $0$ and $1$ are two extreme points of the essential spectrum $\sigma_{ess}(E)$ of $E$, that is,
\begin{equation*}
\{0,1 \} \subset \sigma_{ess}(E) \subset [0,1].
\end{equation*}
Let $\boldsymbol\lambda=(\lambda_{j})_{j\in J}$ be the list of all eigenvalues of $E$ with multiplicity (which is possibly an empty list). 
Let  $\boldsymbol{d}=(d_{i})_{i\in I}$ be a diagonal of $E$, which is given by $d_{i} = \langle Ee_{i},e_{i}\rangle$, with respect to some orthonormal basis $(e_{i})_{i\in I}$ of $\Hil$. Define
\[\delta^{L} = \liminf_{\beta\nearrow 0}(f_{\boldsymbol\lambda}(\beta)-f_{\boldsymbol{d}}(\beta))\quad\text{and}\quad \delta^{U} = \liminf_{\beta\searrow 1}(f_{\boldsymbol\lambda}(\beta)-f_{\boldsymbol{d}}(\beta)).\]

Suppose that $f_{\boldsymbol{d}}(\alpha)<\infty$ for some $\alpha\in(0,1)$ and
 $\delta^L+\delta^U <\infty$. Then the operator $E$ is diagonalizable and $\boldsymbol d$ has accumulation points at $0$ and $1$.
Moreover, there exist $\delta_0 \in [0,\delta^L]$ and $\delta_1 \in [0,\delta^U]$ such that such that 
the following five conclusions hold: 
\begin{enumerate}
\item
$f_{\boldsymbol\lambda}(\alpha)<\infty$ for all $\alpha\in(0,1)$,
\item
$f_{\boldsymbol\lambda}(\alpha)  \leq f_{\boldsymbol{d}}(\alpha) + (1-\alpha)\delta_{0} + \alpha\delta_{1}$ for all $\alpha\in(0,1)$, 
\item
$f_{\boldsymbol\lambda}'(\alpha)-f_{\boldsymbol{d}}'(\alpha)   \equiv \delta_1-\delta_0 \mod 1 $ for some $\alpha\in(0,1)$,
\item
 if $\sum_{\lambda_j<0} |\lambda_j|<\infty$, then $\delta^L-\delta_0 \le \delta^U- \delta_1$, 
\item
if $\sum_{\lambda_j>1} (\lambda_j-1)<\infty$, then $\delta^U- \delta_1 \le \delta^L-\delta_0$.
\end{enumerate}

In addition, suppose that there exists $\alpha \in [0,1]$ such that $(\mathfrak d_\alpha)$ happens, where
\begin{enumerate}
\item[$(\mathfrak d_0)$] $\delta^L=0$ \quad or  \quad $\delta_0=0$ and $\delta_1=\delta^U$,
\item[$(\mathfrak d_\alpha)$] $\alpha \in (0,1)$ and $f_{\boldsymbol\lambda}(\alpha)  = f_{\boldsymbol{d}}(\alpha) + (1-\alpha)\delta_{0} + \alpha\delta_{1}$,
\item[$(\mathfrak d_1)$] $\delta^U=0$ \quad or  \quad  $\delta_1=0$ and $\delta_0=\delta^L$.
\end{enumerate}
Then, the operator $E$ decouples at the point $\alpha\in [0,1]$. That is,
\[
\mathcal H_0=\overline{\lspan }\{ e_i: d_i < \alpha \}\qquad\text{and}\qquad \mathcal H_1=\overline{\lspan} \{ e_i: d_i \ge \alpha \},
\]
are invariant subspaces of $E$ and
\[
\sigma(E|_{\mathcal H_0}) \subset (-\infty,\alpha]\qquad\text{and}\qquad\sigma(E|_{\mathcal H_1}) \subset [\alpha,\infty).
\]
\end{thm}

The following lemma will play a key role in the proof of Theorem \ref{nec}.

\begin{lem}\label{trap} Let $E$ be a positive operator on a Hilbert space $\Hil$ with $\sigma(E)\subset [0,1]$. Let $(e_{i})_{i\in I}$ be a Parseval frame for $\Hil$.
Define the sequence $\boldsymbol{d}=(d_{i})_{i\in I}$ by $d_{i} = \langle Ee_{i},e_{i}\rangle$. If there exists a subset $J \subset I$ such that
\begin{equation}\label{trap0}
\sum_{i \in J} d_i <\infty \qquad\text{and}\qquad   \sum_{i \in I \setminus J } (||e_i||^2-d_{i} )<\infty,\end{equation}
 then $E$ is diagonalizable. Moreover, if $\boldsymbol\lambda=(\lambda_{i})_{i\in I'}$ is the eigenvalue list (with multiplicity) of $E$, then $f_{\boldsymbol\lambda}(\alpha)<\infty$ for all $\alpha\in(0,1)$,
\begin{equation}\label{trap1}
f_{\boldsymbol{\lambda}}(\alpha)
\le 
 (1-\alpha)\sum_{i\in J}d_i + \alpha\sum_{i\in I \setminus J }( ||e_i||^2 - d_i)
\qquad\text{for all }\alpha\in(0,1),\end{equation}
In addition, if 
\begin{equation}\label{trapa}
 \sum_{i \in I  } \min(||e_i||^2,1-||e_{i}||^2)< \infty,
\end{equation}
then
\begin{equation}\label{trap2}f_{\boldsymbol{d}}'(\alpha)-f_{\boldsymbol\lambda}'(\alpha)\in\Z
\qquad\text{for a.e. }\alpha\in(0,1).
\end{equation}
\end{lem}


\begin{proof} By the spectral theorem there is a projection valued measure $\pi$ such that 
\[
E = \int_{[0,1]}\lambda\,d\pi(\lambda).
\]
Fix $\alpha\in(0,1)$ and set
\[
K = \int_{[0,\alpha)}\lambda\,d\pi(\lambda)\quad\text{and}\quad 
T = \int_{[\alpha,1]}(1-\lambda)\,d\pi(\lambda).\]
Let $P$ be the projection given by $P =\pi([\alpha, 1])$. For each $i\in I$ set $k_{i} = \langle Ke_{i},e_{i}\rangle, p_{i} = \langle Pe_{i},e_{i}\rangle$, and $t_{i} = \langle Te_{i},e_{i}\rangle$. Since $E=K+P-T$ we also have 
\begin{equation}\label{2trace0}d_{i} = k_{i}+p_{i}-t_{i}.\end{equation}

We claim that $K$ and $T$ are trace class. Since $K$ and $T$ are positive and $(e_{i})_{i\in I}$ is a Parseval frame, it is enough to show that $\sum_{i\in I}(k_{i}+t_{i})<\infty$. 

Next, we observe
\begin{equation}\label{2trace1}T=\int_{[\alpha,1]}(1-\lambda)\,d\pi(\lambda)\leq \int_{[\alpha,1]}(1-\alpha)\,d\pi(\lambda)= (1-\alpha)P\end{equation}
and
\begin{equation}\label{2trace2}K = \int_{[0,\alpha)}\lambda\,d\pi(\lambda)\leq \int_{[0,\alpha)}\alpha\,d\pi(\lambda)= \alpha(\mathbf{I}-P).\end{equation}

Combing $E=K+P-T$  and \eqref{2trace1} yields $E \ge K+ \frac{\alpha}{1-\alpha}T$. By \eqref{2trace0} we deduce that  $d_{i}\geq k_{i}+\frac{\alpha}{1-\alpha}t_{i}$, and thus
\begin{equation}\label{2trace3}\sum_{i\in J}\left(k_{i}+\frac{\alpha}{1-\alpha}t_{i}\right)\leq \sum_{i\in J}d_{i} <\infty.\end{equation}
From \eqref{2trace2} we deduce that $\mathbf I - E = T-K +(\mathbf I - P)  \ge T +  \frac{1-\alpha}{\alpha} K$. Thus, 
$||e_i||^2-d_{i} \geq \frac{1-\alpha}{\alpha}k_{i}+t_{i}$. By \eqref{trap0} this yields
\begin{equation}\label{2trace4}\sum_{i\in I \setminus J }\left(\frac{1-\alpha}{\alpha}k_{i}+t_{i}\right)\leq \sum_{i\in I \setminus J} (||e_i||^2-d_{i})
 <\infty.\end{equation}
Combining \eqref{2trace3} and \eqref{2trace4} yields that $K$ and $T$ are trace class.

Set $\boldsymbol{p} = (p_{i})_{i\in I}$. Since $p_{i} = d_{i} - k_{i}+t_{i}$ and $1-p_{i} = 1-d_{i}+k_{i} - t_{i}$, from the assumption that $(d_{i})_{i\in J}$ and $(1-d_{i})_{i\in I \setminus J} $ are summable, we see that $(p_{i})_{i\in J}$ and $(1-p_{i})_{i\in I \setminus J }$ are summable. Thus, $f_{\boldsymbol{p}}(\alpha)<\infty$ for all $\alpha\in(0,1)$. By Naimark's dilation theorem $(p_{i})$ is the diagonal of a projection with respect to an orthonormal basis on some larger Hilbert space, by Theorem \ref{Kadison} we have
\[f_{\boldsymbol{p}}'(\alpha) = \sum_{p_{i}\geq \alpha }(1-p_{i}) -\sum_{p_{i}<\alpha }p_{i} \in \Z
\qquad\text{for a.e. }\alpha\in (0,1).\]

By the spectral theorem for compact operators, there is an orthonormal basis $(f_{i})_{i\in J_{1}}$ for $\ran(\mathbf{I}-P)$ consisting of eigenvectors of $K$, with associated eigenvalues $(\lambda_{i})_{i\in J_{1}}$. There is also an orthonormal basis $(f_{i})_{i\in J_{2}}$ for $\ran(P)$ consisting or eigenvectors of $T$ with associated eigenvalues $(1-\lambda_{i})_{i\in J_{2}}$. Since $K$ and $T$ are trace class, both $(\lambda_{i})_{i\in J_{1}}$ and $(1-\lambda_{i})_{i\in J_{2}}$ are summable. For $i\in J_{1}$ we have $(P-T)f_{i} = 0$ and thus $Ef_{i} = \lambda_{i}f_{i}$. For $i\in J_{2}$ we have $Kf_{i} = 0$ so that $Ef_{i} = \lambda_{i}$. 
Therefore, $(f_{i})_{i\in J_{1}\cup J_{2}}$ is an orthonormal basis for $\Hil$ consisting of eigenvectors of $E$ with associated eigenvalues $(\lambda_{i})_{i\in J_{1}\cup J_{2}}$. This shows that $E$ is diagonalizable. 

From the definitions of $K$ and $T$ we see that $\lambda_{i}<\alpha$ for $i\in J_{1}$ and $\lambda_{i}\geq \alpha$ for $i\in J_{2}$. Thus,

\[C_{\boldsymbol\lambda}(\alpha)=\sum_{\lambda_{i}<\alpha}\lambda_{i} = \sum_{i\in J_{1}}\lambda_{i}<\infty\quad\text{and}\quad D_{\boldsymbol\lambda}(\alpha)=\sum_{\lambda_{i}\geq \alpha}(1-\lambda_{i})=\sum_{i\in J_{2}}(1-\lambda_{i})<\infty.\]
This implies $f_{\boldsymbol\lambda}(\alpha)<\infty$. Moreover, using \eqref{2trace3} and \eqref{2trace4} we have
\begin{align*}
f_{\boldsymbol\lambda}(\alpha) & = (1-\alpha)\sum_{i\in J_{1}}\lambda_{i} + \alpha\sum_{i\in J_{2}}(1-\lambda_{i}) = (1-\alpha)\tr(K)+\alpha\tr(T) = (1-\alpha)\sum_{i\in I}k_{i} + \alpha\sum_{i\in I}t_{i}\\
 & = (1-\alpha)\sum_{i\in J}\left(k_{i} + \frac{\alpha}{1-\alpha}t_{i}\right) + \alpha\sum_{i\in I \setminus J }\left(\frac{1-\alpha}{\alpha}k_{i} + t_{i}\right)\\
 & \leq 
  (1-\alpha)\sum_{i\in J}d_i + \alpha\sum_{i\in I \setminus J }( ||e_i||^2 - d_i),
\end{align*}
which shows \eqref{trap1}. 
By \eqref{trap0} and \eqref{trapa} we deduce that for all $\alpha \in (0,1)$
\[
 f_{\boldsymbol{d}}(\alpha)=(1-\alpha)C_{\boldsymbol{d}}(\alpha) + \alpha D_{\boldsymbol{d}}(\alpha) <\infty.
\]
Hence, for a.e. $\alpha \in (0,1)$,
\[
\begin{aligned}
f_{\boldsymbol{d}}'(\alpha)-f_{\boldsymbol\lambda}'(\alpha) 
&=
\sum_{d_{i}\geq \alpha }(1-d_{i}) -\sum_{d_{i}<\alpha }d_{i} 
-\sum_{i \in J_2 }(1-\lambda_{i}) +\sum_{i \in J_1 }\lambda_{i} \\
&\equiv 
\sum_{i \in I \setminus J }(1-d_{i}) -\sum_{i\in J }d_{i} - \tr(T) + \tr(K) \mod 1
\\
&=
\sum_{i \in I \setminus J }(1-p_{i}) -\sum_{i \in J }p_{i} \equiv
f_{\boldsymbol{p}}'(\alpha) \equiv 0
\mod 1.
\end{aligned}
\]
This proves \eqref{trap2}.
\end{proof}

\begin{remark}\label{trar} Suppose that we have an equality in \eqref{trap1}. Then, by analyzing the proof of Lemma \ref{trap} we deduce that we have equalities in \eqref{2trace3} and \eqref{2trace4} and consequently we have $d_{i}= k_{i}+\frac{\alpha}{1-\alpha}t_{i}$ for all $i\in J$ and $||e_i||^2-d_{i} = \frac{1-\alpha}{\alpha}k_{i}+t_{i}$ for all $i \in I \setminus J$. Therefore, 
\[
\begin{aligned}
\langle Te_i,e_i \rangle =t_i=(1-\alpha) p_i= \langle (1-\alpha)Pe_i, e_i \rangle& \qquad\text{for all }i\in J,
\\
\langle Ke_i,e_i \rangle  = k_i= \alpha (||e_i||^2- p_i) = \langle \alpha (\mathbf I - P) e_i,e_i \rangle  
& \qquad\text{ for all } i \in I \setminus J.
\end{aligned}
\]
By \eqref{2trace1} and \eqref{2trace2} we deduce that 
\[e_i \in \ran \pi([0,\alpha]) \text{ for }i \in J
\qquad\text{and}\qquad 
e_i \in \ran \pi([\alpha,1])\text{ for } i \in J \setminus I.
\]
\end{remark}

We will prove Theorem \ref{nec} in several parts. To avoid repetition we will set some notation for the remainder of Section \ref{necProof}.

\begin{nota} Let \(E\) be a self-adjoint operator on a Hilbert space \(\Hil\) such that \(0\) and \(1\) are the extreme points of the essential spectrum \(\sigma_{ess}(E)\) of \(E\), that is,
\begin{equation}\label{nec1}
\{0,1 \} \subset \sigma_{ess}(E) \subset [0,1].
\end{equation}

Let \(\pi\) be the projection valued measure such that
\[
E = \int_{\R}\lambda\,d\pi(\lambda),\]
and define the operators
\[
L_0 = \int_{(-\infty,0)} \lambda \pi(\lambda)
\qquad\text{and}\qquad
L_1 = \int_{(1,\infty)} \lambda \pi(\lambda),
\]
and the projection \(P = \pi([0,1])\).

By \eqref{nec1}, the operators $L_0$ and $L_1$ are diagonalizable. Hence, there exists $N_0,N_1\in \N \cup\{ 0,\infty\}$, a sequence $\lambda_{-1} \le \lambda_{-2} \le \ldots <0$, a sequence $\lambda_1 \ge \lambda_2 \ge \ldots >1 $, and orthonormal sequences $(f_{-j})_{j=1}^{N_0}$ and $(f_j)_{j=1}^{N_1}$ such that
\[
L_0 f = \sum_{j=1}^{N_0} \lambda_{-j} \langle f, f_{-j} \rangle f_{-j} 
\quad\text{and}\quad
L_1 f = \sum_{j=1}^{N_1} \lambda_{j} \langle f, f_{j} \rangle f_j 
\qquad\text{for }f\in \mathcal H.
\]

Let $\boldsymbol\lambda=(\lambda_{j})_{j\in J}$ be the list of all eigenvalues of $E$ with multiplicity (which is possibly an empty list). Since \((\lambda_{-j})_{j=1}^{N_{0}}\) and \((\lambda_{j})_{j=1}^{N_{1}}\) are eigenvalues of \(E\) we will assume that the indexing sets of these sequences are subsets of \(J\).

Let $\boldsymbol{d}=(d_{i})_{i\in I}$ be a diagonal of $E$, which is given by $d_{i} = \langle Ee_{i},e_{i}\rangle$, with respect to some orthonormal basis $(e_{i})_{i\in I}$ of $\Hil$.

Define
\[\delta^{L} = \liminf_{\beta\nearrow 0}(f_{\boldsymbol\lambda}(\beta)-f_{\boldsymbol{d}}(\beta))\quad\text{and}\quad \delta^{U} = \liminf_{\beta\searrow 1}(f_{\boldsymbol\lambda}(\beta)-f_{\boldsymbol{d}}(\beta)).\]

Finally, define the following two quantities, which will be shown to be well-defined in the sequel:

\begin{equation}\label{tra5a}
\delta_0:=\sum_{j=1}^{N_0} |\lambda_{-j}| \sum_{0\le d_i\le 1} |\langle e_i, f_{-j} \rangle|^2  +\sum_{d_i<0} \langle EP e_i,e_i \rangle + \sum_{j=1}^{N_1} \sum_{d_i<0} |\langle e_{i}, f_{j} \rangle|^2 .
\end{equation}

\begin{equation}\label{tra10a}
\delta_1:= 
\sum_{j=1}^{N_1}(\lambda_j-1)  \sum_{0 \le d_i\le 1} |\langle e_i, f_j \rangle|^2+
\sum_{d_i>1} \langle (\mathbf I -E)P e_i,e_i \rangle
+
\sum_{j=1}^{N_0} \sum_{d_i>1} |\langle e_{i}, f_{-j} \rangle|^2.
\end{equation}

\end{nota}

\begin{lem}\label{del0} If \(\delta^{L}<\infty\), then the following three series converge:
\[
\sum_{j=1}^{N_0}
|\lambda_{-j}|  \sum_{d_i\ge 0} |\langle e_i, f_{-j} \rangle|^2,
\qquad
\sum_{d_i<0} \langle EP e_i,e_i \rangle,
\qquad
\sum_{j=1}^{N_1}\lambda_j \sum_{d_i<0} |\langle e_{i}, f_{j} \rangle|^2,
\]
and hence \(\delta_{0}\) is well-defined. Moreover,
\begin{equation}\label{tra5b}
\delta_0 +  \sum_{j=1}^{N_0} |\lambda_{-j}| \sum_{d_i > 1} |\langle e_i, f_{-j} \rangle|^2   + \sum_{j=1}^{N_1} (\lambda_j-1) \sum_{d_i<0} |\langle e_{i}, f_{j} \rangle|^2 \le \delta^L,
\end{equation}
with equality if \(\sum_{\lambda_{j}<0}|\lambda_{j}|<\infty\).
\end{lem}

\begin{proof} Since $E=L_0+EP+L_1$, we have
\begin{equation}\label{tra0}
d_i = \sum_{j=1}^{N_0} \lambda_{-j} | \langle e_i, f_{-j} \rangle |^2 +
\langle EP e_i,e_i \rangle + \sum_{j=1}^{N_1} \lambda_{j} | \langle e_i, f_{j} \rangle |^2.
\end{equation}

\textbf{Case 1}: Suppose $\sum_{\lambda_j<0} |\lambda_j| = \infty$. Since $\delta^L<\infty$, this implies that $\sum_{i\in I_0} |d_i|=\infty$, where $I_0=\{i\in I: d_i<0\}$. Hence, $I_0$ is infinite and there is a bijection $\sigma: -\N \to I_0$ such that $d_{\sigma(-1)} \le d_{\sigma(-2)} \le \ldots <0$. By Proposition \ref{LR} and Lemma \ref{nel} applied to $-E$ and orthonormal sequence $(e_{\sigma(-i)})_{i\in\N}$ we deduce that
\begin{equation}\label{tra5}
\begin{aligned}
\delta^L & = \liminf_{M\to\infty} \sum_{i=1}^M (|\lambda_{-i}| - |d_{\sigma(-i)}|) 
\\
 & \geq \sum_{j=1}^{\infty}
|\lambda_{-j}| \bigg( 1-\sum_{i\in \N} |\langle e_{\sigma(-i)}, f_{-j} \rangle|^2 \bigg) + \sum_{i\in \N} \langle (E \circ \pi([0,\infty))) e_{\sigma(-i)},e_{\sigma(-i)} \rangle
\\
 & = 
\sum_{j=1}^{\infty}
|\lambda_{-j}| \bigg(\sum_{d_{i}\geq 0} |\langle e_{i}, f_{-j} \rangle|^2 \bigg) +
\sum_{i:d_{i}<0} \langle EP e_{i},e_{i} \rangle +\sum_{j=1}^{N_1} \lambda_j \sum_{i:d_{i}<0} |\langle e_{i}, f_{j} \rangle|^2\\
 & = \delta_0 +  \sum_{j=1}^{N_0} |\lambda_{-j}| \sum_{d_i > 1} |\langle e_i, f_{-j} \rangle|^2   + \sum_{j=1}^{N_1} (\lambda_j-1) \sum_{d_i<0} |\langle e_{i}, f_{j} \rangle|^2.
\end{aligned}
\end{equation}
This proves the lemma under the assumption that \(\sum_{\lambda_{j}<0}|\lambda_{j}|=\infty\).

\textbf{Case 2}: Suppose \(\sum_{\lambda_{j}<0}|\lambda_{j}|<\infty\). This implies that the negative part of \(E\) is trace class, and hence we can apply Lemma \ref{nelt} and Proposition \ref{LRS} to \(-E\) to obtain
\begin{align*}
\delta^{L} & = \sum_{j=1}^{N_{0}}
|\lambda_{-j}| \bigg( 1-\sum_{i:d_{i}<0} |\langle e_{i}, f_{-j} \rangle|^2 \bigg) + \sum_{i:d_{i}<0} \langle (E \circ \pi([0,\infty))) e_{i},e_{i} \rangle\\
 & = \sum_{j=1}^{N_{0}}
|\lambda_{-j}| \bigg(\sum_{i:d_{i}\geq 0} |\langle e_{i}, f_{-j} \rangle|^2 \bigg) + \sum_{i:d_{i}<0} \langle EP e_{i},e_{i} \rangle + \sum_{j=1}^{N_1} \lambda_j \sum_{i:d_{i}<0} |\langle e_{i}, f_{j} \rangle|^2\\
 & = \delta_0 +  \sum_{j=1}^{N_0} |\lambda_{-j}| \sum_{d_i > 1} |\langle e_i, f_{-j} \rangle|^2   + \sum_{j=1}^{N_1} (\lambda_j-1) \sum_{d_i<0} |\langle e_{i}, f_{j} \rangle|^2.\qedhere
\end{align*}

\end{proof}

\begin{lem}\label{del1} If \(\delta^{U}<\infty\), then the following three series converge:
\[
\sum_{j=1}^{N_1}
(\lambda_j-1)  \sum_{d_i\le 1} |\langle e_i, f_j \rangle|^2,
\qquad
\sum_{d_i>1} \langle (\mathbf I -E)P e_i,e_i \rangle,
\qquad
\sum_{j=1}^{N_0}(1+|\lambda_{-j}|) \sum_{d_i>1} |\langle e_{i}, f_{-j} \rangle|^2 .
\]
and hence \(\delta_{1}\) is well-defined. Moreover,
\begin{equation}\label{tra10b}
\delta_1+ 
\sum_{j=1}^{N_1}(\lambda_j-1)  \sum_{ d_i<0} |\langle e_i, f_j \rangle|^2
+
\sum_{j=1}^{N_0} |\lambda_{-j}| \sum_{d_i>1} |\langle e_{i}, f_{-j} \rangle|^2  \le \delta^U
\end{equation}
with equality if \(\sum_{j:\lambda_{j}>1}(\lambda_{j}-1)<\infty\).
\end{lem}

\begin{proof}

\textbf{Case 1}: Suppose $\sum_{\lambda_j>1} (\lambda_j-1) = \infty$. Since \(\delta^{U}<\infty\), we have \(\sum_{i\in I_{1}}(d_{i}-1)=\infty\), where \(I_{1} = \{i\in I : d_{i}>1\}\). Hence, \(I_{1}\) is infinte, and there is a bijection \(\sigma:\N\to I_{1}\) such that \(d_{\sigma(1)}\geq d_{\sigma(2)}\geq \ldots>0\). By Proposition \ref{LR} and Lemma \ref{nel} applied to $E-\mathbf{I}$ and orthonormal sequence $(e_{\sigma(i)})_{i\in\N}$ we deduce that
\begin{equation*}\label{tra10}
\begin{aligned}
\delta^U &= \liminf_{M\to\infty} \sum_{i=1}^M (\lambda_{i} - d_{\sigma(i)}) 
\\
& \geq \sum_{j=1}^{\infty}
(\lambda_{j}-1) \bigg( 1-\sum_{i\in \N} |\langle e_{\sigma(i)}, f_{j} \rangle|^2 \bigg) - \sum_{i\in \N} \langle ((E-\mathbf I) \circ \pi((-\infty,1])) e_{\sigma(i)},e_{\sigma(i)} \rangle
\\
& =
\sum_{j=1}^{\infty}
(\lambda_{j}-1) \bigg(\sum_{i:d_{i}\leq 1} |\langle e_{i}, f_{j} \rangle|^2 \bigg) +
\sum_{i:d_{i}>1} \langle (\mathbf I -E)P e_{i},e_{i} \rangle
+ \sum_{j=1}^{\infty}
(1+|\lambda_{-j}|) \sum_{i:d_{i}>1} |\langle e_{i}, f_{-j} \rangle|^2\\
 & = \delta_{1} + \sum_{j=1}^{\infty}|\lambda_{-j}|\sum_{i: d_{i}>1}|\langle e_{i},f_{-j}\rangle|^{2} + \sum_{j=1}^{\infty}(\lambda_{j}-1)\sum_{i:d_{i}<0}|\langle e_{i},f_{j}\rangle|^{2}.
\end{aligned}
\end{equation*}
This proves the claim under the assumption that \(\sum_{\lambda_j>1} (\lambda_j-1) = \infty\).

\textbf{Case 2}: Suppose \(\sum_{j:\lambda_{j}>1}(\lambda_{j}-1)<\infty\). This implies that the positive part of \(E-\mathbf{I}\) is trace class, and hence we can apply Lemma \ref{nelt} and Proposition \ref{LRS} to \(E-\mathbf{I}\) to obtain
\begin{align*}
\delta^U &=  \sum_{j=1}^{N_{1}}
(\lambda_{j}-1) \bigg( 1-\sum_{i:d_{i}>1} |\langle e_{i}, f_{j} \rangle|^2 \bigg) - \sum_{i:d_{i}>1} \langle ((E-\mathbf I) \circ \pi((-\infty,1])) e_{i},e_{i} \rangle
\\
& =
\sum_{j=1}^{N_{1}}
(\lambda_{j}-1) \bigg(\sum_{i:d_{i}\leq 1} |\langle e_{i}, f_{j} \rangle|^2 \bigg) +
\sum_{i:d_{i}>1} \langle (\mathbf I -E)P e_{i},e_{i} \rangle
+ \sum_{j=1}^{N_{0}}
(1+|\lambda_{-j}|) \sum_{i:d_{i}>1} |\langle e_{i}, f_{-j} \rangle|^2\\
 & = \delta_{1} + \sum_{j=1}^{N_{0}}|\lambda_{-j}|\sum_{i: d_{i}>1}|\langle e_{i},f_{-j}\rangle|^{2} + \sum_{j=1}^{N_{1}}(\lambda_{j}-1)\sum_{i:d_{i}<0}|\langle e_{i},f_{j}\rangle|^{2}.\qedhere
\end{align*}
\end{proof}

\begin{lem}\label{neci} If \(\delta^{U}<\infty\), then \(\sum_{d_{i}<1}(1-d_{i})=\infty\). Likewise, if \(\delta^{L}<\infty\), then \(\sum_{d_{i}>0}d_{i}=\infty\).
\end{lem}

\begin{proof}
Assume that $\delta^U<\infty$. On the contrary, suppose that
$
\sum_{d_i <1} (1-d_i) <\infty.
$ 
By Lemma \ref{del1} the following two series converge
\[
\begin{aligned}
 \sum_{d_i \le 1} ( \langle \mathbf I - E)P_1 e_i,e_i \rangle & = -\sum_{j=1}^{N_1}
(\lambda_j-1)  \sum_{d_i\le 1} |\langle e_i, f_j \rangle|^2,
\\
\sum_{d_i>1} \langle (\mathbf I -E)(P_0+P) e_i,e_i \rangle & = \sum_{d_i>1} \langle (\mathbf I -E)P e_i,e_i \rangle+\sum_{j=1}^{N_0}(1+|\lambda_{-j}|) \sum_{d_i>1} |\langle e_{i}, f_{-j} \rangle|^2.
\end{aligned}
\]
Using the identity
\[
\sum_{d_i \le 1} (1-d_i) = \sum_{d_i \le 1} \big( \langle (\mathbf I - E)P_1 e_i,e_i \rangle + \langle (\mathbf I - E)(P_0+P) e_i,e_i \rangle \big)
\]
we deduce that
\[
\sum_{i\in I} \langle (\mathbf I - E)(P_0+P) e_i,e_i \rangle<\infty.
\]
This implies that the operator  $(\mathbf I - E)(P_0+P)$ is trace class, which contradicts our assumption that $0 \in \sigma_{ess}(E)$. The other claim with \(\delta^{L}<\infty\) follows by an analogous argument.
\end{proof}

The following corollary shows the necessity of \eqref{mainthm4}. As a corollary we also obtain the necessity of \eqref{mainthm5} by symmetry.

\begin{cor}\label{part45}
Suppose \(\delta^{L}+\delta^{U}<\infty\). If $\sum_{\lambda_j<0} |\lambda_j|<\infty$, then $\delta^L-\delta_0 \le \delta^U- \delta_1$. If $\sum_{\lambda_j>1} (\lambda_j-1)<\infty$, then $\delta^U- \delta_1 \le \delta^L-\delta_0$.
\end{cor}

\begin{proof} Set 
\[\Delta = \sum_{j=1}^{N_1}(\lambda_j-1)  \sum_{ d_i<0} |\langle e_i, f_j \rangle|^2
+
\sum_{j=1}^{N_0} |\lambda_{-j}| \sum_{d_i>1} |\langle e_{i}, f_{-j} \rangle|^2.\]
By Lemmas \ref{del0} and \ref{del1}, if $\sum_{\lambda_j<0} |\lambda_j|<\infty$, then \(\delta^{L}-\delta_{0} = \Delta\leq \delta^{U}-\delta_{1}\). Similarly, if $\sum_{\lambda_j>1} (\lambda_j-1)<\infty$, then \(\delta^{L}-\delta_{0} \leq \Delta = \delta^{U}-\delta_{1}\).
\end{proof}

\begin{lem}\label{bigeqs} If \(f_{\boldsymbol d}(\alpha)<\infty\) for some \(\alpha\in(0,1)\) and \(\delta^{L}+\delta^{U}<\infty\), then
\begin{equation}\label{tra7a}
\begin{aligned}
 \sum_{0 \le d_i <\alpha } d_i &+ \delta_0 = \sum_{ d_i <\alpha} \langle EP e_i,e_i \rangle
 \\
 &+
 \sum_{j=1}^{N_0}
|\lambda_{-j}| \sum_{\alpha \le d_i \le 1 } |\langle e_{i}, f_{-j} \rangle|^2
+ \sum_{j=1}^{N_1}
\lambda_{j} \sum_{0 \le d_i < \alpha } |\langle e_{i}, f_{j} \rangle|^2 + \sum_{j=1}^{N_1} \sum_{d_i<0} |\langle e_{i}, f_{j} \rangle|^2,
\end{aligned}
\end{equation}
and
\begin{equation}\label{tra12}
\begin{aligned}
\sum_{\alpha \le d_i \le 1} (1-d_i)  &+\delta_1= \sum_{ d_i \ge \alpha} \langle (\mathbf I - E)P e_i,e_i \rangle 
+ \sum_{j=1}^{N_1}
(\lambda_{j}-1) \sum_{0\le d_i \le \alpha } |\langle e_{i}, f_{j} \rangle|^2
\\
& 
 + \sum_{j=1}^{N_0}
(1+|\lambda_{-j}|) \sum_{  \alpha \le d_i \le 1} |\langle e_{i}, f_{-j} \rangle|^2 + \sum_{j=1}^{N_0} \sum_{d_i>1} |\langle e_{i}, f_{-j} \rangle|^2.
\end{aligned}
\end{equation}
\end{lem}

\begin{proof} By \eqref{tra0} we have
\[
\sum_{0\le d_i <\alpha} d_i = \sum_{0 \le d_i <\alpha} \bigg( \sum_{j=1}^{N_0} \lambda_{-j} | \langle e_i, f_{-j} \rangle |^2 +
\langle EP e_i,e_i \rangle  + \sum_{j=1}^{N_1} \lambda_{j} | \langle e_i, f_{j} \rangle |^2 \bigg).
\]
Since $\sum_{0\le d_i <\alpha} d_i <\infty$, we can interchange the order of summation
\begin{equation}\label{tra6}
\begin{aligned}
\sum_{0 \le d_i <\alpha} \langle EP e_i,e_i \rangle
&=  \sum_{0\le d_i <\alpha} d_i + \sum_{j=1}^{N_0}
|\lambda_{-j}| \sum_{0\le d_ i  \le 1 } |\langle e_i, f_{-j} \rangle|^2 
\\
&
\quad - \sum_{j=1}^{N_0}
|\lambda_{-j}| \sum_{\alpha \le d_i \le 1 } |\langle e_{i}, f_{-j} \rangle|^2
- \sum_{j=1}^{N_1}
\lambda_{j} \sum_{0 \le d_i < \alpha } |\langle e_{i}, f_{j} \rangle|^2.
\end{aligned}
\end{equation}
From \eqref{tra6} and the definition of \(\delta_{0}\) we obtain \eqref{tra7a}.

Again, by \eqref{tra0} we have
\[
\sum_{\alpha \le d_i \le 1} (1-d_i) = \sum_{\alpha \le d_i \le 1} \bigg( \sum_{j=1}^{N_1} (1-\lambda_j) | \langle e_i, f_{j} \rangle |^2 +
\langle (\mathbf I - E)P e_i,e_i \rangle 
+ 
 \sum_{j=1}^{N_0} (1+|\lambda_{-j}| ) | \langle e_i, f_{-j} \rangle |^2
\bigg).
\]
Hence,
\begin{equation}\label{tra11}
\begin{aligned}
\sum_{\alpha \le d_i \le 1} \langle (\mathbf I -E)P & e_i,e_i  \rangle
=  \sum_{\alpha \le d_i \le 1} (1-d_i)  + \sum_{j=1}^{N_1}
(\lambda_{j}-1) \sum_{0\le d_i \le 1} |\langle e_{i}, f_{j} \rangle|^2
\\
&+\sum_{j=1}^{N_1}
(\lambda_{j}-1) \sum_{0\le d_i \le \alpha } |\langle e_{i}, f_{j} \rangle|^2
 - \sum_{j=1}^{N_0}
(1+|\lambda_{-j}|) \sum_{  \alpha \le d_i \le 1} |\langle e_{i}, f_{-j} \rangle|^2.
\end{aligned}
\end{equation}
Combining \eqref{tra11} and the definition of \(\delta_{1}\) yields \eqref{tra12}.
\end{proof}

\begin{lem}\label{part2} If \(f_{\boldsymbol d}(\alpha)<\infty\) for some \(\alpha\in(0,1)\) and \(\delta^{L}+\delta^{U}<\infty\), then \(E\) is diagonalizable, and 
\begin{equation}\label{tra21} f_{\boldsymbol\lambda}(\alpha)  \leq f_{\boldsymbol{d}}(\alpha) + (1-\alpha)\delta_{0} + \alpha\delta_{1}\end{equation}
for all $\alpha\in(0,1)$. Moreover, if equality holds in \eqref{tra21} for some \(\alpha\in(0,1)\), then \(E\) decouples at \(\alpha\).
\end{lem}

\begin{proof} Define $g_i=Pe_i$, $i\in\N$. Then, $(g_i)_{i\in\N}$ is a Parseval frame for the subspace $\mathcal K=P(\mathcal H) \subset \mathcal H$. Let $\tilde d_i = \langle EP e_i, e_i \rangle$ for $i\in I$ and $J=\{i\in I: d_i<\alpha\}$.
By Lemma \ref{bigeqs} we deduce that
\begin{equation}\label{tra7}\sum_{i:d_{i}\in[0,\alpha)}d_{i} + \delta_{0}\geq \sum_{i:d_{i}<\alpha}\langle EPe_{i},e_{i}\rangle,\end{equation}
and
\begin{equation}\label{tra13}\sum_{i:d_{i}\in[\alpha,1]}(1-d_{i}) + \delta_{1}\geq \sum_{i:d_{i}\geq\alpha}\langle(\mathbf{I}-E)Pe_{i},e_{i}\rangle.\end{equation}
Thus, the assumption \eqref{trap0} in Lemma \ref{trap} is met for the restriction of the operator $E$ to $\mathcal K$ and a Parseval frame $(g_i)_{i\in\N}$.
Hence, we deduce that $EP$ is diagonalizable. Since the operator $E$ is block diagonal sum $E=L_0+EP+L_1$, it is diagonalizable as well.

Let $\boldsymbol\lambda=(\lambda_{j})_{j\in J}$ be the eigenvalue list of $E$. 
By Lemma \ref{trap}, \eqref{tra7}, and \eqref{tra13} we have for all $\alpha\in(0,1)$,
\begin{equation}\label{tra15}
f_{\boldsymbol{\lambda}}(\alpha) \le  (1-\alpha)\sum_{i\in J}\tilde d_i + \alpha\sum_{i\in I \setminus J }( ||g_i||^2 - \tilde d_i) \le
f_{\boldsymbol{d}}(\alpha) + (1-\alpha)\delta_{0} + \alpha\delta_{1}.
\end{equation}

Suppose that for some $\alpha \in (0,1)$ we have an equality in \eqref{tra21}. This means that we have equalities in the two inequalities appearing in \eqref{tra15}. In particular, we have equalities in \eqref{tra7} and \eqref{tra13}. By \eqref{tra7a} and \eqref{tra12} we deduce that
\[
\begin{aligned}
\langle e_i, f_j \rangle &= 0 \qquad\text{for }i\in I, d_i<\alpha \text{ and } j \ge 1,
\\
\langle e_i, f_{-j} \rangle &= 0 \qquad\text{for }i\in I, d_i\ge \alpha \text{ and } j \ge 1.
\end{aligned}
\]
On the other hand, 
Remark \ref{trar} yields
\[g_i=Pe_i \in \begin{cases}
\ran \pi([0,\alpha]) & \text{if }  d_i<\alpha,\\ 
 \ran \pi([\alpha,1]) & \text{if } d_i\ge \alpha.
 \end{cases}
\]
Combining these observations yields
\[
e_i \in \begin{cases}
\ran \pi((-\infty,\alpha]) & \text{if }  d_i < \alpha,\\ 
 \ran \pi([\alpha,\infty)) & \text{if } d_i\ge \alpha.
 \end{cases}
\]
In addition, if $d_i=\alpha$, then $e_i \in  \ran \pi(\{\alpha\})$. Hence, the spaces $\mathcal H_0=\overline{\lspan }\{ e_i: d_i < \alpha \}$ and $\mathcal H_1=\overline{\lspan} \{ e_i: d_i \ge \alpha \}$  are invariant under $E$. 
\end{proof}

\begin{lem}\label{part3} If \(f_{\boldsymbol d}(\alpha)<\infty\) for some \(\alpha\in(0,1)\) and \(\delta^{L}+\delta^{U}<\infty\), then \(f_{\boldsymbol\lambda}(\alpha)<\infty\) for all \(\alpha\in(0,1)\), and there exists \(\alpha\in(0,1)\) such that
\begin{equation}\label{inttrace}f_{\boldsymbol\lambda}'(\alpha) - f_{\boldsymbol d}'(\alpha) \equiv \delta_{1}-\delta_{0} \mod 1.\end{equation}
\end{lem}

\begin{proof} By \eqref{tra7a} and \eqref{tra12}, for a.e. $\alpha \in (0,1)$ we have
\begin{equation}\label{tra20}
\begin{aligned}
& f_{\boldsymbol{d}}'(\alpha) + \delta_1-\delta_0=  \sum_{\alpha \le d_i \le 1} (1-d_i)  -  \sum_{0 \le d_i <\alpha } d_i + \delta_1 -\delta_0 
\\   
&=  \sum_{ d_i \ge \alpha} \langle (\mathbf I - E)P e_i,e_i \rangle 
+
\sum_{j=1}^{N_0}  \sum_{d_i\ge \alpha } |\langle e_{i}, f_{-j} \rangle|^2  
- \sum_{ d_i <\alpha} \langle EP e_i,e_i \rangle  
-
\sum_{j=1}^{N_1} \sum_{ d_i<\alpha} |\langle e_i, f_j \rangle|^2.
\end{aligned}
\end{equation}
By \eqref{tra20} we have
\[
\begin{aligned}
f_{\boldsymbol{d}}'(\alpha) + \delta_1-\delta_0 &=
 \sum_{ d_i \ge \alpha} ( \langle (\mathbf I - E)P e_i,e_i \rangle + \langle P_0e_{i}, e_i\rangle ) 
- \sum_{ d_i <\alpha} ( \langle EP e_i,e_i \rangle + \langle P_1e_{i}, e_i\rangle )
\\
&=
 \sum_{ d_i \ge \alpha} \langle (\mathbf I - EP -P_1) e_i,e_i \rangle  
- \sum_{ d_i <\alpha}  \langle (EP+P_1) e_i,e_i \rangle 
.
\end{aligned}
\]
Let \(\widehat{\boldsymbol\lambda}\) be the eigenvalue list of \(EP+P_{1}\). Applying Lemma \ref{trap} to the operator $EP+P_1$ and orthonormal basis $(e_{i})_{i\in I}$ implies that the diagonal $\widehat{\boldsymbol{d}}=(\langle (EP+P_1) e_i, e_i \rangle)_{i\in I}$, satisfies
\begin{equation}
  f_{\boldsymbol{d}}'(\alpha) + \delta_1-\delta_0 
  \equiv f_{\widehat{\boldsymbol{d}}}'(\alpha)
  \equiv f_{\widehat{\boldsymbol\lambda}}'(\alpha) \mod 1
\qquad\text{for a.e. }\alpha\in (0,1).
\end{equation}
Finally, \eqref{inttrace} follows from the fact that operators $E$ and $EP+P_1$ have the same eigenvalues in the open interval $(0,1)$, and consequently $f_{\boldsymbol\lambda}(\alpha)=f_{\widehat{\boldsymbol\lambda}}(\alpha)$ for $\alpha \in (0,1)$.
\end{proof}

\begin{lem}\label{daccum} If \(f_{\boldsymbol d}(\alpha)<\infty\) for some \(\alpha\in(0,1)\) and \(\delta^{L}+\delta^{U}<\infty\), then \(\boldsymbol d\) has accumulation points at \(0\) and \(1\).
\end{lem}

\begin{proof} Note that the assumptions \eqref{nec1} and \(f_{\boldsymbol d}(\alpha)<\infty\) for some \(\alpha\in(0,1)\) imply that \(0\) and \(1\) are the only possible accumulation points of \(\boldsymbol d\). From Lemma \ref{neci} we see that \(\sum_{d_{i}<1}(1-d_{i}) = \sum_{d_{i}>0}d_{i} = \infty\), and hence both \(0\) and \(1\) are accumulation points of \(\boldsymbol d\).
\end{proof}

\begin{lem}\label{decoup01} Assume \(f_{\boldsymbol d}(\alpha)<\infty\) for some \(\alpha\in(0,1)\) and \(\delta^{L}+\delta^{U}<\infty\). If \(\delta_{0}=0\) and \(\delta_{1}=\delta^{U}\), then \(E\) decouples at \(0\). If \(\delta_{1}=0\) and \(\delta_{0}=\delta^{L}\), then \(E\) decouples at \(1\). 
\end{lem}

\begin{proof} We will only prove the second implication. The other follows by a symmetric argument. Suppose that $\delta_0=\delta^L$ and $\delta_1=0$. By \eqref{tra5b} we deduce that $e_i \in \pi((-\infty,1])$ if $d_i \le 1$ and $e_i\in \pi([1,\infty))$ if $d_i \ge 1$. In particular, if $d_i=1$, then $e_i \in \pi(\{1\})$ is an eigenvector of $E$. Hence, the spaces $\mathcal H_0=\overline{\lspan }\{ e_i: d_i < 1\}$ and $\mathcal H_1=\overline{\lspan} \{ e_i: d_i \ge 1\}$  are invariant under $E$. 
\end{proof}

Now we are ready to give the proof of Theorem \ref{nec}.

\begin{proof}[Proof of Theorem \ref{nec}] That \(E\) is diagonalizable and conclusions (i) and (ii) follow from Lemma \ref{part2}. Lemma \ref{daccum} shows that \(\boldsymbol d\) has accumulation points at \(0\) and \(1\). Conclusion (iii) follows from Lemma \ref{part3}, whereas (iv) and (v) follow from Corollary \ref{part45}. 
The operator \(E\) decouples at \(\alpha=0\) when \(\delta^{L}=0\), and at \(\alpha=1\) when \(\delta^{U}=0\) in light of Proposition \ref{p211}.
Hence, the ``In addition'' part of the theorem follows from Lemmas \ref{part2} and \ref{decoup01}.
\end{proof}

We finish the section with some necessary conditions that follow when $(\mathfrak d_\alpha)$ in Theorem \ref{nec} happens for some $\alpha \in [0,1]$. The decoupling of $E$ yields the following additional list of necessary conditions.

\begin{thm}\label{necdec} Under the same assumptions as in Theorem \ref{nec} suppose one of the following happens:
\begin{enumerate}
\item[$(\mathfrak d_0)$] holds. 
If $\sum_{0 < d_i < 1} (1-d_i)<\infty$, then 
$$
\sum_{0 <\lambda_j < 1} (1-\lambda_j) \le \delta_1+\sum_{0 < d_i < 1} (1-d_i),$$
and additionally if $\sum_{\lambda_j>0} \lambda_j <\infty$, then we also have
$$
 \delta_0+\sum_{ 0\le d_i <1} (1-d_i) \le  \sum_{0\le \lambda_j< 1} (1-\lambda_j) ,
$$
If
$\sum_{0 < d_i < 1} (1-d_i)=\infty$, then $\sum_{0\le \lambda_j< 1}(1- \lambda_j) =\infty.
$

\item[$(\mathfrak d_\alpha)$] holds for some  $\alpha \in (0,1)$. Then
 there exists $z_1,z_2 \in \N \cup \{0\}$ such that:
 \begin{itemize}
 \item
$\#|\{j:\lambda_j=\alpha\}| = z_1+z_2+ \#|\{i: d_i=\alpha\}|$,
\item
$(d_i)_{d_i<\alpha}$ is a diagonal of a compact operator with an eigenvalue list $(\lambda_j)_{\lambda_j<\alpha} \oplus \alpha 1_{z_1}$, where \(1_{z_{1}}\) is the sequence consisting of \(z_{1}\) terms equal to \(1\), and
\item $(d_i-1)_{d_i>\alpha}$ is a diagonal of a compact operator with an eigenvalue list\break $(\lambda_{j}- 1)_{\lambda_j>\alpha}\oplus (\alpha-1)1_{z_2}$.
\end{itemize}

\item[$(\mathfrak d_1)$] holds. If $\sum_{0 < d_i < 1} d_i<\infty$, then 
$$
\sum_{0 <\lambda_j < 1} \lambda_j \le \delta_0+\sum_{0 < d_i < 1} d_i,$$ 
and additionally if $\sum_{\lambda_j<0} |\lambda_j|<\infty$, then we also have
$$
 \delta_0+\sum_{ 0<d_i \le1} d_i \le  \sum_{0<\lambda_j\le 1} \lambda_j.
$$
If
$\sum_{0 < d_i < 1} d_i=\infty$, then $\sum_{0<\lambda_j\le 1} \lambda_j =\infty.
$
\end{enumerate}
\end{thm}

\begin{proof}[Proof of $(\mathfrak d_\alpha)$] This is a consequence of Theorem \ref{nec}, which shows that the operator $E$ decouples at $\alpha$. More precisely, let
\[
\mathcal H_0=\overline{\lspan }\{ e_i: d_i < \alpha \}\qquad\text{and}\qquad \mathcal H_1=\overline{\lspan} \{ e_i: d_i > \alpha \},
\]
Since $d_i=\alpha$ implies that $e_i$ is an eigenvector of $E$ with eigenvalue $\alpha$, the operator $E|_{\mathcal H_0 \oplus \mathcal H_1}$ has eigenvalue list $(\lambda_j)_{\lambda_j\ne \alpha}\oplus \alpha 1_z$, where $z=\#|\{j: \lambda_j = \alpha\}| - \#|\{i: d_i = \alpha\}|$. Let $z_k$, $k=0,1$, be the multiplicity of the eigenvalue $\alpha$ of the operator $E|_{\mathcal H_k}$. Clearly, $z=z_0+z_1$.
Then, $E|_{\mathcal H_0}$ is a compact self-adjoint operator with eigenvalue list $(\lambda_j)_{\lambda_j< \alpha}\oplus \alpha 1_{z_0}$. Likewise, $E|_{\mathcal H_1}$ is a self-adjoint operator with eigenvalue list $(\lambda_j)_{\lambda_j> \alpha}\oplus \alpha 1_{z_1}$. Hence, its essential spectrum contains only the point $1$. This yields the required conclusion.
\end{proof}

\begin{proof}[Proof of $(\mathfrak d_1)$]
We need to consider two possible scenarios. Recall that the spaces $\mathcal H_0=\overline{\lspan }\{ e_i: d_i < 1\}$ and $\mathcal H_1=\overline{\lspan} \{ e_i: d_i \ge 1\}$  are invariant under $E$. 
If $\sum_{0<d_i<1} d_i <\infty$, then by Theorem \ref{cpttrace} the positive part of $E|_{\mathcal H_0}$ is trace class and
\[
 \sum_{0<\lambda_j<1} \lambda_j  \le \tr((E|_{\mathcal H_0})_+) \le \delta^L + \sum_{ 0<d_i<1} d_i.
\]
If we assume additionally that $\sum_{\lambda_j<0}|\lambda_j|<\infty$, then  by Theorem \ref{cpttrace} we have the equality
\[
\delta^L + \sum_{ 0<d_i<1} d_i = \tr((E|_{\mathcal H_0})_+)  .\]
Combing this with the fact that $e_i$ is an eigenvector of $E$ when $d_i = 1$, yields
\[
\delta^L + \sum_{ 0<d_i \le 1} d_i =  \tr((E|_{\mathcal H_0})_+) + \#|\{i: d_i=1\}| \le \sum_{0<\lambda_j\le 1} \lambda_j.
\]

On the other hand, if $\sum_{0<d_i<1} d_i =\infty$, then  $1$ is an accumulation point of the sequence $\{d_i: 0<d_i<1\}$, and hence, the positive part of $E|_{\mathcal H_0}$ is not compact. Therefore,
 \[
  \sum_{0<\lambda_j\le 1} \lambda_j  \ge \tr((E|_{\mathcal H_0})_+) =\infty.
 \]
Conclusion $(\mathfrak d_0)$ is shown by symmetric considerations.
\end{proof}

In addition, we have the following cardinality estimates for an operator as in Theorem \ref{nec}, which are consequences of decoupling:
\begin{enumerate}
\item 
If $(\mathfrak d_0)$ holds, then 
\[
\quad \#|\{j:\lambda_{j}=0\}|\geq \#|\{i:d_{i}=0\}|
\qquad\text{and}\qquad
\#|\{j:\lambda_{j} \le 0\}|\geq \#|\{i:d_{i} \le0\}|.
\]
\item 
If $(\mathfrak d_1)$ holds, then
\[
 \#|\{j:\lambda_{j}=1\}|\geq \#|\{i:d_{i}=1\}|
\qquad\text{and}\qquad
\#|\{j:\lambda_{j} \ge 1\}|\geq \#|\{i:d_{i} \ge1\}|.
\]
\item 
Finally, if both $(\mathfrak d_0)$ and $(\mathfrak d_1)$ hold, then
\[
 \#|\{i:0<d_{i} <1\}|
\ge \#|\{j:0< \lambda_{j} <1\}|
\quad\text{and}\quad 
\#|\{j:0\le \lambda_{j} \le 1\}|  \ge  \#|\{i:0\le d_{i} \le 1\}|. 
\]
\end{enumerate}

Since these additional necessary conditions are not exhaustive, in each of the subsequent sufficiency results we shall assume that decoupling condition $(\mathfrak d_\alpha)$ does not happen.

\section{Equivalence of Riemann and Lebesgue majorization}\label{S16}

In this section we show the equivalence of Riemann and Lebesgue interior majorization for nondecreasing sequences in \([0,1]\) satisfying the Blaschke condition. The concept of \textit{Lebesgue interior majorization} was introduced in \cite[Definition 4.2]{npt} in the case where \(\boldsymbol\lambda\) takes only finitely many values in \([0,1]\). The following definition is a generalization to the case where \(\boldsymbol\lambda\) takes possibly infinitely many values.

\begin{defn}\label{maj} Let $\boldsymbol{d} = (d_{i})_{i\in I}$ and let $\boldsymbol{\lambda} = (\lambda_{i})_{i\in J}$ be sequences in $[0,1]$. We say that $\boldsymbol{d}$ satisfies {\it interior majorization} by $\boldsymbol{\lambda}$ if $f_{\boldsymbol{d}}(\alpha)<\infty$ for all $\alpha\in(0,1)$ and both of the following hold:
\begin{equation}\label{maj1}f_{\boldsymbol{d}}(\alpha)\geq f_{\boldsymbol{\lambda}}(\alpha)\quad\text{for all }\alpha\in(0,1),\end{equation}
and 
\begin{equation}\label{maj2}f_{\boldsymbol{d}}'(\alpha)\equiv f_{\boldsymbol\lambda}'(\alpha)\mod 1\quad\text{for some }\alpha\in(0,1).\end{equation}

\end{defn}

Similarly, \textit{Riemann interior majorization} \cite[Definition 5.2]{npt} was also introduced for the case where \(\boldsymbol\lambda\) takes only finitely many values.

\begin{defn}\label{rim}
Let $\boldsymbol{d}=(d_{i})_{i\in\Z}$ and $\boldsymbol\lambda=(\lambda_{i})_{i\in\Z}$ be nondecreasing seqeunces in $[0,1]$ with $f_{\boldsymbol{d}}(\alpha)<\infty$ and $f_{\boldsymbol\lambda}(\alpha)<\infty$ for all $\alpha\in(0,1)$. We say that $\boldsymbol{d}$ is {\it{Riemann majorized}} by $\boldsymbol\lambda$ if there is some $k\in\Z$ such that the following two hold:
\begin{align}
\label{rim2}\delta_{m}:=\sum_{i=-\infty}^{m}(d_{i-k}-\lambda_{i}) &\geq  0\qquad \text{for all }m\in\Z,
\\
\label{rim3}\lim_{m\to\infty}\delta_{m} & = 0.
\end{align}
\end{defn}

\begin{thm}\label{eqmajs} Let $\boldsymbol{d}=(d_{i})_{i\in\Z}$ and $\boldsymbol\lambda=(\lambda_{i})_{i\in\Z}$ be nondecreasing seqeunces in $[0,1]$ with $f_{\boldsymbol{d}}(\alpha)<\infty$ and $f_{\boldsymbol\lambda}(\alpha)<\infty$ for all $\alpha\in(0,1)$. The sequence $\boldsymbol{d}$ is Riemann majorized by $\boldsymbol\lambda$ if and only if $\boldsymbol{d}$ satisfies interior majorization by $\boldsymbol\lambda$.
\end{thm}

\begin{proof} First, we will assume that \(\boldsymbol\lambda\) is a sequence in \((0,1)\). Let \((A_{j})_{j=-\infty}^{\infty}\) be an increasing sequence in \((0,1)\) consisting of all of the distinct values of the sequence \(\boldsymbol\lambda\). Moreover, set \(N_{j} = \#|\{i\in\Z : \lambda_{i} = A_{j}\}|\).

Without loss of generality, by shifting the indices of \(\boldsymbol d\) and \(\boldsymbol\lambda\), we may assume that \(\lambda_{i}<A_{1}\) \(\iff\) \(i\leq 0\) and  $d_{i}<A_{1}$ $\iff$ $i\leq 0$. For each $r\in\Z$ set
\[m_{r} = \max\{i\in\Z: d_{i}<A_{r}\}\quad\text{and}\quad \sigma_{r} = \max\{i\in\Z:\lambda_{i}<A_{r}\}.\]

With this notation we have
\begin{align}
\label{eqmajsa} A_{r-1}\leq d_i < A_r & \iff  m_{r-1}+1\leq i \le m_r, \\
\label{eqmajsb} \lambda_i = A_{r-1} & \iff \sigma_{r-1} +1 \le i \le \sigma_r.
\end{align}
Moreover, note that the shifts of \(\boldsymbol d\) and \(\boldsymbol\lambda\) we chose imply \(m_{1}=\sigma_{1}=0\). Therefore, we have
\begin{equation}\label{eqmajs3}
C_{\boldsymbol{d}}(A_r) = \sum_{i=-\infty}^{m_r} d_{i},\qquad 
D_{\boldsymbol{d}}(A_r) = \sum_{i=m_{r}+1}^{\infty}(1-d_{i}).
\end{equation}
Since $f_{\boldsymbol{d}}(\alpha)<\infty$ and $f_{\boldsymbol\lambda}(\alpha)<\infty$ for all $\alpha\in(0,1)$, we see that $C_{\boldsymbol{d}}(A_{1})<\infty$, $D_{\boldsymbol{d}}(A_{1})<\infty$, $\sum_{i=-\infty}^{0}\lambda_{i}<\infty$, and $\sum_{i=1}^{\infty}(1-\lambda_{i})<\infty$.

Next, we define
\begin{equation}\label{eqmajs4}k_{0}:=\sum_{i=-\infty}^{\infty}(d_{i}-\lambda_{i}).\end{equation}
From this we have
\begin{equation}\label{eqmajs6}
\begin{aligned}
C_{\boldsymbol{d}}(A_{1}) - D_{\boldsymbol{d}}(A_{1}) & = \sum_{i=-\infty}^{0}d_{i} - \sum_{i=1}^{\infty}(1-d_{i}) = \sum_{i=-\infty}^{0}\lambda_{i} - \sum_{i=1}^{\infty}(1-\lambda_{i}) + k_{0}\\
 & = C_{\boldsymbol\lambda}(A_{1}) - D_{\boldsymbol\lambda}(A_{1}) + k_{0}.
\end{aligned}
\end{equation}
By Proposition \ref{f} the difference \(f_{\boldsymbol{d}}'(\alpha)-f_{\boldsymbol\lambda}'(\alpha)\equiv k_{0}\mod 1\) for all points \(\alpha\in(0,1)\) at which \(f_{\boldsymbol{d}}\) and \(f_{\boldsymbol\lambda}\) are both differentiable. Hence, \eqref{maj2} is equivalent to \(k_{0}\) being a integer.

Next, note that for \(r,k\in\Z\) we have
\[\sum_{i=-\infty}^{r}(d_{i-k}-\lambda_{i}) = \sum_{i=-\infty}^{r-k}(d_{i}-\lambda_{i}) - \sum_{i=r-k+1}^{r}\lambda_{i}.\]
Letting $r\to\infty$ and using the assumption that $\lambda_{i}\to 1$ as $i\to\infty$ we have
\begin{equation}\label{eqmajs5}\sum_{i=-\infty}^{\infty}(d_{i-k}-\lambda_{i}) = \sum_{i=-\infty}^{\infty}(d_{i}-\lambda_{i})-k.\end{equation}

If we assume \eqref{maj2}, then \(k_{0}\in\Z\). From \eqref{eqmajs5} we deduce that \eqref{rim3} holds with \(k=k_{0}\). Conversely, if  \eqref{rim3} holds for some \(k\in\Z\), then \eqref{eqmajs5} implies \(k_{0}=k\in\N\), and hence \eqref{maj2} holds. This proves the equivalence of \eqref{maj2} and \eqref{rim3}. For the rest of the proof we will assume that both \eqref{maj2} and \eqref{rim3} hold. 

Next, we claim that the inequality \eqref{maj1} with \(\alpha=A_{r}\) is equivalent to
\begin{equation}\label{eqmajs2}\sum_{i=-\infty}^{m_{r}}d_{i}\geq \sum_{j=-\infty}^{\sigma_{r}}\lambda_{i} + A_{r}(k_{0}-\sigma_{r}+m_{r}).\end{equation}

From \eqref{f1}, \eqref{eqmajsa}, and \eqref{eqmajsb} we deduce that for all \(r\in\Z\)
\begin{equation}\label{eqmajs7}C_{\boldsymbol\lambda}(A_{r}) - D_{\boldsymbol\lambda}(A_{r}) =  C_{\boldsymbol\lambda}(A_{1})-D_{\boldsymbol\lambda}(A_{1})  + \sigma_{r}.\end{equation}
and
\begin{equation}\label{eqmajs8}C_{\boldsymbol{d}}(A_{r}) - D_{\boldsymbol{d}}(A_{r}) =  C_{\boldsymbol{d}}(A_{1})-D_{\boldsymbol{d}}(A_{1}) + m_{r}.\end{equation}

Using \eqref{eqmajs6}, \eqref{eqmajs7}, and \eqref{eqmajs8} we have
\begin{equation}\label{eqmajs1}C_{\boldsymbol{d}}(A_{r})-D_{\boldsymbol{d}}(A_{r}) = C_{\boldsymbol\lambda}(A_{r}) - D_{\boldsymbol\lambda}(A_{r})  + k_{0}-\sigma_{r}+m_{r}.\end{equation}
Using the identity \eqref{eqmajs1}, we can rewrite \eqref{maj1} with \(\alpha=A_{r}\), as the equivalent inequality
\[C_{\boldsymbol{d}}(A_{r})\geq C_{\boldsymbol\lambda}(A_{r}) + A_{r}(k_{0}-\sigma_{r}+m_{r}).\]
This is exactly \eqref{eqmajs2}, and hence if \eqref{maj1} holds, then \eqref{eqmajs2} holds for all \(r\in\Z\). Conversely, by Proposition \ref{9.2}, if \eqref{eqmajs2} holds for all \(r\in\Z\), then \eqref{maj1} holds.


To complete the proof we will show that \eqref{eqmajs2} holds for all \(r\in\Z\) if and only if \eqref{rim2} holds.

Assume \eqref{eqmajs2} holds. We will show that for each $r\in\Z$ we have $\delta_{m}\geq 0$ for $m=\sigma_{r}+1,\ldots,\sigma_{r+1}$. There are three cases to consider:

{\bf Case 1.} Assume that $\sigma_{r}+1\leq m_{r}+k_{0}\leq \sigma_{r+1}$. First we will show that $\delta_{m_{r}+k_{0}}\geq 0$. Using \eqref{eqmajs2} and then \eqref{eqmajsb}
\begin{align*}
\delta_{m_{r}+k_{0}} & = \sum_{i=-\infty}^{m_{r}+k_{0}}(d_{i-k_{0}}-\lambda_{i}) = \sum_{i=-\infty}^{m_{r}}d_{i} - \sum_{i=-\infty}^{m_{r}+k_{0}}\lambda_{i}
 \geq \sum_{i=-\infty}^{\sigma_{r}}\lambda_{i} + A_{r}(k_{0}-\sigma_{r}+m_{r}) - \sum_{i=-\infty}^{m_{r}+k_{0}}\lambda_{i}\\
 & = \sum_{i=-\infty}^{\sigma_{r+1}}\lambda_{i} + A_{r}(k_{0}-\sigma_{r+1}+m_{r}) - \sum_{i=-\infty}^{m_{r}+k_{0}}\lambda_{i} = A_{r}(k_{0}-\sigma_{r+1}+m_{r}) + \sum_{i=m_{r}+k_{0}+1}^{\sigma_{r+1}}A_{r}  = 0.
\end{align*}
By \eqref{eqmajsa} and \eqref{eqmajsb}
\[
\begin{aligned}
d_{i-k_{0}} - \lambda_i \ge A_r - \lambda_i  \ge 0 & \qquad\text{for } m_{r}+k_{0}\leq i\leq \sigma_{r+1},
\\
d_{i-k_{0}} - \lambda_i < A_r - \lambda_i =0 & \qquad\text{for } \sigma_{r}+1\leq i\leq m_{r}+k_{0}.
\end{aligned}
\]
Combining this with $\delta_{m_{r}+k_{0}}\geq 0$ implies that $\delta_{m}\geq 0$ for all $m = \sigma_{r}+1,\ldots,\sigma_{r+1}$.

{\bf Case 2.} Assume $m_{r}+k_{0}>\sigma_{r+1}$. Using \eqref{eqmajs2} and then \eqref{eqmajsa}
\begin{align*}
\delta_{\sigma_{r+1}} & = \sum_{i=-\infty}^{\sigma_{r+1}}(d_{i-k_{0}}-\lambda_{i})
  = \sum_{i=-\infty}^{\sigma_{r+1}-k_{0}}d_{i} - \sum_{i=-\infty}^{\sigma_{r}} \lambda_{i} - (\sigma_{r+1}-\sigma_{r})A_{r}\\
	& \geq   \sum_{i=-\infty}^{\sigma_{r+1}-k_{0}}d_{i} + A_{r}(k_{0}-\sigma_{r+1}+m_{r}) - \sum_{i=-\infty}^{m_{r}}d_{i} = - \sum_{i=\sigma_{r+1}-k_{0}+1}^{m_{r}}d_{i} + A_{r}(k_{0}-\sigma_{r+1}+m_{r})\\
	& > -\sum_{i=\sigma_{r+1}-k_{0}+1}^{m_{r}}A_{r} + A_{r}(k_{0}-\sigma_{r+1}+m_{r})= 0.
\end{align*}
For all $i$ such that $\sigma_{r}+1 \le i \le \sigma_{r+1}$, \eqref{eqmajsa} and \eqref{eqmajsb} imply $d_{i-k_{0}}-\lambda_{i} < A_r - \lambda_i=0$.
Combining this with $\delta_{\sigma_r+1}\geq 0$ implies that $\delta_{m}\geq 0$ for all $m = \sigma_{r}+1,\ldots,\sigma_{r+1}$.

{\bf Case 3.} Assume $m_{r}+k_{0}\leq \sigma_{r}$. First, we calculate
\begin{align*}
\delta_{\sigma_{r}} & = \sum_{i=-\infty}^{\sigma_{r}}(d_{i-k_{0}}-\lambda_{i}) = \sum_{i=-\infty}^{m_{r}}d_{i} + \sum_{i=m_{r}+1}^{\sigma_{r}-k_{0}}d_{i} - \sum_{i=-\infty}^{\sigma_{r}}\lambda_{i}\\
 & \geq A_{r}(k_{0} - \sigma_{r}+m_{r}) + \sum_{i=m_{r}+1}^{\sigma_{r}-k_{0}}d_{i} \geq  A_{r}(k_{0} - \sigma_{r}+m_{r}) + \sum_{i=m_{r}+1}^{\sigma_{r}-k_{0}}A_{r} = 0.
\end{align*}
For $i$ such that $\sigma_{r}+1\leq i\leq \sigma_{r+1}$, using \eqref{eqmajsa} and \eqref{eqmajsb} we have $d_{i-k_{0}}-\lambda_{i}>A_{r}-\lambda_{i} = 0$. Since we have already demonstrated that $\delta_{\sigma_{r}}\geq 0$ this shows that $\delta_{m}\geq 0$ for $m=\sigma_{r}+1,\ldots,\sigma_{r+1}$.

Finally, assume \eqref{rim2} holds for some \(k\in\Z\). Fix $r\in\Z$ and assume $\sigma_{r}\geq m_{r}+k$. Using $\delta_{m_{r}+k}\geq 0$ and the fact that $\lambda_i \le A_{r-1}<A_{r}$ for $i \leq \sigma_{r}$, we have
\begin{equation*}
\sum_{i=-\infty}^{m_{r}}d_{i} = \sum_{i=-\infty}^{m_{r}+k}d_{i-k} \geq \sum_{i=-\infty}^{m_{r}+k}\lambda_{i} = \sum_{i=-\infty}^{\sigma_{r}}\lambda_{i} - \sum_{i=m_{r}+k+1}^{\sigma_{r}}\lambda_{i} \geq \sum_{i=-\infty}^{\sigma_{r}}\lambda_{i} - (\sigma_{r}-m_{r}-k)A_{r}.
\end{equation*}
Next assume $\sigma_{r}<m_{r}+k$. Using $\delta_{m_{r}+k}\geq 0$ and the fact $\lambda_i \ge A_r$ for $i \ge \sigma_{r}+1$, we have
\[\sum_{i=-\infty}^{m_{r}}d_{i} = \sum_{i=-\infty}^{m_{r}+k}d_{i-k} \geq \sum_{i=-\infty}^{m_{r}+k}\lambda_{i} = \sum_{i=-\infty}^{\sigma_{r}}\lambda_{i} + \sum_{i=\sigma_{r}+1}^{m_{r}+k}\lambda_{i} \geq \sum_{i=-\infty}^{\sigma_{r}}\lambda_{i} + (m_{r}+k-\sigma_{r})A_{r}.\]
This proves that $(d_{i})$ satisfies \eqref{eqmajs2} for each \(r\in\Z\).

To complete the proof we must consider the cases where \(0\) and \(1\) are terms of \(\boldsymbol\lambda\). If both \(0\) and \(1\) are terms of \(\boldsymbol\lambda\), then \(\boldsymbol\lambda\) takes only finitely many values, and the result is a reformulation of Theorem 5.2 from \cite{npt}. If \(0\) is a term of \(\boldsymbol\lambda\) and \(1\) is not, then we define the sequence \((A_{j})\), but indexed by \(\N\cup\{0\}\) and with \(A_{0}=0\), and use the same proof as above by restricting \(r\) to \(\N\). For the final case where \(1\) is a term of \(\boldsymbol\lambda\) and \(0\) follows from a symmetric proof.
\end{proof}


\section{Sufficiency of interior majorization}\label{S17}

In this section we will show that interior majorization is sufficient for the diagonal-to-diagonal result for sequences in \((0,1)\) satisfying the Blaschke condition. To achieve this we will show that Riemann interior majorization is sufficient for nondecreasing sequences.

\begin{thm}\label{rimsuff} Let $(\lambda_{i})_{i\in\Z}$ and $(d_{i})_{i\in\Z}$ be two nondecreasing sequences in $[0,1]$ such that $(d_{i})_{i=-\infty}^{0}$, $(\lambda_{i})_{i=-\infty}^{0}$, $(1-d_{i})_{i=1}^{\infty}$, and $(1-\lambda_{i})_{i=1}^{\infty}$ are summable.
Define
\[\delta_{m} = \sum_{i=-\infty}^{m}(d_{i}-\lambda_{i})\quad\text{for }m\in\Z.\]
Suppose that
\begin{equation}\label{rmaj1}\delta_{m}\geq 0\quad\text{for all }m\in\Z\end{equation}
and
\begin{equation}\label{rmaj2}\lim_{m\to\infty}\delta_{m} = 0.\end{equation}
If $(\lambda_{i})$ is a diagonal of a self-adjoint operator $E$, then so is $(d_{i})$.
\end{thm}

\begin{proof} First, suppose there exists \(m_{0}\in\Z\) such that \(\delta_{m_{0}}=0\). Set \(I_{1} = J_{1} = \Z\cap(-\infty,m_{0}]\) and \(I_{2} = J_{1} = \Z\cap (m_{0},\infty)\). For \(k=1,2\), let \(E_{k}\) be a self-adjoint operator with diagonal \((\lambda_{i})_{i\in I_{k}}\). The sequences \((\lambda_{m_{0}+1-i})_{i=1}^{\infty}\) and \((d_{m_{0}+1-i})_{i=1}^{\infty}\) are nonincreasing, and for each \(n\in \N\) we have
\[\delta_{n}^{(1)}:=\sum_{i=1}^{n}(\lambda_{m_{0}+1-i}-d_{m_{0}+1-i}) = \delta_{m_{0}+1-n}-\delta_{m_{0}} = \delta_{m_{0}+1-n}\geq 0.\]
It is also clear that \(\delta_{n}^{(1)}\to0\) as \(n\to\infty\). By Proposition \ref{posSH}, \((d_{i})_{i\in I_{1}}\) is also a diagonal of \(E_{1}\). Similarly, the sequences \((\lambda_{m_{0}+i})_{i=1}^{\infty}\) and \((d_{m_{0}+i})_{i=1}^{\infty}\) are nondecreasing, and for each \(n\in\N\)
\[\delta_{n}^{(2)} = \sum_{i=1}^{n}(d_{m_{0}+i}-\lambda_{m_{0}+i}) = \delta_{m_{0}+n}-\delta_{m_{0}} = \delta_{m_{0}+n}\geq 0.\]
Again, it is clear that \(\delta_{n}^{(2)}\to 0\) as \(n\to\infty\). By a symmetric version of Proposition \ref{posSH} (\cite[Theorem 3.6]{unbound}) \((d_{i})_{i\in I_{2}}\) is also a diagonal of \(E_{2}\). Since \(E_{1}\) and \(E_{2}\) were arbitrary, by Lemma \ref{subdiag} if \(E\) is any self-adjoint operator with diagonal \((\lambda_{i})_{i\in\Z}\), then \((d_{i})_{i\in\Z}\) is also a diagonal of \(E\).

We may now assume that \(\delta_{m}>0\) for all \(m\in\Z\). Let \(E\) be a self-adjoint operator with diagonal \((\lambda_{i})_{i\in\Z}\). Fix \(i_{0}\in\Z\) such that \(\lambda_{i_{0}-1}<\lambda_{i_{0}}\). Since \(\delta_{m}\to 0\) as \(m\to-\infty\), we can choose \(k_{0}<i_{0}\) such that \(\delta_{k_{0}}\leq \delta_{m}\) for \(m=k_{0}+1,k_{0}+2,\ldots,i_{0}\). Additionally, we may choose \(k_{0}\) small enough that 
\begin{equation}\label{rmaj3}d_{k_{0}}+\delta_{k_{0}}<\lambda_{i_{0}}-\lambda_{i_{0}-1}.\end{equation}
Define the sequence \((\tilde{\lambda}_{i})_{i\in\Z}\) by
\[\tilde{\lambda}_{i} = \begin{cases} d_{i} & i\leq k_{0},\\ \lambda_{i} & i\in\{k_{0}+1,\ldots,i_{0}-1\}\cup\{i_{0}+1,i_{0}+2,\ldots\},\\ \lambda_{i_{0}} - \delta_{k_{0}} & i=i_{0}.\end{cases}\]

Set \(I_{1} = J_{1} = \{i_{0},k_{0},k_{0}-1,k_{0}-2,\ldots\}\), and \(I_{2} = J_{2} = \Z\setminus I_{1}\). For \(k=1,2\) let \(E_{k}\) be a self-adjoint operator with diagonal \((\lambda_{i})_{i\in I_{k}}\). It is clear that \((\tilde{\lambda}_{i})_{i\in I_{2}} = (\lambda_{i})_{i\in I_{2}}\) is a diagonal of \(E_{2}\). Define the sequences
\[\boldsymbol{\mu} = (\lambda_{i_{0}},\lambda_{k_{0}},\lambda_{k_{0}-1},\ldots)\quad\text{and}\quad \tilde{\boldsymbol\mu} = (\lambda_{i_{0}}-\delta_{k_{0}},d_{k_{0}},d_{k_{0}-1},\ldots).\]
By the assumption that \(d_{k_{0}}+\delta_{k_{0}}<\lambda_{i_{0}}\) we see that both \(\boldsymbol\mu\) and \(\tilde{\boldsymbol\mu}\) are nonincreasing. Moreover, for each \(n\in \N\) we have
\[\delta_{n}^{\mu}:=\sum_{i=1}^{n}(\mu_{i} - \tilde{\mu}_{i}) = \delta_{k_{0}+1-n}\geq 0,\]
and clearly \(\tilde{\delta}_{n}\to 0\) as \(n\to\infty\). Note that \(\boldsymbol\mu\) is a diagonal of \(E_{1}\). By Proposition \ref{posSH}, \(\tilde{\boldsymbol\mu}\) is also a diagonal of \(E_{1}\). Since \(E_{1}\) and \(E_{2}\) were arbitrary, by Lemma \ref{subdiag}  \((\tilde{\lambda}_{i})_{i\in\Z}\) is also a diagonal of \(E\).

For the final application of Lemma \ref{subdiag} we define the sets \(I_{1}=J_{1} = \{k_{0}+1,k_{0}+2,\ldots\}\) and \(I_{2}=J_{2} = \Z\setminus I_{1}\). For \(k=1,2\) let \(E_{k}\) be a self-adjoint operator with diagonal \((\tilde{\lambda}_{i})_{i\in I_{k}}\). It is clear that \((d_{i})_{i\in I_{2}}\) is a diagonal of \(E_{2}\). Define the sequences
\[\boldsymbol{\eta} = (\lambda_{k_{0}+1},\lambda_{k_{0}+2},\ldots,\lambda_{i_{0}-1},\lambda_{i_{0}}-\delta_{k_{0}},\lambda_{i_{0}+1},\ldots)\quad\text{and}\quad \tilde{\boldsymbol\eta} = (d_{k_{0}+1},d_{k_{0}+2},\ldots).\]
From \eqref{rmaj3} we see that \(\boldsymbol\eta\) is nondecreasing. For each \(n\in\N\) we have
\[\delta_{n}^{\eta} := \sum_{i=1}^{n}(\tilde{\eta}_{i} - \eta_{i}) = \begin{cases} \delta_{k_{0}+n}-\delta_{k_{0}} & n=1,\ldots,i_{0}-1,\\ \delta_{k_{0}+n} & n\geq i_{0}.\end{cases}\]
By our choice of \(k_{0}\) we see that \(\delta_{n}^{\eta}\geq 0\) for all \(n\in\N\). Moreover, \(\delta_{n}^{\eta} \to0\) as \(n\to\infty\). By a symmetric version of Proposition \ref{posSH} we conclude that \(\tilde{\boldsymbol\eta}\) is a diagonal of \(E_{1}\), and thus \((d_{i})_{i\in I_{1}}\) is a diagonal of \(E_{1}\). Since \(E_{1}\) and \(E_{2}\) were arbitrary, by Lemma \ref{subdiag}  \((d_{i})_{i\in\Z}\) is also a diagonal of \(E\).
\end{proof}

\begin{thm}\label{archaic}
Let $\boldsymbol \lambda=(\lambda_{i})_{i\in\N}$ and $\boldsymbol d=(d_{i})_{i\in\N}$ be two sequences in $(0,1)$ which have accumulation points at $0$ and $1$.
Assume that
\[
f_{\boldsymbol \lambda}(\alpha)  \le f_{\boldsymbol d}(\alpha)< \infty
\qquad\text{for all }\alpha \in (0,1),
\]
and
\[
f_{\boldsymbol\lambda}'(\alpha) \equiv f_{\boldsymbol{d}}'(\alpha) \mod 1\qquad \text{for some }\alpha\in(0,1).
\]
If $\boldsymbol \lambda$ is a diagonal of a self-adjoint operator $E$, then so is $\boldsymbol d$.
\end{thm}

\begin{proof} The assumptions that \(f_{\boldsymbol{d}}(\alpha)\) and \(f_{\boldsymbol\lambda}(\alpha)\) are finite imply that we can reindex \(\boldsymbol d\) and \(\boldsymbol\lambda\) by \(\Z\) in nondecreasing order. Theorem \ref{eqmajs} implies that \(\boldsymbol d\) is Riemann majorized by \(\boldsymbol\lambda\). Finally, Theorem \ref{rimsuff} shows that if \(\boldsymbol d\) is a diagonal of some self-adjoint operator, then \(\boldsymbol\lambda\) is a diagonal as well.
\end{proof}

\begin{thm}\label{archaic2}
Let $\boldsymbol \lambda=(\lambda_{i})_{i\in I}$ and $\boldsymbol d=(d_{i})_{i\in J}$ be two sequences in $(0,1)$.
Assume that
\[
f_{\boldsymbol \lambda}(\alpha)  \le f_{\boldsymbol d}(\alpha)< \infty
\qquad\text{for all }\alpha \in (0,1),
\]
and
\begin{equation}\label{archaic2.0}
f_{\boldsymbol\lambda}'(\alpha) \equiv f_{\boldsymbol{d}}'(\alpha) \mod 1 \qquad \text{for some }\alpha\in(0,1).
\end{equation}
There exists a sequence \(\tilde{\boldsymbol\lambda}\) obtained by appending some number of zeros and ones to \(\boldsymbol\lambda\) such that if $\tilde{\boldsymbol \lambda}$ is a diagonal of a self-adjoint operator $E$, then so is $\boldsymbol d$.
\end{thm}

\begin{proof} First we will show that \(\#|I|\leq \#|J|\). Suppose \(\#|J|<\infty\). Choose \(\alpha_{0}\in(0,1)\) such that \(\alpha_{0}<d_{i}\) for all \(i\in J\). On \((0,\alpha_{0})\) the function \(f_{\boldsymbol{d}}\) is linear and has slope \(D_{\boldsymbol{d}}(\alpha_{0}) = \sum_{i\in J}(1-d_{i})\). For \(\alpha\in(0,\alpha_{0})\) we have
\[\sum_{i\in J}(1-d_{i}) = \frac{f_{\boldsymbol{d}}(\alpha)}{\alpha} \geq \frac{f_{\boldsymbol\lambda}(\alpha)}{\alpha} = \frac{1-\alpha}{\alpha}C_{\boldsymbol\lambda}(\alpha) + D_{\boldsymbol\lambda}(\alpha)\geq D_{\boldsymbol\lambda}(\alpha) =  \sum_{\lambda_{i}\geq\alpha}(1-\lambda_{i}).\]
Letting \(\alpha\to 0\) we have
\[\sum_{i\in J}(1-d_{i})\geq \sum_{i\in I}(1-\lambda_{i}).\]
A symmetric argument shows
\begin{equation}\label{archaic2.1}\sum_{i\in J}d_{i}\geq \sum_{i\in I}\lambda_{i}.\end{equation}
From these we deduce that \(\#|I|<\infty\), and adding these inequalities shows \(\#|J|\geq \#|I|\).


\textbf{Case 1.} Suppose both \(0\) and \(1\) are accumulation points of \(\boldsymbol{d}\). Define \(\smash{\tilde{\boldsymbol\lambda}}\) by appending infinitely many zeros to \(\boldsymbol\lambda\) if \(0\) is not an accumulation point of \(\boldsymbol\lambda\), and appending infinitely many ones to \(\boldsymbol\lambda\) if \(1\) is not an accumulation point of \(\boldsymbol\lambda\). Thus, the sequence \(\tilde{\boldsymbol\lambda}\) can be reindexed by \(\Z\) in nondecreasing order, \(\tilde{\boldsymbol\lambda} = (\lambda_{i})_{i\in\Z}\). Similarly, since both \(0\) and \(1\) are accumulation points of \(\boldsymbol{d}\) and \(C_{\boldsymbol{d}}(\alpha),D_{\boldsymbol{d}}(\alpha)<\infty\) for \(\alpha\in(0,1)\), we can reindex \(\boldsymbol{d}\) by \(\Z\) in nondecreasing order, \(\boldsymbol{d}=(d_{i})_{i\in\Z}\). By Theorem \ref{eqmajs} we see that \(\boldsymbol{d}\) is Riemann majorized by \(\tilde{\boldsymbol\lambda}\), and Theorem \ref{rimsuff} gives the desired conclusion.\medskip

\textbf{Case 2.} Suppose \(1\) is the not an accumulation point of \(\boldsymbol{d}\). From \eqref{archaic2.1} we deduce that \(1\) is also not an accumulation point of \(\boldsymbol\lambda\).

 Choose \(\alpha_{0}\in(0,1)\) such that \(\alpha_{0}>d_{i}\) for all \(i\in J\) and \(\alpha_{0}>\lambda_{i}\) for all \(i\in I\). By Proposition \ref{f} we have
\[f_{\boldsymbol{d}}'(\alpha_{0}) = -C_{\boldsymbol{d}}(\alpha_{0}) = -\sum_{i\in J}d_{i}\quad\text{and}\quad f_{\boldsymbol\lambda}'(\alpha_{0}) = -C_{\boldsymbol\lambda}(\alpha_{0}) = -\sum_{i\in J}\lambda_{i}.\]
From \eqref{archaic2.0} and \eqref{archaic2.1} we see that there is some \(k\in\N\cup\{0\}\) such that
\begin{equation}\label{archaic2.2}\sum_{i\in J}d_{i} = k+\sum_{i\in I}\lambda_{i}.\end{equation}

Applying Lemma \ref{intToComp} to the sequences \(\boldsymbol{\lambda}\) and \(\boldsymbol{d}\) we deduce from \eqref{archaic2.2} that 
\(\kappa = -k\) and \(\beta=k\) (Note that since the sequences are in \((0,1)\) we have \(\delta^{U}=\delta^{L}=\delta_{0}=\delta_{1} = 0\)).
Hence, 
\[\liminf_{\alpha\searrow 0}\delta(\alpha,\boldsymbol{\lambda},\boldsymbol{d}) = -k
\quad
\text{and}
\quad
\delta(\alpha,\boldsymbol{\lambda},\boldsymbol{d})\geq -k(1-\alpha)\quad\text{for all }\alpha\neq 0.\]

Define \(\smash{\tilde{\boldsymbol\lambda}}\) by appending \(k\) ones, and appending and \(m\in\N\cup\{0\}\cup\{\infty\}\) zeros to \(\boldsymbol\lambda\), where \(m\) will be chosen later. Now, we note that for \(\alpha<1\) we have
\[\delta(\alpha,\boldsymbol{\tilde\lambda},\boldsymbol{d}) = \delta(\alpha,\boldsymbol{\lambda},\boldsymbol{d}) + k(1-\alpha).\]
Hence
\[\liminf_{\alpha\searrow 0}\delta(\alpha,\boldsymbol{\tilde\lambda},\boldsymbol{d}) = 0
\quad
\text{and}
\quad
\delta(\alpha,\boldsymbol{\tilde\lambda},\boldsymbol{d})\geq 0\quad\text{for all }\alpha\neq 0.\]
Proposition \ref{LR} (or its finite variant if \(\boldsymbol{d}\) is a finite sequence) implies
\[\sum_{i=1}^{k}d_{i}^{\downarrow}\leq \sum_{i=1}^{k}\tilde{\lambda}_{i}^{\downarrow}\quad\text{for all }k.\]
and from \eqref{archaic2.2} we see that these sequences have the same (finite) sum. Together these imply that the number of positive terms of \(\boldsymbol{d}\) is at least the number of positive terms in \(\boldsymbol\lambda\). Choose \(m\) so that the sequences have the same number of terms. Finally, by applying Proposition \ref{posSH} (or the Schur-Horn Theorem if \(\boldsymbol{d}\) is a finite sequence) we see that if \(\boldsymbol{\tilde\lambda}\) is a diagonal of some self-adjoint operator \(E\), then \(\boldsymbol{d}\) is as well.

\textbf{Case 3.} Suppose \(0\) is not an accumulation point of \(\boldsymbol{d}\). Apply Case 2 to the sequences \((1-\lambda_{i})\) and \((1-d_{i})\), and we get a sequence \((1-\tilde{\lambda}_{i})\) which is obtained by appending zeros and ones to \((1-\lambda_{i})\). Set \(\tilde{\boldsymbol\lambda}:=(\tilde{\lambda}_{i})\), and note that this sequence is obtained by appending zeros and ones to \(\boldsymbol\lambda\). If \(E\) is any self-adjoint operator with diagonal  \(\boldsymbol{\tilde\lambda}\), then \(\mathbf{I}-E\) has diagonal \((1-\tilde{\lambda}_{i})\). Hence \((1-d_{i})\) is also a diagonal of \(\mathbf{I}-E\), and thus \(\boldsymbol{d}\) is a diagonal of \(E\).

\end{proof}

\section{Blaschke condition diagonal-to-diagonal result}\label{S18}

In this section we prove the diagonal-to-diagonal result for sequences having accumulation points only at \(0\) and \(1\), and satisfying the Blaschke condition. We start by proving a preliminary result under restrictive conditions on sequences 
$\boldsymbol\lambda$ and $\boldsymbol d$.

\begin{lem}\label{lsuff}
Let $\boldsymbol\lambda=(\lambda_{i})_{i\in \N}$ and  $\boldsymbol{d} = (d_{i})_{i\in \N}$ be two sequences, which have  accumulation points only at $0$ and $1$, and such that
 $f_{\boldsymbol \lambda}(1/2) <\infty$, $f_{\boldsymbol d}(1/2) <\infty$. Assume that the sequence $\boldsymbol \lambda$ has infinitely many terms $0<\lambda_i<1/2$ and infinitely many terms $1/2<\lambda_i<1$.
Define
\[\delta^{L} = \liminf_{\beta\nearrow 0}(f_{\boldsymbol\lambda}(\beta)-f_{\boldsymbol{d}}(\beta))\quad\text{and}\quad \delta^{U} = \liminf_{\beta\searrow 1}(f_{\boldsymbol\lambda}(\beta)-f_{\boldsymbol{d}}(\beta)).\]
Assume that $\delta^L+\delta^U<\infty$,
\begin{equation*}
 f_{\boldsymbol{d}}(\alpha)  \leq f_{\boldsymbol\lambda}(\alpha)  
\qquad\text{for all }\alpha\in\R\setminus[0,1],
\end{equation*}
and there exist $\delta_0 \in [0,\delta^L]$ and $\delta_1 \in [0,\delta^U]$ such that 
\begin{align*}
f_{\boldsymbol\lambda}(\alpha)  &\le f_{\boldsymbol{d}}(\alpha) + (1-\alpha)\delta_{0} + \alpha\delta_{1} & \text{for all }\alpha\in(0,1),\\
f_{\boldsymbol\lambda}'(\alpha)-f_{\boldsymbol{d}}'(\alpha)  & \equiv \delta_1-\delta_0 \mod 1 & \text{for some }\alpha\in(0,1).
\end{align*}
In addition, assume one of the following holds:
\begin{enumerate}[(i)]
\item \(\lambda_{i},d_{i}>0\) for all \(i\in\N\), and \(\delta^{U}=\delta_{1}>0\),
\item \(\lambda_{i},d_{i}<1\) for all \(i\in\N\), and \(\delta^{L}=\delta_{0}>0\),
\item  $\delta^L>0$, $\delta^U>0$, and $\delta^L-\delta_0=\delta^U - \delta_1$.
\end{enumerate}
If $\boldsymbol \lambda $ is a diagonal of a self-adjoint operator $E$, then so is $\boldsymbol d$.
\end{lem}


\begin{proof}
We present a parallel proof that works simultaneously for each of the mutually exclusive assumptions $(i)$--$(iii)$. However, $(iii)$ has the most interesting proof since the first two cases are simplified modifications of this argument.

Suppose momentarily that $\delta^L>0$.
Let \(0<\alpha_{0}<1/2\) be small enough that 
\[\sum_{\lambda_{i}<\alpha_{0}}\lambda_{i}<\frac{\delta^{L}}{2}\quad\text{and}\quad\N\cap\left[\frac{1}{1-\alpha_{0}}\sum_{\lambda_{i}\geq\alpha_{0}}(1-\lambda_{i}),\frac{\delta^{L}}{2\alpha_{0}}\right]\neq \varnothing.\]
The first inequality clearly holds for all \(\alpha_{0}>0\) sufficiently small. For the second condition we note that the ratio of the endpoints of the given interval is
\[\frac{\delta^{L}(1-\alpha_{0})}{2\alpha_{0}\sum_{\lambda_{i}\geq\alpha_{0}}(1-\lambda_{i})}.\]
By Proposition \ref{f} the above ratio goes to \(\infty\) as \(\alpha_{0}\searrow 0\). Since the left endpoint of the interval is also going to \(\infty\) as \(\alpha\searrow0\) we see that the second condition is also satisfied for all sufficiently small \(\alpha_{0}\). Hence, we let \(\alpha_{0}\) be sufficiently small, and let \(N\in\N\) such that 
\[\frac{1}{1-\alpha_{0}}\sum_{\lambda_{i}\geq\alpha_{0}}(1-\lambda_{i})\leq N\leq \frac{\delta^{L}}{2\alpha_{0}}.\]
Let \((\tilde{d}_{i})_{i\in I_{0}}\) be an infinte positive sequence in \((0,\alpha_{0}]\) containing all terms of \(\boldsymbol\lambda\) in \((0,\alpha_{0})\), \(N\) copies of \(\alpha_{0}\), and a finite sequence of additional terms such that 
\[
\sum_{i\in I_0} \tilde d_i = \delta^L.
\]
In the case when $\delta^L=0$, the above construction is mute and we let $\alpha_0=0$ and $I_0 = \varnothing$.

Likewise, suppose momentarily that $\delta^U>0$. By a symmetric argument we can find \(\alpha_{1}\in(1/2,1)\) sufficiently close to \(1\) and \(M\in\N\) such that
\[\sum_{\lambda_{i}\geq \alpha_{1}}(1-\lambda_{i})<\frac{1}{2}\delta^{U}\quad\text{and}\quad \frac{1}{\alpha_{1}}\sum_{\lambda_{i}<\alpha_{1}}\lambda_{i}\leq M\leq \frac{\delta^{U}}{2(1-\alpha_{1})}.\]
Let \((\tilde{d}_{i})_{i\in I_{1}}\) be an infinte positive sequence in \([\alpha_{1},1)\) containing all terms of \(\boldsymbol\lambda\) in \([\alpha_{1},1)\), \(M\) copies of \(\alpha_{1}\), and a finite sequence of additional terms such that 
\[
\sum_{i\in I_1} (1-\tilde{d}_i) = \delta^{U}.
\]
In the case when $\delta^U=0$, we let $\alpha_1=0$ and $I_1 = \varnothing$.

Let $\tilde {\boldsymbol d}= (d_{i})_{ 0<d_i < 1 } \oplus (\tilde{d}_{i})_{ i\in I_0 }\oplus (\tilde{d}_{i})_{ i\in I_1} $. Each of the assumptions $(i)$--$(iii)$ implies that $I_0 \cup I_1$ is infinite, and hence $\tilde {\boldsymbol d}$ is an infinite sequence.
For \(\alpha\in(\alpha_{0},\alpha_{1})\) we have
\[f_{\tilde{\boldsymbol{d}}}(\alpha) = f_{\boldsymbol{d}}(\alpha) + (1-\alpha)\delta^{L}+\alpha\delta^{U} \geq f_{\boldsymbol{d}}(\alpha) + (1-\alpha)\delta_{0}+\alpha\delta_{1}\geq f_{\boldsymbol{\lambda}}(\alpha).\]
If $\alpha_0>0$, then for \(\alpha\in(0,\alpha_{0}]\) we have
\begin{align*}
f_{\boldsymbol\lambda}(\alpha) & = (1-\alpha)\sum_{\lambda_{i}<\alpha}\lambda_{i} + \alpha \sum_{\alpha\leq \lambda_{i}< \alpha_{0}}(1-\lambda_{i}) + \alpha\sum_{\lambda_{i}\geq\alpha_{0}}(1-\lambda_{i})\\
 & \leq (1-\alpha)\sum_{\tilde{d}_{i}<\alpha}\tilde{d}_{i} + \alpha \sum_{\alpha\leq \tilde{d}_{i}< \alpha_{0}}(1-\tilde{d}_{i}) + \alpha(1-\alpha_{0})N\\
 & \leq (1-\alpha)\sum_{\tilde{d}_{i}<\alpha}\tilde{d}_{i} + \alpha \sum_{\alpha\leq \tilde{d}_{i}< \alpha_{0}}(1-\tilde{d}_{i}) + \alpha\sum_{\tilde{d}_{i}\geq\alpha_{0}}(1-\tilde{d}_{i}) = f_{\tilde{\boldsymbol{d}}}(\alpha)
\end{align*}
By a symmetric argument we will also have \(f_{\boldsymbol\lambda}(\alpha)\leq f_{\tilde{\boldsymbol{d}}}(\alpha)\) for all \(\alpha\in[\alpha_{1},1)\) if $\alpha_1<1$. Combining these estimates yields
\[
f_{\boldsymbol\lambda}(\alpha)  \le f_{\tilde {\boldsymbol{d}}}(\alpha) \qquad \text{for all }\alpha\in(0,1).
\]
Moreover, we calculate
\[
f_{\boldsymbol\lambda}'(\alpha)-f_{\tilde {\boldsymbol{d}}}'(\alpha) \equiv f_{\boldsymbol\lambda}'(\alpha)-f_{\boldsymbol{d}}'(\alpha)   + \delta^L- \delta^U \equiv \delta_1-\delta_0 + \delta^L- \delta^U = 0 \mod 1 \qquad \text{for a.e.\ }\alpha\in(0,1).
\]
By Theorem \ref{archaic}, if a self-adjoint operator has diagonal $(\lambda_{j})_{0<\lambda_j <1}$, then $\tilde {\boldsymbol d}$ is also its diagonal. Hence, $\boldsymbol \eta = (\lambda_{j})_{\lambda_j \le 0} \oplus \tilde {\boldsymbol d} \oplus (\lambda_{j})_{ \lambda_j \ge 1}$ is a diagonal of $E$.

We claim that if a self-adjoint operator has diagonal $\boldsymbol \mu = (\lambda_{j})_{ \lambda_j \le 0} \oplus (d_i)_{ 0< d_i<1/2 } \oplus (\tilde d_i)_{ i\in I_0}$, then $ (d_{i})_{ d_i<1/2 } $ is also its diagonal. Indeed, 
\[
\delta( \alpha, \boldsymbol \mu ,  (d_{i})_{ d_i<1/2 } ) =
\begin{cases} \delta(\alpha,\boldsymbol \lambda, \boldsymbol d) & \alpha< 0,
\\
 \sum_{i\in I_0, \tilde d_i > \alpha} (\tilde d_i-\alpha) & \alpha >0.
\end{cases}
\]
By Lemma \ref{dclem}
\[
\liminf_{\alpha \nearrow 0} \delta( \alpha, \boldsymbol \mu ,  (d_{i})_{ d_i<1/2 } ) = \delta^L =\sum_{i\in I_0} \tilde d_i= \liminf_{\alpha \searrow 0} \delta( \alpha, \boldsymbol \mu ,  (d_{i})_{ d_i<1/2 } ).
\]
Hence, Theorem \ref{step5} applies and yields the claim if $\delta^L>0$. If $\delta^L=0$, then the claim holds trivially by $(i)$ since $\boldsymbol \mu =  ( d_i)_{ d_i<1/2 }$.

By a similar argument involving Theorem \ref{step5} or the assumption $(ii)$, if a self-adjoint operator has diagonal $\boldsymbol \nu =(d_i)_{ 1/2 \le d_i<1 } \oplus (\tilde d_i)_{ i\in I_1} \oplus  (\lambda_{j})_{ \lambda_j \ge 1} $, then $(d_{i})_{ d_i \ge 1/2} $ is also its diagonal. Since $\boldsymbol \eta = \boldsymbol \mu \oplus \boldsymbol \nu$ is a diagonal of $E$, by Lemma \ref{subdiag} so is $\boldsymbol d$.
\end{proof}

We use Lemma \ref{lsuff} to show the following diagonal-to-diagonal result under more general conditions on sequences $\boldsymbol\lambda$ and $\boldsymbol d$, which are nearly identical with the necessary conditions in Theorem \ref{nec}.

\begin{thm}\label{suff}
Let $\boldsymbol\lambda=(\lambda_{i})_{i\in \N}$ and  $\boldsymbol{d} = (d_{i})_{i\in \N}$ be two sequences, which have  accumulation points only at $0$ and $1$, and such that
 $f_{\boldsymbol \lambda}(1/2) <\infty$ and $f_{\boldsymbol d}(1/2) <\infty$.
Define
\[\delta^{L} = \liminf_{\beta\nearrow 0}(f_{\boldsymbol\lambda}(\beta)-f_{\boldsymbol{d}}(\beta))\quad\text{and}\quad \delta^{U} = \liminf_{\beta\searrow 1}(f_{\boldsymbol\lambda}(\beta)-f_{\boldsymbol{d}}(\beta)).\]
Assume that $\delta^L+\delta^U<\infty$,
\begin{equation}\label{suff1}
 f_{\boldsymbol{d}}(\alpha)  \leq f_{\boldsymbol\lambda}(\alpha)  
\qquad\text{for all }\alpha\in\R\setminus[0,1],
\end{equation}
and there exist $\delta_0 \in [0,\delta^L]$ and $\delta_1 \in [0,\delta^U]$ such that 
\begin{align}
\label{suff2}
f_{\boldsymbol\lambda}(\alpha)  &\le f_{\boldsymbol{d}}(\alpha) + (1-\alpha)\delta_{0} + \alpha\delta_{1} & \text{for all }\alpha\in(0,1),\\
\label{suff3}f_{\boldsymbol\lambda}'(\alpha)-f_{\boldsymbol{d}}'(\alpha)  & \equiv \delta_1-\delta_0 \mod 1 & \text{for some }\alpha\in(0,1),
\\
\label{suff4}
\sum_{\lambda_i<0} |\lambda_i|<\infty &\implies
\delta^L-\delta_0 \le \delta^U- \delta_1,
\\
\label{suff5}
\sum_{\lambda_i>1} (\lambda_i-1)<\infty
&\implies \delta^U- \delta_1 \le \delta^L-\delta_0.
\end{align}
In addition, in the case when $\delta_0=\delta^L$ or $\delta_1=\delta^U$ assume that
\begin{equation}\label{suff6}
 \delta_0, \delta_1>0 \qquad\text{and}\qquad 
f_{\boldsymbol\lambda}(\alpha)  < f_{\boldsymbol{d}}(\alpha) + (1-\alpha)\delta_{0} + \alpha\delta_{1}  \text{ for all }\alpha\in(0,1).
\end{equation}
If $\boldsymbol \lambda $ is a diagonal of a self-adjoint operator $E$, then so is $\boldsymbol d$.
\end{thm}

\begin{remark}
Assume that $E$ is an operator satisfying hypotheses of Theorem \ref{nec}. Then, if the last assumption \eqref{suff6} of Theorem \ref{suff} is {\bf not} met, then we necessarily land in one of the following three scenarios:
\begin{itemize}
\item $(\mathfrak d_0)$ holds,
\item $(\mathfrak d_\alpha)$ holds for some $\alpha \in (0,1)$, or
\item $(\mathfrak d_1)$ holds.
\end{itemize}
By Theorem \ref{nec} we deduce that the operator $E$ decouples at $\alpha \in [0,1]$. Hence, we have additional necessary conditions on eigenvalue list $\boldsymbol \lambda $ and diagonal $\boldsymbol d$, see Theorem \ref{necdec}, which we need to impose to obtain diagonal-to-diagonal result.
\end{remark}

\begin{proof}
Without loss of generality we can assume that $\delta_0<\delta^L$ and $\delta_1<\delta^U$ by decreasing $\delta_0$ and $\delta_1$ by the same quantity if necessary. Indeed, if $\delta_0=\delta^L$ or $\delta_1=\delta^U$, then by \eqref{suff6} and Proposition \ref{f}, there exists a constant $0<c \le \min(\delta_0,\delta_1)$ such that
\[
f_{\boldsymbol\lambda}(\alpha) + c \le  f_{\boldsymbol{d}}(\alpha) + (1-\alpha)\delta_{0} + \alpha\delta_{1}  \qquad\text{for all }\alpha \in (0,1).
\]
Hence, replacing $\delta_0$ and $\delta_1$ by $\delta_0-c$ and $\delta_1-c$, respectively, does not affect assumptions \eqref{suff2}--\eqref{suff5}.

{\bf Step 1.} We reduce to the case when there are infinitely many $\lambda_j \in (0,1/2)$ and infinitely many $\lambda_j \in (1/2,1)$. 
If there are only finitely many $\lambda_j \in (0,1/2)$, then there are infinitely many $\lambda_j \le 0$, because $0$ was assumed to be an accumulation point of $\boldsymbol \lambda$. 

Take $0<\epsilon< \min(\delta^L-\delta_0,\delta^U - \delta_1)$. By increasing $\delta_0$ and $\delta_1$ by $\epsilon$ we can assume that $\delta_0,\delta_1 \ge \epsilon$ and
\begin{equation}\label{suff9}
f_{\boldsymbol\lambda}(\alpha) + \epsilon \le  f_{\boldsymbol{d}}(\alpha) + (1-\alpha)\delta_{0} + \alpha\delta_{1}  \qquad\text{for all }\alpha \in (0,1).
\end{equation}
Choose $\alpha_0<0$ such that
\begin{equation}\label{suff10}
\delta(\alpha,\boldsymbol \lambda, \boldsymbol d) \ge \epsilon \qquad\text{for } \alpha_0<\alpha<0.
\end{equation}
Likewise, choose $\alpha_1>1$ such that 
\begin{equation}\label{suff12}
\delta(\alpha,\boldsymbol \lambda, \boldsymbol d) \ge \epsilon
\qquad\text{for all }1<\alpha<\alpha_1.
\end{equation}
If there are infinitely many $\lambda_j>1$, then we choose $j_0$ such that $1<\lambda_{j_0}<\alpha_1$. Otherwise, we choose for $\lambda_{j_0}$ the smallest $\lambda_j>1$. We also choose an infinite subset $J_0$ such that $\alpha_0 \le \lambda_j \le 0$ for all $j\in J_0$ and $\sum_{j\in J_0} |\lambda_j | < \min(\epsilon/2,\lambda_{j_0}-1)$. 

Take any $\eta$ such that $\sum_{j\in J_0} |\lambda_j | < \eta< \min(\epsilon/2,\lambda_{j_0}-1)$ and let $\tilde \lambda_{j_0} = \lambda_{j_0}- \eta$. Choose a positive sequence $(\tilde \lambda_j)_{j\in J_0}$ and such that 
\[
\sum_{j\in J_0} \tilde \lambda_j= \eta -\sum_{j\in J_0} |\lambda_j | = \lambda_{j_0} - \tilde \lambda_{j_0} + \sum_{j\in J_0} \lambda_j .
\]
By Theorem \ref{step5} if $(\lambda_{j})_{j\in J_0 \cup \{j_0\}}$ is a diagonal of a self-adjoint operator, then so is 
$(\tilde \lambda_j)_{j\in J_0 \cup \{j_0\}}$. Define $\tilde \lambda_j=\lambda_j$ for $j \not \in J_0 \cup\{ j_0\}$. Then, the sequences $\tilde {\boldsymbol \lambda}$ and $\boldsymbol d$ satisfy the assumptions \eqref{suff1}--\eqref{suff5} of Theorem \ref{suff} with 
\begin{equation}\label{suff14}
\tilde \delta^L= \delta^L - \sum_{j\in J_0} |\lambda_j | ,
\quad
 \tilde \delta_0= \delta_0 - \sum_{j\in J_0} |\lambda_j |,
 \quad
 \tilde \delta^U = \delta^U - \eta, 
 \quad
 \tilde \delta^1= \delta^1 -\eta.
\end{equation}
Indeed, \eqref{suff10} and \eqref{suff12} imply that \eqref{suff1} holds for $\tilde {\boldsymbol \lambda}$ and $\boldsymbol d$ in the range $\alpha<0$ and $\alpha>1$, respectively. Likewise, \eqref{suff9} and \eqref{suff14} imply that \eqref{suff2} holds.
Indeed, for any $\alpha \in (0,1)$ we have by Proposition \ref{fd}
\[\begin{aligned}
f_{\tilde {\boldsymbol\lambda}}(\alpha) \le f_{\boldsymbol\lambda}(\alpha) + (1-\alpha)\sum_{j\in J_0} \tilde \lambda_j 
&\le    f_{\boldsymbol{d}}(\alpha) + (1-\alpha)\delta_{0} + \alpha\delta_{1} + \eta -\epsilon
\\
&\le f_{\boldsymbol{d}}(\alpha) + (1-\alpha)\delta_{0} + \alpha\delta_{1} - \eta 
\le   f_{\boldsymbol{d}}(\alpha) + (1-\alpha) \tilde \delta_{0} + \alpha\tilde \delta_{1}  .
\end{aligned}
\]
Perhaps the least obvious is \eqref{suff3}
\[
f_{\tilde {\boldsymbol\lambda}}'(\alpha)-f_{\boldsymbol{d}}'(\alpha)  =
f_{\boldsymbol\lambda}'(\alpha) - \sum_{j\in J_0} \tilde \lambda_j -f_{\boldsymbol{d}}'(\alpha)
\equiv \delta_1-\delta_0 + \sum_{j\in J_0} |\lambda_j | -\eta =
\tilde \delta_1- \tilde \delta_0 \mod 1.
\]
Finally, \eqref{suff4} and \eqref{suff5} follow immediately from the definition of $\boldsymbol \lambda$ and \eqref{suff14}.

This process modifies the diagonal sequence $\boldsymbol \lambda$ to another diagonal sequence $\tilde {\boldsymbol \lambda}$ such that $\tilde \lambda_j \in (0,1/2)$ for infinitely many $j$ and the assumptions of Theorem \ref{suff} for $\tilde {\boldsymbol \lambda}$ and $\boldsymbol d$ are met. Therefore, without loss of generality we can assume that there are infinitely many $\lambda_j \in (0,1/2)$. By a similar procedure we can assume that there are infinitely many $\lambda_j \in (1/2,1)$.

{\bf Step 2.} Next we reduce to the case when $\delta^L-\delta_0=\delta^U - \delta_1$. Suppose that $\delta^L-\delta_0>\delta^U - \delta_1$. By \eqref{suff4} this implies that $\sum_{j \in J_0} |\lambda_j|=\infty$, where $J_0=\{j: \lambda_j<0\}$. Since $\delta^L<\infty$ we also have $\sum_{i\in I_0} |d_i|=\infty$, where $I_0=\{i: d_i<0\}$.
Define $\tilde \delta^L:=\delta_0+ \delta^U - \delta_1$. Since $\tilde \delta^L<\delta^L$, there exists $\alpha_0<0$ such that
\begin{equation}\label{suff20}
\delta(\alpha, \boldsymbol \lambda,\boldsymbol d ) \ge \tilde \delta^L \qquad\text{for all } \alpha\in (\alpha_0,0).
\end{equation}
Let $\boldsymbol \nu$ be the sequence $(d_{i})_{i\in I_0}$ appended by a finite sequence which consists of terms in $(\alpha_0,0)$ and adds up to $-\tilde \delta^L$. Define a sequence $\tilde{\boldsymbol \lambda}=(\lambda_j)_{j\in \N \setminus J_0} \oplus \boldsymbol \nu$. 

By \eqref{suff20} we have
\[
\delta(\alpha, \boldsymbol \lambda, \tilde{\boldsymbol \lambda} ) = \delta(\alpha, (\lambda_j)_{j\in J_0}, \boldsymbol \nu ) \ge 0 \qquad\text{for all } \alpha<0.
\]
By Proposition \ref{kwd2d}, if $(\lambda_{j})_{j\in J_0}$ is a diagonal of a self-adjoint operator, then $\boldsymbol \nu$ is also a diagonal. Hence, by Lemma \ref{subdiag} we deduce that $\tilde{\boldsymbol \lambda}$ is a diagonal of $E$.

On the other hand,
\[
\delta(\alpha, \tilde{ \boldsymbol \lambda} , \boldsymbol d)  =  \delta(\alpha, \boldsymbol \nu , (d_{i})_{i\in I_0} ) \ge 0 \qquad\text{for all } \alpha<0.
\]
By Lemma \ref{dclem} we have
\[
\liminf_{\alpha \nearrow 0} \delta(\alpha, \tilde{ \boldsymbol \lambda} , \boldsymbol d) = \tilde \delta^L.
\]
Therefore, sequences $\tilde{\boldsymbol \lambda}$ and $\boldsymbol d$ satisfy assumptions \eqref{suff1}-\eqref{suff3} with the exception that $\delta^L$ is replaced by $\tilde \delta^L=\delta_0+ \delta^U - \delta_1$. In a similar way we deal with the case when $\delta^L-\delta_0>\delta^U - \delta_1$ using \eqref{suff5}.

{\bf Step 3.} By two previous reductions we can assume that:
\begin{itemize}
\item the sequence $\boldsymbol \lambda$ has infinitely many terms $0<\lambda_i<1/2$ and infinitely many terms $1/2<\lambda_i<1$, 
\item $\delta^L,\delta^U>0$, and
\item  $\delta^L-\delta_0=\delta^U - \delta_1>0$.
\end{itemize}
Hence, the assumption $(iii)$ of Lemma \ref{lsuff} is met, which completes the proof.
\end{proof}

We can also use Lemma \ref{lsuff} to show the following diagonal-to-diagonal result when exactly one of the exterior parts is missing.

\begin{thm}\label{csuff}
Let $\boldsymbol\lambda=(\lambda_{i})_{i\in \N}$ and  $\boldsymbol{d} = (d_{i})_{i\in \N}$ be two sequences, which satisfy the same assumptions \eqref{suff1}--\eqref{suff5}, as in Theorem \ref{suff}.
In addition, assume one of the following holds:
\begin{enumerate}[(i)]
\item \(\lambda_{i},d_{i}>0\) for all \(i\in\N\), and \(\delta_{1}>0\),
\item \(\lambda_{i},d_{i}<1\) for all \(i\in\N\), and \(\delta_{0}>0\),
\end{enumerate}
If $\boldsymbol \lambda $ is a diagonal of a self-adjoint operator $E$, then so is $\boldsymbol d$.
\end{thm}

\begin{proof} Assume \textit{(ii)} holds. If \(\delta^{L}=\delta_{0}\), then the conclusion follows by Lemma \ref{lsuff}. Hence, we may assume that \(\delta^{L}>\delta_{0}\) and we will proceed as in the proof of Step 2 in Theorem \ref{suff}.
By \eqref{suff4}, we have \(\sum_{j\in J_0 }|\lambda_{j}|=\infty\), where $J_0=\{j: \lambda_{j}<0\}$. Since \(\delta_{L}<\infty\), we also have \(\sum_{i\in I_0}|d_{i}|=\infty\), where $I_0=\{i : d_i<0\}$. There exists \(\alpha_{0}<0\) such that 
\begin{equation}\label{suff30}
\delta(\alpha,\boldsymbol\lambda,\boldsymbol d)\geq \delta_{0}\quad \text{for all }\alpha\in(\alpha_{0},0).
\end{equation}

Let $\boldsymbol \nu$ be the sequence $(d_{i})_{i\in I_0}$ appended by a finite sequence which consists of terms in $(\alpha_0,0)$ and adds up to $- \delta_0$. Define a sequence $\tilde{\boldsymbol \lambda}=(\lambda_{j})_{j\in \N \setminus J_0} \oplus \boldsymbol \nu$. 

By \eqref{suff30} we have
\[
\delta(\alpha, \boldsymbol \lambda, \tilde{\boldsymbol \lambda} ) = \delta(\alpha, (\lambda_{j})_{j\in J_0}, \boldsymbol \nu ) \ge 0 \qquad\text{for all } \alpha<0.
\]
By Proposition \ref{kwd2d}, if $(\lambda_{j})_{j\in J_0}$ is a diagonal of a self-adjoint operator, then $\boldsymbol \nu$ is also a diagonal. Hence, by Lemma \ref{subdiag} we deduce that $\tilde{\boldsymbol \lambda}$ is a diagonal of $E$.

On the other hand,
\[
\delta(\alpha, \tilde{ \boldsymbol \lambda} , \boldsymbol d)  =  \delta(\alpha, \boldsymbol \nu , (d_{i})_{i\in I_0} ) \ge 0 \qquad\text{for all } \alpha<0.
\]
By Lemma \ref{dclem} we have
\[
\liminf_{\alpha \nearrow 0} \delta(\alpha, \tilde{ \boldsymbol \lambda} , \boldsymbol d) = \delta_0.
\]
Therefore, sequences $\tilde{\boldsymbol \lambda}$ and $\boldsymbol d$ satisfy assumptions \eqref{suff1}-\eqref{suff3} with the exception that $\delta^L$ is replaced by $\tilde \delta^L=\delta_0$. Hence, the assumption (ii) in Lemma \ref{lsuff}. Consequently, $\boldsymbol d$ is a diagonal of $E$.
\end{proof}

\section{Non-Blaschke diagonal result}\label{S19}

In this section we show that all sequences that satisfy exterior majorization, but do not satisfy the Blaschke condition in the interior, are diagonals of self-adjoint operators with at least two points in the essential spectrum. Moreover, in the case that the exterior excess is infinite we show the same result even if the Blaschke condition holds.

M\"uler and Tomilov \cite[Theorem 1.1]{mt} have shown the following result about the existence of diagonals satisfying non-Blaschke condition. Their results also hold for $n$-tuples of self-adjoint operators, but we will not need such level of generality. 
For convenience we assume that $0$ and $1$ are two extreme points of the essential spectrum, but they can be replaced by any pair of distinct real numbers.

\begin{thm}\label{tomilov}
Let $E$ be a self-adjoint operator on $\mathcal H$ such that $0$ and $1$ are two extreme points of the essential spectrum $\sigma_{ess}(E)$ of $E$, that is,
\begin{equation}\label{ess}
\{0,1 \} \subset \sigma_{ess}(E) \subset [0,1].
\end{equation}
Let  $\boldsymbol{d}=(d_{i})_{i\in I}$ be a sequence in $(0,1)$ such that 
\[
\sum_{i\in I} \min(d_i,1-d_i) = \infty.
\]
Then, $\boldsymbol d$ is a diagonal of $E$.
\end{thm}

\begin{lem}\label{nbdL}
Let $E$ be a self-adjoint operator on $\mathcal H$ such that $0$ and $1$ are two extreme points of the essential spectrum $\sigma_{ess}(E)$ of $E$, that is,
\begin{equation*}
\{0,1 \} \subset \sigma_{ess}(E) \subset [0,1].
\end{equation*}
In addition, assume \(\sigma(E)\subset(-\infty,1]\).
Let $\boldsymbol\lambda=(\lambda_{j})_{j\in J}$ be the list of all eigenvalues of $E$ with multiplicity (which is possibly an empty list). 
If  $\boldsymbol{d}=(d_{i})_{i\in I}$ is a sequence in \((-\infty,1)\) such that \(f_{\boldsymbol{d}}(\alpha)=\infty\) for some \(\alpha\in(0,1),\)
\[\delta^{L} = \liminf_{\beta\nearrow 0}(f_{\boldsymbol\lambda}(\beta)-f_{\boldsymbol{d}}(\beta))>0,\]
and
\begin{equation*}
 f_{\boldsymbol{d}}(\alpha)  \leq f_{\boldsymbol\lambda}(\alpha)  
\qquad\text{for all }\alpha<0,
\end{equation*}
then \(\boldsymbol{d}\) is a diagonal of \(E\).
\end{lem}

\begin{proof} We need consider two cases.

{\bf Case 1.} Suppose that the sequence $(d_{i})_{d_i \le 0}$ is infinite. 
Let $\boldsymbol \nu$ be an infinite sequence in $(0,1)$ that sums to $\delta^L$. 
By Theorem \ref{tomilov} the sequence $\boldsymbol \nu \oplus (d_{i})_{0<d_i<1}$ is a diagonal of $EP|_{\mathcal H_0}$, where $\pi$ is the spectral measure of $E$,  $P=\pi(0,\infty)$, and $\mathcal H_0=\ran(P)$. Hence, the sequence $(\lambda_{j})_{\lambda_j \le 0} \oplus  \boldsymbol \nu \oplus (d_{i})_{0<d_i<1} $ is a diagonal of $E$.
In the case that $\delta^L<\infty$, by Theorem \ref{step5} we conclude that if $(\lambda_{j})_{\lambda_j \le 0} \oplus \boldsymbol \nu $ is a diagonal of a self-adjoint operator, then so is $(d_{i})_{d_i \le 0}$. We deduce the same conclusion in the case that $\delta^L=\infty$ using Theorem \ref{step5}.
Hence, by Lemma \ref{subdiag} we conclude that $\boldsymbol d$ is a diagonal of $E$.

{\bf Case 2.} Suppose that $I_0=\{i: d_i \le 0\}$ is finite. Let $J_0$ be a finite set with at most $\#|I_0|$ elements such that $(\lambda_{j})_{j\in J_0}$ contains the smallest elements of $(\lambda_{j})_{\lambda_j \le 0}$ and 
\[
\delta(\alpha,(\lambda_{j})_{j\in J_0}, (d_{i})_{i\in I_0}) \ge 0\qquad\text{for } \alpha<0.
\]
Hence,
\[
{\tilde \delta}^L= \sum_{j\in J_0} |\lambda_j| - \sum_{i\in I_0} |d_i| \ge 0
\]
and the inequality is strict if $N=\#|I_0|-\#|J_0|>0$. In this case we choose $\alpha \in (0,1)$ such that $N\alpha<\tilde \delta^L$. Otherwise, if $N=0$ we let $\alpha=0$. Next we choose a finite set $K_0$ such that $0<d_i<1$ for $i\in K_0$ and $\sum_{i\in K_0} (1-d_i)> \tilde \delta^L$. Choose $d_i \le \tilde d_i <1$, $i\in K_0$, such that 
\[
\sum_{i\in K_0} (\tilde d_i - d_i) = {\tilde \delta}^L -\alpha N.
\]

By the Schur-Horn theorem, if $(\lambda_{j})_{j\in J_0} \oplus \alpha 1_{N} \oplus (\tilde{d}_{i})_{i\in K_0}$ is a diagonal of a self-adjoint (finite-dimensional) operator then, so is $(d_{i})_{i \in I_0 \cup K_0}$. Here, $1_N$ denotes the the constant sequence of $N$ terms equal to $1$.
 By Theorem \ref{tomilov} the sequence $\alpha 1_{N} \oplus (\tilde{d}_{i})_{i\in K_0} \oplus (d_{i})_{i\in I \setminus (I_0 \cup K_0) }$ is a diagonal of $E|_{\mathcal H_0}$, where $\mathcal H_0$ is the orthogonal complement of eigenvectors corresponding to eigenvalues $\lambda_j$, $j\in J_0$. 
Hence, the sequence $(\lambda_{j})_{j\in J_0} \oplus  \alpha 1_{N} \oplus (\tilde{d}_{i})_{i\in K_0} \oplus (d_{i})_{i\in I \setminus (I_0 \cup K_0) }$ is a diagonal of $E$.  Therefore, by Lemma \ref{subdiag} we conclude that $\boldsymbol d$ is a diagonal of $E$.
\end{proof}

\begin{thm}\label{nbd}
Let $E$ be a self-adjoint operator on $\mathcal H$ such that $0$ and $1$ are two extreme points of the essential spectrum $\sigma_{ess}(E)$ of $E$, that is,
\begin{equation}
\{0,1 \} \subset \sigma_{ess}(E) \subset [0,1].
\end{equation}
Let $\boldsymbol\lambda=(\lambda_{j})_{j\in J}$ be the list of all eigenvalues of $E$ with multiplicity (which is possibly an empty list). 
Let  $\boldsymbol{d}=(d_{i})_{i\in I}$ be a sequence such that
\begin{equation}
 f_{\boldsymbol{d}}(\alpha)  \leq f_{\boldsymbol\lambda}(\alpha)  
\qquad\text{for all }\alpha\in\R\setminus[0,1].
\end{equation}
Define
\[\delta^{L} = \liminf_{\beta\nearrow 0}(f_{\boldsymbol\lambda}(\beta)-f_{\boldsymbol{d}}(\beta))\quad\text{and}\quad \delta^{U} = \liminf_{\beta\searrow 1}(f_{\boldsymbol\lambda}(\beta)-f_{\boldsymbol{d}}(\beta)).\]
Assume that one of the following happens:
\begin{enumerate}[(a)]
\item
$\delta^L,\delta^U> 0$ and $f_{\boldsymbol{d}}(\alpha)=\infty$ for some $\alpha\in(0,1)$,
\item
$\delta^L> 0$, $\delta^U=\infty$, and $\sum_{d_i>0} d_i=\infty$,
\item[(b')]
$\delta^{L}=0$, $\delta^{U}=\infty$, $\sum_{d_i>0} d_i=\infty$, and $d_{i}>0$ for all $i\in I$
\item
$\delta^U> 0$, $\delta^L=\infty$, and $\sum_{d_i<1} (1-d_i)=\infty$,
\item[(c')]
$\delta^U = 0$, $\delta^L=\infty$, $\sum_{d_i<1} (1-d_i)=\infty$, and $d_{i}<1$ for all $i\in I$
\item
$\delta^L=\delta^U =\infty$.
\end{enumerate} 
Then, $\boldsymbol d$ is a diagonal of $E$.
\end{thm}

\begin{proof}[Proof of (a)] We claim that there exists two orthogonal invariant subspaces $\mathcal K_0$ and $\mathcal K_1$ of $E$ and a partition $\{I_0,I_1\}$ of the index set $I$ such that:
\begin{enumerate}
\item
$\mathcal H = \mathcal K_0 \oplus \mathcal K_1$ and
the operators $E_0= E|_{\mathcal K_0}$ and $E_1= E|_{\mathcal K_1}$ satisfy \eqref{ess}
\item 
$d_i <1$ for all $i \in I_0$ and $d_i>0$ for all $i\in I_1$,
\item
letting $\boldsymbol d^k=(d_{i})_{i\in I_k}$, we have
$f_{\boldsymbol d^k}(\alpha)=\infty$ for $k=0,1$ and $\alpha \in (0,1)$,
\item 
letting $\boldsymbol \lambda^k$ be the the list of all eigenvalues of $E_k$ with multiplicity, we have
\begin{align}\label{k0}
f_{\boldsymbol \lambda^0}(\alpha) \ge f_{\boldsymbol d^0}(\alpha)\quad\text{for all }\alpha<0 &\qquad\text{and}\qquad
\liminf_{\beta\nearrow 0}(f_{\boldsymbol\lambda^0}(\beta)-f_{\boldsymbol{d}^0}(\beta))>0,
\\
\label{k1}
f_{\boldsymbol \lambda^1}(\alpha) \ge f_{\boldsymbol d^1}(\alpha)\quad\text{for all }\alpha>1
&\qquad\text{and}\qquad
 \liminf_{\beta\searrow 1}(f_{\boldsymbol\lambda^1}(\beta)-f_{\boldsymbol{d}^1}(\beta))>0.
 \end{align}

\end{enumerate}
To prove the claim
we can partition a sequence $(d_{i})_{0<d_i<1}$ into two subsequences, both of which satisfy non-Blaschke condition. Then we append all the terms $d_i \le 0$ to the first subsequence to form $\boldsymbol d^0=(d_{i})_{i\in I_0}$. Likewise, we append all the terms $d_i \ge  1$ to the second subsequence to form $\boldsymbol d^1=(d_{i})_{i\in I_1}$. This shows (ii) and (iii). 

Let $\pi$ be the spectral measure of $E$.
If the sequence $(\lambda_{j})_{\lambda_j \le 0}$ is infinite, then we require that the subspaces $\mathcal K_0$ and $\mathcal K_1$ each contain infinitely many eigenvectors corresponding to eigenvalues $\lambda_j\le 0$ and \eqref{k0} holds. This is possible since we assume that $\delta^L>0$. Otherwise, if $(\lambda_{j})_{\lambda_j \le 0}$ is finite, then we require that $\mathcal K_0$ contains all eigenvectors $\lambda_j \le 0$. In this case, the rank of the projection $\pi(0,\delta)$ is infinite for all $\delta>0$. Hence, we can partition the interval $(0,1/2)$ into two Borel sets $B_0$ and $B_1$ such that $\operatorname{rank}(\pi(B_k))=\infty$, $k=0,1$, to guarantee that $0\in \sigma_{ess}(E_0)$ and $0\in \sigma_{ess}(E_1)$.
Likewise, if the sequence $(\lambda_{j})_{\lambda_j \ge 1}$ is infinite, then we require that the subspaces $\mathcal K_0$ and $\mathcal K_1$ each contain infinitely many eigenvectors corresponding to eigenvalues $\lambda_j\ge 1$ and \eqref{k1} holds. 
 Otherwise, we require that $\mathcal K_1$ contains all eigenvectors $\lambda_j \ge 1$. This yields (i) and (iv).

By the claim, we can reduce the proof of Theorem \ref{nbd} to the case when the sequence $\boldsymbol{d}=(d_{i})_{i\in I}$ satisfies $d_i<1$ for all $i\in I$, albeit $\delta^U$ could take value $0$. This is exactly Lemma \ref{nbdL}.
\end{proof}

\begin{proof}[Proof of (c)]  In light of \textit{(a)}, it suffices to consider the case when $f_{\boldsymbol{d}}(\alpha)<\infty$ for some $\alpha\in(0,1)$. Together with the assumption that  $\sum_{d_i<1}(1-d_i)=\infty$, this implies $0$ is an accumulation point of the subsequence $(d_{i})_{d_i<1}$. Hence, there exists an infinite subset $I_0 \subset I$ such that $0$ is the only accumulation point of $(d_{i})_{i\in I_0}$ and
\begin{equation}\label{k5}
\{i\in I: d_i \le 0\} \subset I_0.
\end{equation}
Choose a sequence $\tilde{\boldsymbol d} = (\tilde{d}_{i})_{i\in \N}$ such that $\tilde d_i\in(0,1)$ for all $i\in \N$, $\lim \tilde d_i =0$, and $\sum\tilde d_i=\infty$. Choose an infinite subset $J_0 \subset \{j\in J: \lambda_j<0\}$ such that $\sum_{j\in J_0} |\lambda_j|<\infty$ and $\delta(\alpha, (\lambda_{j})_{j\in J \setminus J_0}, \boldsymbol d) \ge 0$ for all $\alpha <0$. Consequently,
\begin{equation}\label{k6}
\liminf_{\alpha \nearrow 0} \delta(\alpha, (\lambda_{j})_{j\in J_1}, \boldsymbol d) = \infty,
\qquad\text{where  } J_1=\{j\in J\setminus J_0 : \lambda_j<0\}.
\end{equation}

By part \textit{(a)} the sequence $(d_{i})_{d_i>0} \oplus \tilde{\boldsymbol d}$ is a diagonal of $E|_{\mathcal H_0}$, where $\mathcal H_0$ is the orthogonal complement of the eigenvectors corresponding to eigenvalues $\lambda_j$, $j\in J_1$. 
Hence, the sequence $(d_{i})_{d_i>0} \oplus \tilde{\boldsymbol d} \oplus (\lambda_{j})_{j\in J_1}$ is a diagonal of $E$.
Let $I_1=\{i \in I_0: d_i>0\}$. By Theorem \ref{step5} and \eqref{k6}, if $(d_{i})_{i\in I_1} \oplus \tilde{\boldsymbol d} \oplus  (\lambda_{j})_{j\in J_1}$ is a diagonal of a self-adjoint operator, then so is $(d_{i})_{i\in I_0}$. By \eqref{k5} we have
\[\{i\in I : d_{i}>0\}\setminus I_{1} = I\setminus I_{0}.\]
Therefore, by Lemma \ref{subdiag}  we conclude that $(d_{i})_{i\in I_0} \oplus (d_{i})_{i\in I \setminus I_0} =\boldsymbol d$ is a diagonal of $E$.
\end{proof}

\begin{proof}[Proof of (c')] The proof is identical to the proof of \textit{(c)}, with the two applications of (a) replaced with Lemma \ref{nbdL}. 
\end{proof}

\begin{proof}[Proof of (b), (b'), and (d)] By symmetry, parts \textit{(c)} and \textit{(c')} imply parts \textit{(b)} and \textit{(b')}, respectively. Assume that $\delta^L=\delta^U =\infty$. Since $\boldsymbol d$ is an infinite sequence, one of the sums  $\sum_{d_i>0} d_i$ or $\sum_{d_i<1} (1-d_i)$ must be infinite. Hence, part \textit{(d)} follows immediately from parts \textit{(b)} and \textit{(c)}.
\end{proof}

\section{Main result for operators with at least \(2\)-point essential spectrum} \label{S8}

In this section we formulate and prove our main result, Theorem \ref{pmt}, on diagonals of self-adjoint operators with $\ge 2$ points in their essential spectrum. To achieve this we combine necessity results from Section \ref{necProof} and sufficiency results from Sections \ref{S16}--\ref{S19}. The statement of Theorem \ref{pmt} uses the concept of decoupling, see Definition \ref{decouple}.
Finally, we also prove Theorem \ref{mainthm}, which yields a characterization of $\mathcal D(T)$ for the class of self-adjoint operators sharing the same spectral measure as $T$ with a possible exception of the multiplicities of the eigenvalues $0$ and $1$. While the proof of Theorem \ref{pmt} merely combines earlier results, the proof of Theorem \ref{mainthm} requires additional work.

\begin{thm}\label{pmt}
Let $E$ be a self-adjoint operator on $\mathcal H$ such that $0$ and $1$ are two extreme points of the essential spectrum $\sigma_{ess}(E)$ of $E$, that is,
\begin{equation}\label{pmt0}
\{0,1 \} \subset \sigma_{ess}(E) \subset [0,1].
\end{equation}
Let $\boldsymbol\lambda=(\lambda_{j})_{j\in J}$ be the list of all eigenvalues of $E$ with multiplicity (which is possibly an empty list). 
Let $\boldsymbol{d} = (d_{i})_{i\in \N}$ be a sequence of reals. Define
\[\delta^{L} = \liminf_{\beta\nearrow 0}(f_{\boldsymbol\lambda}(\beta)-f_{\boldsymbol{d}}(\beta))\quad\text{and}\quad \delta^{U} = \liminf_{\beta\searrow 1}(f_{\boldsymbol\lambda}(\beta)-f_{\boldsymbol{d}}(\beta)).\]

\noindent {\rm (Necessity)} If $\boldsymbol{d}$  is a diagonal of $E$, then
\begin{equation}
\label{pmt1}
 f_{\boldsymbol{d}}(\alpha)  \leq f_{\boldsymbol\lambda}(\alpha)  
\qquad\text{for all }\alpha\in\R\setminus[0,1],
\end{equation}
and one of the following five conditions holds:
\begin{enumerate}
\item
$\delta^L=\delta^U =\infty$,
\item
$\delta^L<\infty$, $\delta^U=\infty$, and $\sum_{d_i>0} d_i=\infty$,
\item
$\delta^U<\infty$, $\delta^L=\infty$, and $\sum_{d_i<1} (1-d_i)=\infty$,
\item 
$\delta^L+\delta^U <\infty$ and $f_{\boldsymbol d}(1/2) =\infty$, or
\item
$\delta^L+\delta^U <\infty$, $f_{\boldsymbol \lambda}(1/2) <\infty$, $f_{\boldsymbol d}(1/2) <\infty$, $\boldsymbol d$ has accumulation points at $0$ and $1$, and there exists $\delta_0 \in [0,\delta^L]$ and $\delta_1 \in [0,\delta^U]$ such that:
\begin{align}
\label{pmt2}
f_{\boldsymbol\lambda}(\alpha)  &\leq f_{\boldsymbol{d}}(\alpha) + (1-\alpha)\delta_{0} + \alpha\delta_{1} & \text{for all }\alpha\in(0,1),\\
\label{pmt3}f_{\boldsymbol\lambda}'(\alpha)-f_{\boldsymbol{d}}'(\alpha)  & \equiv \delta_1-\delta_0 \mod 1 & \text{for some }\alpha\in(0,1),
\\
\label{pmt4}
\sum_{\lambda_i<0} |\lambda_i|<\infty &\implies
\delta^L-\delta_0 \le \delta^U- \delta_1,
\\
\label{pmt5}
\sum_{\lambda_i>1} (\lambda_i-1)<\infty
&\implies \delta^U- \delta_1 \le \delta^L-\delta_0.
\end{align}
\end{enumerate}

In addition, the operator $E$ decouples at a point $\alpha\in [0,1]$ if \((\mathfrak{d}_{\alpha})\) holds, where
\begin{enumerate}
\item[$(\mathfrak d_0)$] $\delta^L=0$ \quad or  \quad $\delta_0=0$ and $\delta_1=\delta^U$,
\item[$(\mathfrak d_\alpha)$] $\alpha \in (0,1)$ and $f_{\boldsymbol\lambda}(\alpha)  = f_{\boldsymbol{d}}(\alpha) + (1-\alpha)\delta_{0} + \alpha\delta_{1}$,
\item[$(\mathfrak d_1)$] $\delta^U=0$ \quad or  \quad  $\delta_1=0$ and $\delta_0=\delta^L$.
\end{enumerate}

\noindent {\rm (Sufficiency)}  Conversely, if exterior majorization \eqref{pmt1} holds, one of conditions (i)--(v) holds, and $(\mathfrak d_\alpha)$ does {\bf not} hold for any $\alpha \in [0,1]$, then $\boldsymbol{d}$  is a diagonal of $E$.
\end{thm}

\begin{remark}\label{pmt6}
Note that the numbers \(\delta_{0}\) and \(\delta_{1}\) are only defined when (v) holds. Hence, in the case that (v) does not hold, the statements \((\mathfrak{d}_{\alpha})\) must be properly interpreted using the following convention. If one of the conditions (i)--(iv) holds, then \(\delta_{0}\) and \(\delta_{1}\) are not defined, and we interpret any equality involving them in the statements \((\mathfrak{d}_{\alpha})\) as false. With this convention, we note that if one of (i)--(iv) holds, then \((\mathfrak{d}_{\alpha})\) is false for all \(\alpha\in(0,1)\), whereas the statement \((\mathfrak{d}_{0})\) becomes \(\delta^{L}=0\), and  \((\mathfrak{d}_{1})\) becomes \(\delta^{U}=0\).
\end{remark}

\begin{proof} (Necessity) First, note that \eqref{pmt1} follows by Theorem \ref{cptschurv2}. There are now four mutually exclusive possibilities, depending on whether \(\delta^{L}\) and \(\delta^{U}\) are finite or infinite. If both are infinite, then there is nothing to show. If exactly one of them is finite, then the desired conclusion follows from Lemma \ref{neci}. Finally, if both are finite, then either \(f_{\boldsymbol d}(1/2)=\infty\), in which case there is nothing to show, otherwise Theorem \ref{nec} yields exactly conclusion (v) of Theorem \ref{pmt} and conditions that guarantee that the operator \(E\) decouples at \(\alpha\in[0,1]\). 

(Sufficiency)
Suppose that \eqref{pmt1} holds. Since \((\mathfrak{d}_{\alpha})\) does not hold for \(\alpha\in\{0,1\}\) we see that \(\delta^{U},\delta^{L}>0\). If any of the conditions (i)-(iv) hold, then \(\boldsymbol d\) is a diagonal of \(E\) by Theorem \ref{nbd}. If condition (v) holds, then Theorem \ref{nec} implies \(E\) is diagonalizable, and hence \(\boldsymbol\lambda\) is a diagonal of \(E\). That \((\mathfrak{d}_{0})\) and \((\mathfrak{d}_{1})\) do not hold shows that if either \(\delta_{0}=\delta^{L}\) or \(\delta_{1}=\delta^{U}\), then \(\delta_{0},\delta_{1}>0\). The assumption that \((\mathfrak{d}_{\alpha})\) does not hold for any \(\alpha\in(0,1)\) implies that have strict inequality in \eqref{pmt2}. Hence, if \(\delta_{0}=\delta^{L}\) or \(\delta_{1}=\delta^{U}\) holds, then \eqref{suff6} also holds. Thus, we can apply Theorem \ref{suff} to deduce that \(\boldsymbol{d}\) is a diagonal of \(E\).
\end{proof}

\begin{proof}[Proof of Theorem \ref{mainthm}]
The necessity direction follows directly from the necessity direction of Theorem \ref{pmt}. For the sufficiency direction, note that if any of \textit{(i)-(iv)} hold, then by Theorem \ref{pmt} the sequence \(\boldsymbol{d}\) is a diagonal of \(T\). Thus we may assume that \textit{(v)} holds. In particular, this implies that \(T\) is diagonalizable. Moreover, an operator \(T'\) is unitarily equivalent modulo \(\{0,1\}\) to \(T\) if and only if \(T'\) is diagonalizable, and the eigenvalue lists of \(T\) and \(T'\) are identical apart from \(0\) and \(1\). Hence, to complete the proof, we must construct an operator \(T'\) with diagonal \(\boldsymbol{d}\) such that the eigenvalue lists of \(T\) and \(T'\) are identical apart from \(0\) and \(1\).

In the case that \(\delta_{0}<\delta^{L}\) and \(\delta_{1}<\delta^{U}\), the conclusion follows from Theorem \ref{suff}. Thus, we can assume either $\delta_0=\delta^L$ or $\delta^1=\delta^U$ (or both). By symmetry, it suffices to consider the case where \(\delta_{0}=\delta^{L}\).

\textbf{Case 1.} Suppose we have equality in \eqref{mainthm2} for some \(\alpha_{0}\in(0,1)\). Define the sequences \(\boldsymbol\lambda_{0}:=(\lambda_{i})_{\lambda_{i}<\alpha_{0}}\), \(\boldsymbol\lambda_{1}:=(1-\lambda_{i})_{\lambda_{i}\geq\alpha_{0}}\), \(\boldsymbol{d}_{0}:=(d_{i})_{d_{i}<\alpha_{0}}\), and \(\boldsymbol{d}_{1}:=(1-d_{i})_{d_{i}\geq\alpha_{0}}\). By Corollary \ref{intToComp1}, Corollary \ref{intToComp2}, and Proposition \ref{LR} we see that
\[(\boldsymbol d_{0})_{+}\prec (\boldsymbol\lambda_{0})_{+},\ (\boldsymbol d_{0})_{-}\prec (\boldsymbol\lambda_{0})_{-},\ (\boldsymbol d_{1})_{+}\prec (\boldsymbol\lambda_{1})_{+},\ (\boldsymbol d_{1})_{-}\prec (\boldsymbol\lambda_{1})_{-}\]
\[\liminf_{\alpha\nearrow 0}\delta(\alpha,\boldsymbol\lambda_{0},\boldsymbol d_{0}) = \delta^{L},\quad \liminf_{\alpha\searrow 0}\delta(\alpha,\boldsymbol\lambda_{0},\boldsymbol d_{0}) = \delta_{0},\]
\[\liminf_{\alpha\nearrow 0}\delta(\alpha,\boldsymbol\lambda_{1},\boldsymbol d_{1}) = \delta^{U},\quad\text{and}\quad \liminf_{\alpha\searrow 0}\delta(\alpha,\boldsymbol\lambda_{1},\boldsymbol d_{1}) = \delta_{1}.\]
Note that \(\delta_{0}= \delta^{L}\) and \(\delta_{1}\leq \delta^{U}\). Moreover, if \((\boldsymbol d_{1})_{-}\in\ell^{1}\), then since \(\delta^{U}<\infty\) we also have \((\boldsymbol\lambda_{1})_{-}\in\ell^{1}\) and hence \eqref{pmt5} implies that \(\delta^{U} = \delta_{1}\). Hence, by Theorem \ref{intropv2} there are compact self-adjoint operators \(T_{0},S_{0},T_{1},\) and \(S_{1}\) such that for \(i=0,1\)
\begin{itemize}
\item \(T_{i}\) has diagonal \(\boldsymbol d_{i}\),
\item \(S_{i}\) has eigenvalue list \(\boldsymbol\lambda_{i}\),
\item \(T_{i}\oplus\boldsymbol{0}\) and \(S_{i}\oplus\boldsymbol{0}\) are unitarily equivalent, where \(\boldsymbol{0}\) is the zero operator on an infinite-dimensional Hilbert space.
\end{itemize}
Finally, we note that the operator \(T'=T_{0}\oplus(\mathbf{I}-T_{1})\) is diagonalizable, has diagonal \(\boldsymbol d\), and both \(T\) and \(T'\)  both have identical eigenvalue lists apart from \(0\) and \(1\). This completes Case 1.

For the final three cases we will assume we have strict inequality in \eqref{mainthm2} for all \(\alpha\in(0,1)\). If \(\delta_{0},\delta_{1}>0\), then Theorem \ref{suff} gives the desired conclusion. Thus, we need only consider the case that \(\delta_{0}=0\) or \(\delta_{1}=0\) (or both).

\textbf{Case 2.} Suppose \(\delta_{0}=0\), and \(\delta_{1}>0\). This implies \(\delta^{L}=0\). By Proposition \ref{LR} \(\boldsymbol{d}_{-}\prec \boldsymbol\lambda_{-}\). From this we deduce that the number of negative terms in the sequence \(\boldsymbol{d}\) is at least the number of negative terms in \(\boldsymbol\lambda\). If \(\boldsymbol{d}\) has \(N<\infty\) negative terms, then by the Schur-Horn Theorem there is an operator \(T_{0}\) with diagonal \((d_{i})_{d_{i}<0}\) and eigenvalues \((-\lambda_{i}^{-\downarrow})_{i=1}^{N}\) (note that this is all of the negative eigenvalues of \(T\) and possibly some zeros). If \(\boldsymbol{d}\) has infinintely may negative terms, then by Proposition \ref{posSH} there is a self-adjoint operator \(S_{0}\) with eigenvalue list \((\boldsymbol\lambda_{-})^{\downarrow}\) and diagonal \((\boldsymbol{d}_{-})^{\downarrow}\). Set \(T_{0} = -S_{0}\). Thus, in either case, we have a diagonalizable negative operator \(T_{0}\) with diagonal \((d_{i})_{d_{i}<0}\) and eigenvalues consisting of \((\lambda_{i})_{\lambda_{i}<0}\) and possibly some zeros. By Theorem \ref{csuff} there is a self-adjoint operator \(T_{1}\) eigenvalue list \((\lambda_{i})_{\lambda_{i}>0}\) and diagonal \((d_{i})_{d_{i}>0}\). Finally, let \(\boldsymbol{0}\) be the zero operator on a space of dimension \(\#|\{i : d_{i}=0\}|\) and \(T'=T_{0}\oplus\boldsymbol{0}\oplus T_{1}\) is the desired operator.

\textbf{Case 3.} Suppose \(\delta_{0}>0\), and \(\delta_{1}=0\). By \eqref{mainthm5}, either \((\lambda_{i}-1)_{\lambda_{i}>1}\) is nonsummable, or \(\delta_{U}=\delta_{1}=0\). In the nonsummable case we apply Proposition \ref{kwd2d} to obtain a compact self-adjoint operator \(T_{1}\) with diagonal \((d_{i}-1)_{d_{i}>1}\) and eigenvalues \((\lambda_{i}-1)_{\lambda_{i}>1}\). If \(\delta^{U}=0\), then, as in the previous case, either by the Schur-Horn Theorem or Proposition \ref{posSH} there is a self-adjoint operator \(T_{1}\) with diagonal \((d_{i}-1)_{d_{i}>1}\) and eigenvalues \((\lambda_{i}-1)_{\lambda_{i}>1}\) and possibly some zeros. By Theorem \ref{csuff} there is a self-adjoint operator \(T_{0}\) eigenvalue list \((\lambda_{i})_{\lambda_{i}<1}\) and diagonal \((d_{i})_{d_{i}<1}\). Finally, let \(\mathbf{I}'\) be the identity operator on a space of dimension \(\#|\{i : d_{i}=1\}|\) and \(T'=T_{0}\oplus\mathbf{I}'\oplus (T_{1}+\mathbf{I})\) is the desired operator.

\textbf{Case 4.} Suppose \(\delta_{0}=\delta_{1}=0\).	Since \(\delta_{0}=0\) we can construct the operator \(T_{0}\) as in Case 2. Since \(\delta_{1}=0\) we can construct the operator \(T_{1}\) as in Case 3. Define the sequences \(\boldsymbol{d}_{I} = (d_{i})_{d_{i}\in(0,1)}\) and \(\boldsymbol{\lambda}_{I} = (\lambda_{i})_{\lambda_{i}\in(0,1)}\). By our assumption that \(\delta_{0}=\delta_{1} =0\) we see that \(\boldsymbol{d}_{I}\) and \(\boldsymbol\lambda_{I}\) satisfy the assumptions of Theorem \ref{archaic2}. Hence, there is a self-adjoint operator \(T_{I}\) with diagonal \(\boldsymbol{d}_{I}\) and eigenvalues consisting of \(\boldsymbol\lambda_{I}\), possibly with extra zeros and ones. As before, let \(\boldsymbol{0}\)  be the zero operator on a space of dimension \(\#|\{i : d_{i}=0\}|\), and let \(\mathbf{I}'\) be the identity operator on a space of dimension \(\#|\{i : d_{i}=1\}|\). The operator \(T'=T_{0}\oplus\boldsymbol{0}\oplus T_{I}\oplus \mathbf{I}'\oplus (\mathbf{I}+T_{1})\) is as desired.
\end{proof}

\section{Algorithm for determining diagonals of self-adjoint operators \\
with at least $3$-point essential spectrum} \label{S21}

In this section we give an algorithm for determining whether a given sequence is a diagonal of a self-adjoint operator $E$ with at least three points in essential spectrum. We introduce the concept of pruning, which allows us to reduce the problem to a simpler operator $E'$ with fewer points in its essential spectrum.

\begin{defn} Let \(E\) be a self-adjoint operator on a Hilbert space \(\Hil\) with projection-valued measure \(\pi\) such that 
\[E = \int_{\R}\lambda\,d\pi(\lambda).\]
The \textit{right essential spectrum} of \(E\) is the set
\[\sigma_{ess}^{+}(E): = \big\{\lambda\in\R : \operatorname{rank}\pi([\lambda,\lambda+\eps])=\infty\text{ for all }\eps>0\big\}.\]
Similarly, the \textit{left essential spectrum} of \(E\) is the set \(\sigma_{ess}^{-}(E) := \{\lambda : -\lambda\in\sigma_{ess}^{+}(-E)\}\).
\end{defn}

Note that \(\lambda\notin\sigma_{ess}^{+}(E)\) if and only if \(\dim\ker (E-\lambda\mathrm{I})<\infty\) and there exists \(\eps>0\) such that \((\lambda,\lambda+\eps)\cap\sigma(E)=\varnothing\).

\begin{thm}\label{mainthm2.}
Let $E$ be a self-adjoint operator on $\mathcal H$ such that $0$ and $1$ are two extreme points of the essential spectrum $\sigma_{ess}(E)$ of $E$, that is,
\begin{equation}\label{mainthm2.0}
\{0,1 \} \subset \sigma_{ess}(E) \subset [0,1].
\end{equation}
Let $\boldsymbol\lambda=(\lambda_{j})_{j\in J}$ be the list of all eigenvalues of $E$ with multiplicity (which is possibly an empty list). If \(\sigma_{ess}(E)=\{0,1\}\), then additionally assume \(f_{\boldsymbol\lambda}(1/2)=\infty\).
Let $\boldsymbol{d} = (d_{i})_{i\in \N}$. Define
\[\delta^{L} = \liminf_{\beta\nearrow 0}(f_{\boldsymbol\lambda}(\beta)-f_{\boldsymbol{d}}(\beta))\quad\text{and}\quad \delta^{U} = \liminf_{\beta\searrow 1}(f_{\boldsymbol\lambda}(\beta)-f_{\boldsymbol{d}}(\beta)).\]
{\rm (Necessity)} If $\boldsymbol{d}$  is a diagonal of $E$, then
\begin{equation}
\label{mainthm2.1}
 f_{\boldsymbol{d}}(\alpha)  \leq f_{\boldsymbol\lambda}(\alpha)  
\qquad\text{for all }\alpha\in\R\setminus[0,1],
\end{equation}
\begin{equation}
\label{mainthm2.2}
\delta^{L}=0\implies \#|\{i: d_{i}= 0\}|\leq \#|\{j:\lambda_{j}= 0\}| \text{ and } \#|\{i: d_{i}\leq 0\}|\leq \#|\{j:\lambda_{j}\leq 0\}|,
\end{equation}
\begin{equation}
\label{mainthm2.3}
\delta^{U}=0\implies \#|\{i: d_{i}= 1\}|\leq \#|\{j:\lambda_{j}= 1\}| \text{ and } \#|\{i: d_{i}\geq 1\}|\leq \#|\{j:\lambda_{j}\geq 1\}|,
\end{equation}
and one of the following four conditions holds:
\begin{enumerate}
\item
$\delta^L=\delta^U =\infty$,
\item
$\delta^L<\infty$, $\delta^U=\infty$, and $\sum_{d_i>0} d_i=\infty$, 
\item
$\delta^U<\infty$, $\delta^L=\infty$, and $\sum_{d_i<1} (1-d_i)=\infty$, or
\item 
$\delta^L+\delta^U <\infty$ and $f_{\boldsymbol d}(1/2) =\infty$.
\end{enumerate}

\noindent {\rm (Sufficiency)}  Conversely, if \eqref{mainthm2.1}--\eqref{mainthm2.3} hold, one of the conditions (i)--(iv) holds,
\begin{equation}\label{mainthm2.4}\delta^{L}=0\implies 0\in\sigma_{ess}^{+}(E),\quad\text{and}\quad \delta^{U}=0\implies 1\in\sigma_{ess}^{-}(E),\end{equation}
then \(\boldsymbol{d}\) is a diagonal of \(E\).
\end{thm}

\begin{proof}[Proof of Theorem \ref{mainthm2.}]
The necessity direction follows directly from the necessity direction of Theorem \ref{pmt}. Indeed, by our assumption $f_{\boldsymbol{\lambda}}(1/2)=\infty$, the conclusion (v) in Theorem \ref{pmt} cannot happen, whereas conclusions \eqref{mainthm2.2} and \eqref{mainthm2.3} follow from Proposition \ref{p211} and Lemma \ref{card}.

For the sufficiency direction, note that if any of \textit{(i)-(iv)} hold and $\delta^L,\delta^U>0$, then by the necessity direction of Theorem \ref{pmt} the sequence \(\boldsymbol{d}\) is a diagonal of \(E\). 

Next suppose that
 $\delta^L=0$. We claim that we can reduce to the case $d_i \ne 0$ for all $i$. Indeed, suppose momentarily that Theorem \ref{mainthm2.} is true for sequences $\boldsymbol d=(d_{i})_{i\in \N}$ such that $d_i \ne 0$ for all $i$. Take an arbitrary sequence $\boldsymbol d$ satisfying the assumptions of Theorem \ref{mainthm2.}. Let $I_0=\{i: d_i = 0\}$. If $I_0$ is finite, then by \eqref{mainthm2.2} there exists $J_0\subset J$ such that $\#|I_0|=\#|J_0|$ and $\lambda_j=0$ for all $j\in J_0$. If $I_0$ is infinite, then by \eqref{mainthm2.2} there exists an infinite subset $J_0\subset J$ such that $\lambda_j=0$ for all $j\in J_0$, $\lambda_j=0$ for infinitely many  $j\in J \setminus J_0$, and
 \begin{equation}\label{mtt1}
\#| \{i \in \N \setminus I_0: d_i \le 0\}| \le \# | \{j \in J \setminus J_0: \lambda_i \le 0 \}|.
\end{equation}
Either way, \eqref{mtt1} holds. Next we decompose the operator $E=\mathbf 0 \oplus E_0$, where $\mathbf 0$ is the zero operator on a space of dimension $\#|J_0|$, and $0 \in \sigma_{ess}(E_0)$.  Applying Theorem \ref{mainthm2.} in the case of nonzero diagonal sequences yields that  $(d_{i})_{i \in \N \setminus I_0}$ is a diagonal of the operator $E_0$. Clearly, $(d_{i})_{i \in I_0}$ is a diagonal of the operator $\mathbf 0$. Hence, $\boldsymbol d$ is a diagonal of $E$, which completes the reduction step. 
Hence, without loss of generality we can assume that $d_i \ne 0$ for all $i$. By a similar argument, we can also assume $d_i \ne 1$ for all $i$ when $\delta^U=0$.

By \eqref{mainthm2.1} and $\delta^L=0$, Proposition \ref{LR} yields $\boldsymbol d^\downarrow \preccurlyeq \boldsymbol \lambda^\downarrow$, and hence 
\[
\#|\{j:\lambda_{j}<0\}| \le \#|\{i: d_{i}< 0\}|.
\]
Thus, there exists a set $J_1$ satisfying
\[
\{j: \lambda_j < 0 \} \subset J_1 \subset \{j: \lambda_j \le 0\}
\]
such that $I_1=\{i: d_i < 0\}$ and $J_1$ have the same cardinality. Moreover, we can guarantee that \(\{j\in J\setminus J_{1} : \lambda_{j}=0\}\) is infinite if \(\{j\in J : \lambda_{j}=0\}\) is infinite.
We can decompose the operator $E=E_1 \oplus E_2$ such that $E_1$ is a diagonalizable operator with eigenvalue list $(\lambda_{j})_{j\in J_1}$. In particular, this implies that $\sigma(E_2) \subset [0,\infty)$ and by \eqref{mainthm2.4} we also have \(0\in\sigma_{ess}(E_{2})\).

Since $\delta^L=0$ either (ii) or (iv) holds. If (ii) holds, then by Theorem \ref{nbd}(b') we deduce that $(d_{i})_{i \in \N \setminus I_1}$ is a diagonal of $E_2$. Next, we assume that (iv) holds. If $\delta^U>0$, then by the symmetric variant of Lemma \ref{nbdL} we deduce that $(d_{i})_{i \in \N \setminus I_1}$ is a diagonal of $E_2$. If $\delta^U=0$, then we can also assume $d_i \ne 1$ for all $i$. Let $I_3 = \{i: d_i >1\}$ and $I_4=\{i: 0<d_i <1\}$. By a similar argument as in previous paragraph there exists a subset $J_3 \subset \{j: \lambda_j\ge 1\}$ and a decomposition of the operator $E_2=E_3 \oplus E_4$, such that $E_3$ is a diagonalizable operator with eigenvalue list $(\lambda_{j})_{j\in J_3}$, and $\sigma(E_4) \subset [0,1]$. By \eqref{mainthm2.4} we also have \(1\in\sigma_{ess}(E_{4})\), and hence \(\{0,1\}\subset\sigma_{ess}(E_{4})\). Since
\[
f_{\boldsymbol d}(1/2) = \sum_{i\in I_4} \min(d_i,1-d_i) = \infty,
\]
by Theorem \ref{tomilov} we have that $(d_{i})_{i \in I_4}$ is a diagonal of $E_4$. 

If $I_1$ is finite, then by the Schur-Horn theorem $(d_{i})_{i\in I_1}$ is a diagonal of $E_1$. If $I_1$ is infinite, then Proposition \ref{posSH} $(d_{i})_{i\in I_1}$ is a diagonal of $E_1$. Likewise, \((d_{i})_{i\in I_{3}}\) is a diagonal of \(E_{3}\). Hence, \(\boldsymbol{d}=(d_{i})_{i\in I_{1}\cup I_{3}\cup I_{4}}\) is a diagonal of \(E=E_{1}\oplus E_{3}\oplus E_{4}\).
\end{proof}

As an immediate corollary of Theorem \ref{mainthm2.} we have the following characterization of diagonals of a large class of operators. Part (i) of Corollary \ref{BE} was already shown by Loreaux and Weiss \cite[Theorem 4.6]{lw3}. 

\begin{cor}\label{BE} Let \(E\) be a self-adjoint operator on \(\Hil\). Set \(a=\inf\sigma_{ess}(E)\) and \(b=\sup\sigma_{ess}(E)\), and suppose that $\sigma(E)\subset [a,b]$. Assume that either
\begin{enumerate}
\item $\#|\sigma_{ess}(E)|\geq 3$, or
\item $\#|\sigma_{ess}(E)|=2$ and the list of eigenvalues $\boldsymbol \lambda= (\lambda_j)_{j\in J}$ of $E$ satisfies
\[\sum_{j\in J}\min(\lambda_j-a,b-\lambda_{j})=\infty.\]
\end{enumerate}
Let \(\boldsymbol{d}=(d_{i})_{i\in I}\) be a sequence in \([a,b]\). Then, \(\boldsymbol{d}\) is a diagonal of \(E\) if and only if
\[\#|\{i: d_{i}= a\}|\leq \dim\ker(E-a\mathbf{I}),\quad \#|\{i: d_{i}= b\}|\leq \dim\ker(E-b\mathbf{I}),\quad\text{and}\]
\[\sum_{i\in I}\min(d_{i}-a,b-d_{i})=\infty.\]

\end{cor}

To describe our algorithm, it is convenient to introduce the concept of a pruning of an operator.

\begin{defn} Let $E$ be a self-adjoint operator such that \eqref{mainthm2.0} holds. Let $\boldsymbol{d} = (d_{i})_{i\in \N}$. Suppose that 
\[
\delta^L=0 \qquad\text{and} \qquad 0 \not\in \sigma^+_{ess}(E).
\]
Let $\pi$ be the spectral measure of $E$.
A {\it pruning} of $E$ at $0$ is the pair $(E',\boldsymbol d')$, where $E'$ is the restriction of $E$ to the range of the projection $\pi([0,\infty))$ and $\boldsymbol d' = (d_{i})_{d_i \ge 0}$.
Likewise, if 
\[
\delta^U=0 \qquad\text{and}\qquad 1 \not\in \sigma^-_{ess}(E),
\]
a pruning of $E$ at $1$ is the pair $(E',\boldsymbol d')$, where $E'$ is the restriction of $E$ to the range of the projection $\pi((-\infty,1])$ and $\boldsymbol d' = (d_{i})_{d_i \le 1}$.
\end{defn}

The following lemma two describe what happens when the necessity condition in Theorem \ref{mainthm2.} holds, but sufficiency condition fails.

\begin{lem}\label{prune}
Let $E$ be a self-adjoint operator on $\mathcal H$ such that \eqref{mainthm2.0} holds. Let $\boldsymbol{d} = (d_{i})_{i\in \N}$ be such that \eqref{mainthm2.1}, \eqref{mainthm2.2}, and one of the four conditions (i)--(iv) in Theorem \ref{mainthm2.} hold. Suppose that 
\[
\delta^L=0 \qquad\text{and} \qquad 0 \not\in \sigma^+_{ess}(E).
\]
Let $(E',\boldsymbol d')$ be a pruning of $(E,\boldsymbol d)$ at 0. Let $\alpha_0 = \min \sigma_{ess}(E') $. Let $\boldsymbol\lambda'$ be the list of all eigenvalues of $E'$ with multiplicity (which is possibly an empty list). Suppose that
\[
{\delta'}^L:= \liminf_{\alpha\nearrow \alpha_0} \delta(\alpha,\boldsymbol\lambda',\boldsymbol d') \le 0.
\]
Then, $\boldsymbol{d}$ is a diagonal of $E$ $\iff$ $\boldsymbol{d}'$ is a diagonal of $E'$.
\end{lem}

\begin{proof} 
Note that $0 \not\in \sigma^+_{ess}(E)$ implies that $\alpha_0>0$.
Suppose first that $\boldsymbol{d}$ is a diagonal of $E$. Since $\delta^L=0$, the operator $E$ decouples at $0$ with respect to $\boldsymbol d$ as in Definition \ref{decouple}. In particular, the operator $E|_{\mathcal H_1}$, which is a restriction of $E'$, has diagonal $\boldsymbol d'$. The operators $E'$ and $E|_{\mathcal H_1}$ may only differ by the multiplicity of the eigenvalue 0. That is the eigenvalue list of $E|_{\mathcal H_1}$ is the same as that of $E'$ with the exception that it might contain fewer zero eigenvalues.
By Theorem \ref{cptschurv2} applied to  $E|_{\mathcal H_1}$, the lower excess of $E|_{\mathcal H_1}$ at $\alpha_0$ is $\ge 0$ and at most ${\delta'}^L$. This yields that ${\delta'}^L=0$ and $E|_{\mathcal H_1}=E'$. Hence, $\boldsymbol d'$ is a diagonal of $E'$. Conversely, suppose that $\boldsymbol d'$ is a diagonal of $E'$. By a special case of Theorem \ref{kp}, a characterization of diagonals of positive compact operators with trivial kernel due to Kaftal and Weiss \cite{kw}, $(d_{i})_{d_i<0}$ is a diagonal of the restriction of $E$ to the range of $\pi(-\infty,0)$. Hence, $\boldsymbol d$ is a diagonal of $E$.
\end{proof}

\begin{lem}\label{prune2}
Let $E$ be a self-adjoint operator on $\mathcal H$ such that $\#|\sigma_{ess}(E)| \ge 5$ and \eqref{mainthm2.0} holds. Let $\boldsymbol{d} = (d_{i})_{i\in \N}$ be such that \eqref{mainthm2.1}, \eqref{mainthm2.2}, \eqref{mainthm2.3}, and one of the four conditions (i)--(iv) in Theorem \ref{mainthm2.} hold. Suppose that 
\[
\delta^L=0 \quad\text{and} \quad 0 \not\in \sigma^+_{ess}(E)
\quad\text{and}\quad
\delta^U=0 \quad\text{and} \quad 1 \not\in \sigma^-_{ess}(E).
\]
Let $(E',\boldsymbol d')$ be a pruning of $(E,\boldsymbol d)$ both at 0 and $1$. Let $\alpha_0 = \min \sigma_{ess}(E')$ and $\alpha_1 = \max \sigma_{ess}(E')$. Let $\boldsymbol\lambda'$ be the list of all eigenvalues of $E'$ with multiplicity (which is possibly an empty list). Suppose that 
\[
{\delta'}^L:= \liminf_{\alpha\nearrow \alpha_0} \delta(\alpha,\boldsymbol\lambda',\boldsymbol d') > 0
\quad\text{and}\quad
{\delta'}^U:= \liminf_{\alpha\searrow \alpha_1} \delta(\alpha,\boldsymbol\lambda',\boldsymbol d') > 0.
\]
Then, $\boldsymbol{d}$ is a diagonal of $E$ $\iff$ $\boldsymbol{d}'$ is a diagonal of $E'$.
\end{lem}

\begin{proof}
Note that the assumptions $0 \not\in \sigma^+_{ess}(E)$, $1 \not\in \sigma^-_{ess}(E)$, and $\#|\sigma_{ess}(E)| \ge 5$, imply that $0<\alpha_0<\alpha_1<1$. 
Suppose first that $\boldsymbol{d}$ is a diagonal of $E$ with respect to an orthonormal basis $(e_{i})$ of $\mathcal H$. Since $\delta^L=\delta^U=0$, the operator $E$ decouples at $0$ and at $1$ with respect to $\boldsymbol d$ as in Definition \ref{decouple}. That is, the spaces 
\[
\mathcal H_0=\overline{\lspan }\{ e_i: d_i < 0 \}, \qquad \mathcal H_{[0,1]}=\overline{\lspan} \{ e_i: 0 \le d_i \le 1 \},
\qquad \mathcal H_{1}=\overline{\lspan} \{ e_i:  d_i \ge 1 \}
\]
are invariant subspaces of $E$ and
\[
\sigma(E|_{\mathcal H_0}) \subset (-\infty,0]
,\qquad
\sigma(E|_{\mathcal H_{[0,1]}}) \subset [0,1]
,\qquad\sigma(E|_{\mathcal H_1}) \subset [1,\infty).
\]
In particular, the operator $E|_{\mathcal H_{[0,1]}}$ has diagonal $\boldsymbol d'$. The operators $E'$ and $E|_{\mathcal H_{[0,1]}}$ may only differ by the multiplicity of the eigenvalues 0 and 1. That is, the eigenvalue list of $E|_{\mathcal H_{[0,1]}}$ is the same as that of $E'$ with the exception that it might contain fewer eigenvalues equal to $0$ or $1$. Hence, both the lower excess of $E'$ at $0$ and the upper excess of $E'$ at $1$ are positive. Since the normalization of the operator $E|_{\mathcal H_{[0,1]}}$ satisfies the necessary conditions of Theorem \ref{mainthm2.} for the normalization of the diagonal $\boldsymbol d'$, so does the normalization of the operator $E'$. Since both upper and lower excesses of $E'$ are positive and $\#|\sigma_{ess}(E')| \ge 3$, by the sufficiency of Theorem \ref{mainthm2.}, the sequence $\boldsymbol d'$ is a diagonal of $E'$. 

Conversely, suppose that $\boldsymbol d'$ is a diagonal of $E'$. By a special case of Theorem \ref{kp}, a characterization of diagonals of positive compact operators with trivial kernel due to Kaftal and Weiss \cite{kw}, $(d_{i})_{d_i<0}$ is a diagonal of the restriction of $E$ to the range of $\pi(-\infty,0)$. Likewise, $(d_{i})_{d_i>1}$ is a diagonal of the restriction of $E$ to the range of $\pi(1,\infty)$. Hence, $\boldsymbol d$ is a diagonal of $E$.
\end{proof}

\begin{figure}
\centering
\resizebox{5.8in}{!}{
\begin{tikzpicture}[node distance=2cm]

\node (input) [io] {Input: self-adjoint operator $E$ such that $\#|\sigma_{ess}(E)| \ge 5 $ and a sequence $(d_{i})$};

\node (pro7) [process, below of=input] {Normalize $E$ so that $\{0,1\} \subset \sigma_{ess}(E) \subset [0,1]$};
\node (dec3a) [decision, below of=pro7, yshift=-1cm] {$\!\#|\sigma_{ess}(E)|\! \ge \! 5\!$};
\node (pro8) [decision, right of=dec3a, xshift=3cm] {(Necessity) Theorem \ref{mainthm2.}};
\node (dec5) [decision, below of=pro8, yshift=-3cm] {$\delta^L=0$,  $0 \not\in \sigma^+_{ess}(E)$, and ${\delta'}^L \le 0$};
\node (no3) [no, below of=dec3a, yshift=-3cm] {$(d_{i})$ is not a \\ diagonal of $E$};
\node (pro9) [process, right of=dec5, xshift=1.5cm, yshift=2.5cm] {Prune $E$ at $0$~or~$1$, respectively};
\node (dec6) [decision, below of=dec5, yshift=-3cm] {$\delta^U=0$, $1 \not\in \sigma^-_{ess}(E)$, and ${\delta'}^U \le 0$};
\node (dec7) [decision, left of=dec6, xshift=-3cm] {$\delta^L=0$ and $0 \not\in \sigma^+_{ess}(E)$};
\node (dec8) [decision, below of=dec7, yshift=-2.5cm] {$\delta^U=0$ and $1 \not\in \sigma^-_{ess}(E)$};
\node (pro12) [process, left of=dec7, xshift=-3cm, yshift=-2cm] {Prune $E$ at $0$ and normalize $E$ so that $\{0,1\} \subset \sigma_{ess}(E) \subset [0,1]$};
\node (pro13) [process, below of=pro12, yshift=-2.5cm] {Prune $E$ at $1$ and normalize $E$ so that $\{0,1\} \subset \sigma_{ess}(E) \subset [0,1]$};
\node (dec13) [decision, below of=dec8, yshift=-2.5cm] {(Necessity) Theorem \ref{mainthm2.}};
\node (no4) [no, left of=dec13, xshift=-3cm] {$(d_{i})$ is not a \\ diagonal of $E$};
\node (pro11) [stop, right of=dec13, xshift=3cm] {$(d_{i})$ is a \\ diagonal of $E$};

\node (output) [kernel, left of=dec3a, xshift=-3cm] {Apply 4pt algorithm};

\draw [arrow] (input) -- (pro7);
\draw [arrow] (pro7) -- (dec3a);
\draw [arrow] (dec3a) -- node [yshift=0.25cm] {yes} (pro8);
\draw [arrow] (dec3a) -- node [yshift=0.25cm] {no} (output);
\draw [arrow] (pro8) -- node [xshift=0.45cm] {yes} (dec5);
\draw [arrow] (dec5) -- node [xshift=0.45cm] {no} (dec6);
\draw [arrow] (dec5) -| node [xshift=-1cm, yshift=0.25cm] {yes} (pro9);
\draw [arrow] (pro9) |-  (pro7);

\draw [arrow] (dec6) -- node [yshift=0.25cm] {no} (dec7);
\draw [arrow] (dec6) -| node [xshift=-1cm, yshift=0.25cm] {yes} (pro9);

\draw [arrow] (dec7) -| node [yshift=0.25cm, xshift=2cm] {yes} (pro12);
\draw [arrow] (dec7) -- node [xshift=0.45cm] {no} (dec8);
\draw [arrow] (pro12) -- (dec8);
\draw [arrow] (dec8) -| node [yshift=0.25cm, xshift=2cm] {yes} (pro13);
\draw [arrow] (dec8) -- node [xshift=0.45cm] {no} (dec13);
\draw [arrow] (pro13) -- (dec13);
\draw [arrow] (dec13) -- node [yshift=0.25cm] {yes} (pro11);
\draw [arrow] (dec13) -- node [yshift=0.25cm] {no} (no4);

\draw [arrow] (pro8) -- node [yshift=0.25cm] {no} (no3);

\end{tikzpicture}
}

\caption{Algorithm for pruning operators with $\ge 5$-point essential spectrum by countably transfinite recursive process.}
\label{diagram4}
\end{figure}
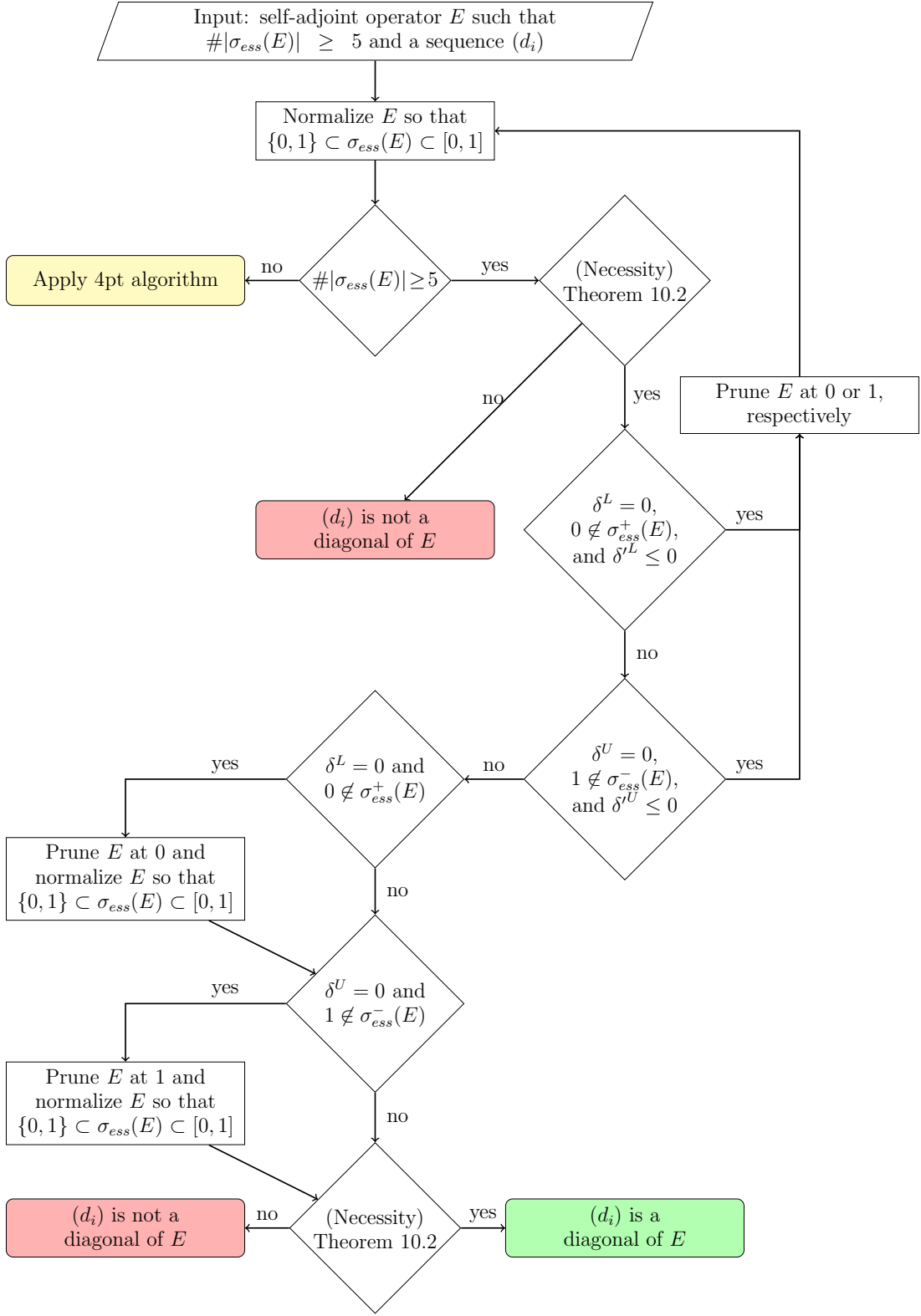

\begin{figure}
\resizebox{5in}{!}{
\begin{tikzpicture}[node distance=2cm]

\node (input) [io] {Input: self-adjoint operator $E$ such that $\#|\sigma_{ess}(E)|=k$, where $k=3,4$, and a sequence $(d_{i})$};
\node (pro7) [process, below of=input] {Normalize $E$ so that $\{0,1\} \subset \sigma_{ess}(E) \subset [0,1]$};
\node (pro8) [decision, below of=pro7,  yshift=-1cm] {(Necessity) Theorem \ref{mainthm2.}};
\node (no3) [no, right of=pro8, xshift=3cm] {$(d_{i})$ is not a \\ diagonal of $E$};
\node (dec5) [decision, below of=pro8, yshift=-2.5cm] {$\delta^L=0$ and $0 \not\in \sigma^+_{ess}(E)$};
\node (pro9) [process, left of=dec5, xshift=-3cm, yshift=-2.25cm] {Split into compact and $(k-1)$ point essential spectrum operators};
\node (pro12) [kernel, below of=pro9, xshift=0cm, yshift=-0.5cm] {Compact kernel problem};
\node (pro13) [kernel, below of=pro9, xshift=0cm, yshift=-1.5cm] {Apply $(k-1)$pt algorithm};
\node (dec6) [decision, below of=dec5, yshift=-2.5cm] {$\delta^U=0$ and $ 1 \not\in \sigma^-_{ess}(E)$};
\node (pro11) [stop, right of=dec6, xshift=3cm] {$(d_{i})$ is a \\ diagonal of $E$};

\draw [arrow] (input) -- (pro7);
\draw [arrow] (pro7) -- (pro8);
\draw [arrow] (pro8) -- node [xshift=0.45cm] {yes} (dec5);
\draw [arrow] (pro8) -- node [yshift=0.25cm] {no} (no3);
\draw [arrow] (dec5) -- node [xshift=0.45cm] {no} (dec6);
\draw [arrow] (dec5) -- node [yshift=0.25cm] {yes} (pro9);
\draw [arrow] (pro9) -- (pro12);

\draw [arrow] (dec6) -- node [yshift=0.25cm] {no} (pro11);
\draw [arrow] (dec6) -- node [yshift=0.25cm] {yes} (pro9);

\end{tikzpicture}
}
\caption{Algorithm for operators with $3$-point or $4$-point essential spectrum.}
\label{diagram3}
\end{figure}
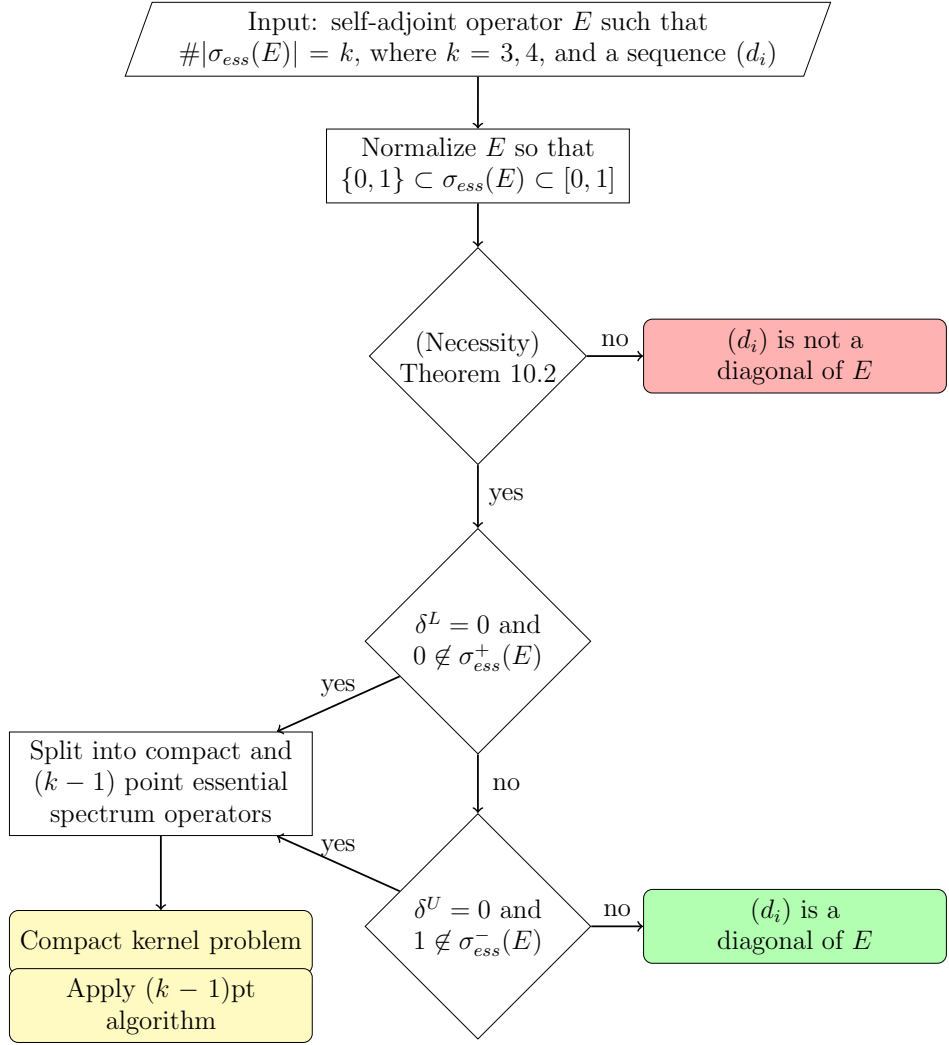

The algorithm for determining whether a sequence $\boldsymbol d$ is a diagonal of an operator $E$ with $\#|\sigma_{ess}(E)| \ge 5$ is represented by Figure \ref{diagram4}. This procedure starts by normalizing $E$ so that \eqref{mainthm2.0} holds. Then we check whether all necessary conditions \eqref{mainthm2.1}--\eqref{mainthm2.3} and one of the conditions (i)--(iv) in Theorem \ref{mainthm2.} are satisfied. If not, then $\boldsymbol d$ is not a diagonal of $E$. Otherwise, we check whether the three conditions $\delta^L=0$, $0 \not \in \sigma^+_{ess}(E)$, and $\delta'^L\le 0$ from Lemma \ref{prune} all hold. If yes, then we can prune an operator $E$ at $0$, and then repeat the process of normalizing and checking necessary conditions in Theorem \ref{mainthm2.} as long as we have at least 5 points in the essential spectrum. Since the minimum of the essential spectrum of $E$ (without normalizing) increases with each pruning, this iterative procedure has to stop after a countable number of steps. This can be described by countably transfinite induction procedure. Using a symmetric variant of Lemma \ref{prune}, we also run the procedure of pruning operator $E$ from the right.

After countable number of pruning steps we have a trichotomy:
\begin{enumerate}
\item either the operator $E$ has exactly 4 points in its essential spectrum, in which case we need to apply the 4pt algorithm described in Figure \ref{diagram3}, 
\item the necessity in Theorem \ref{mainthm2.} fails, in which case $\boldsymbol d$ is not a diagonal of $E$, or 
\item no more pruning as in Lemma \ref{prune} can be done, in which case the resulting pair $(E, \boldsymbol d)$ has the property that $\boldsymbol d$ is diagonal of $E$ if and only the same is true for the original input operator and sequence.
\end{enumerate}
In the last case, we will perform up to two additional prunings. That is, we check if $\delta^L=0$ and $0 \not \in \sigma^+_{ess}(E)$. If yes, then we prune operator $E$ at $0$, renormalize $E$, and observe that the new lower excess $\delta^L>0$. If no, we leave $E$ intact. 
Likewise, we check if $\delta^U=0$ and $1 \not \in \sigma^-_{ess}(E)$. If yes, then we prune operator $E$ at $1$, renormalize $E$, and observe that the new upper excess $\delta^U>0$. If no, we leave $E$ intact. As a result we obtain a normalized operator $E'$ with at least 3 points in its essential spectrum and a new sequence $\boldsymbol d'$, which satisfies 
\begin{equation}\label{mainthm2.5}{\delta'}^{L}=0\implies 0\in\sigma_{ess}^{+}(E'),\quad\text{and}\quad {\delta'}^{U}=0\implies 1\in\sigma_{ess}^{-}(E'),\end{equation}
where ${\delta'}^L$ and ${\delta'}^U$ are excesses corresponding to $(E',\boldsymbol d')$.
  We claim that $\boldsymbol d$ is a diagonal of $E$ if and only if $\boldsymbol d'$ is a diagonal of $E'$. Lemma \ref{prune2} shows this claim in the case when $E'$ is obtained by pruning of $E$ at both $0$ and $1$. The case when pruning occurs only on one side follows in a similar way. Since both lower and upper excess of $E'$ are positive, the necessary conditions in Theorem \ref{mainthm2.} are also sufficient. Hence, the algorithm concludes by giving a definite answer whether the $\boldsymbol d$ is a diagonal of $E$ or not.

The algorithm for determining whether a sequence $\boldsymbol d$ is a diagonal of an operator $E$ with $\#|\sigma_{ess}(E)| = 3$, $4$ is represented by Figure \ref{diagram3}. This procedure starts by normalizing $E$ so that \eqref{mainthm2.0} holds. Then the algorithm splits into two similar cases depending whether we have $3$ or $4$ points in the essential spectrum.

In the latter case, we check whether all necessary conditions \eqref{mainthm2.1}--\eqref{mainthm2.3} and one of the conditions (i)--(iv) in Theorem \ref{mainthm2.} are satisfied. If not, then $\boldsymbol d$ is not a diagonal of $E$. Otherwise, $\boldsymbol d$ is a diagonal of $E$ provided that \eqref{mainthm2.4} holds. In the case that both $\delta^L=0$ and $0 \not \in \sigma^+_{ess}(E)$ hold, Theorem \ref{isplit} applies. In particular, we look for all possible splittings of $E$ into $E_1$ and $E_2$ at $0$ with $z_0=\#|\{i: d_i=0\}|$ and apply both Theorem \ref{kp} to test whether    
$(d_{i})_{d_i<0}$ is a diagonal of $E_1$ and the 3pt algorithm to test whether $(d_{i})_{d_i>0}$ is a diagonal of $E_2$. 
If the necessary condition in Theorem \ref{kp} fail or the 3pt algorithm gives negative answer for all possible splittings, then $\boldsymbol d$ is not a diagonal of $E$. On the other hand, if the sufficient condition in Theorem \ref{kp} holds and the 3pt algorithm gives affirmative answer for some splitting, then $\boldsymbol d$ is a diagonal of $E$. This leaves out the possibility that for some splittings the 3pt algorithm is inconclusive (this only happens if the 2pt algorithm needs to invoked) or we encounter the kernel problem for positive compact operators. Since $\dim \ker E<\infty$, the algorithm requires analyzing only a finite number of splittings of $E$, if any. We deal in a similar way in the case when $\delta^U=0$ and $1 \not \in \sigma^-_{ess}(E)$ hold. This leaves out the final case when \eqref{mainthm2.4} holds, where Theorem \ref{mainthm2.} yields that $\boldsymbol d$ is a diagonal of $E$.

The algorithm for \(3\)-point essential spectrum is performed in the same way with the exception that in the case $\delta^L=0$ and $0 \not \in \sigma^+_{ess}(E)$ or the case $\delta^U=0$ and $1 \not \in \sigma^-_{ess}(E)$, we need to consider all possible splittings and apply Theorem \ref{kp} and the 2pt algorithm. Hence, the 3pt algorithm is inconclusive only when we encounter the kernel problem for positive compact operators or the 2pt algorithm is inconclusive. Likewise, the 4pt algorithm is inconclusive only when we encounter the kernel problem for positive compact operators or the 3pt algorithm is inconclusive. Finally, the $\ge 5$pt algorithm is inconclusive only when it reduces to the 4pt algorithm. In fact, the following theorem shows that a combination of algorithms in Figures \ref{diagram4} and \ref{diagram3} for operators with $\ge 3$ points in their essential spectrum is conclusive for a large class of operators.

\begin{thm}\label{concl}
Let $E$ be a self-adjoint operator such that
\begin{equation}\label{concl0}
\exists \ \lambda \in \sigma_{ess}(E) \qquad \inf \sigma_{ess}^+(E)<\lambda< \sup \sigma_{ess}^-(E).
\end{equation}
Let $\boldsymbol d$ be a bounded sequence of real numbers. Then, the algorithms in Figures \ref{diagram4} and \ref{diagram3} yield a definite answer whether $\boldsymbol d$ is a diagonal of $E$, or not.
\end{thm}

To show this result we need and analogue of Lemma \ref{prune2}.

\begin{lem}\label{prune3}
Let $E$ be a self-adjoint operator on $\mathcal H$ such that $\#|\sigma_{ess}(E)| = 4$ and \eqref{mainthm2.0} and \eqref{concl0} hold. Let $\boldsymbol{d} = (d_{i})_{i\in \N}$ be such that \eqref{mainthm2.1}, \eqref{mainthm2.2}, \eqref{mainthm2.3}, and one of the four conditions (i)--(iv) in Theorem \ref{mainthm2.} hold. Suppose that 
\begin{equation}\label{prune9}
\delta^L=0 \quad\text{and} \quad 0 \not\in \sigma^+_{ess}(E).
\end{equation}
Let $(E',\boldsymbol d')$ be a pruning of $(E,\boldsymbol d)$ at $0$. Let $\alpha_0 = \inf \sigma_{ess}(E')$. Suppose that $\alpha_0 \in \sigma^+_{ess}(E)$.
Then, $\boldsymbol{d}$ is a diagonal of $E$ $\iff$ $\boldsymbol{d}'$ is a diagonal of $E'$.
\end{lem}

\begin{proof}
Note that the condition \eqref{mainthm2.0} and  $0 \not\in \sigma^+_{ess}(E)$ imply that $0<\alpha_0<1$. 
Suppose first that $\boldsymbol{d}$ is a diagonal of $E$ with respect to an orthonormal basis $(e_{i})$ of $\mathcal H$. Since $\delta^L=0$, the operator $E$ decouples at $0$ with respect to $\boldsymbol d$ as in Definition \ref{decouple}. In particular, the operator $E|_{\mathcal H_1}$ has diagonal $\boldsymbol d'$. The operators $E'$ and $E|_{\mathcal H_1}$ may only differ by the multiplicity of the eigenvalue 0. 
Since the normalization of the operator $E|_{\mathcal H_1}$ satisfies the necessary conditions of Theorem \ref{mainthm2.} for the normalization of the diagonal $\boldsymbol d'$, so does the normalization of the operator $E'$. Since both $\alpha_0 \in \sigma_{ess}^+(E')$, $1 \in \sigma_{ess}^-(E')$, and $\#|\sigma_{ess}(E')| = 3$, by the sufficiency of Theorem \ref{mainthm2.}, the sequence $\boldsymbol d'$ is a diagonal of $E'$. 

Conversely, suppose that $\boldsymbol d'$ is a diagonal of $E'$. By a special case of Theorem \ref{kp}, a characterization of diagonals of positive compact operators with trivial kernel due to Kaftal and Weiss \cite{kw}, $(d_{i})_{d_i<0}$ is a diagonal of the restriction of $E$ to the range of $\pi(-\infty,0)$. Hence, $\boldsymbol d$ is a diagonal of $E$.
\end{proof}

\begin{proof}[Proof of Theorem \ref{concl}] Observe that the $\ge 5$pt algorithm in Figure \ref{diagram4} is inconclusive only when it reduces to the 4pt algorithm. By the assumption \eqref{concl0}, the 3pt algorithm in Figure \ref{diagram3} is conclusive. Finally, the 4pt algorithm might be inconclusive only when \eqref{prune9} or its symmetric variant at $1$ holds. However, Lemma \ref{prune3} (or its symmetric variant) shows that the question whether $\boldsymbol d$ is a diagonal of $E$ reduces to the same question for the pruning of $E$ at $0$, which is answered by the 3pt algorithm. Either way, our algorithms yield a definite answer for operators satisfying \eqref{concl0}.
\end{proof}

\begin{remark}\label{concl5} As a corollary of Theorem \ref{concl}, the $\ge 5$pt algorithm is conclusive for the class of self-adjoint operators with uncountable spectrum, which includes all non-diagonalizable self-adjoint operators. Indeed, if a self-adjoint operator $E$ has uncountable spectrum, then $\sigma(E)$ has cardinality of the continuum $\mathfrak c$, and so does the essential spectrum $\sigma_{ess}(E)$. Consequently, the set $\sigma_{ess}^+(E) \cap \sigma_{ess}^-(E)$ has cardinality of the continuum $\mathfrak c$, and \eqref{concl0} holds.
\end{remark}

We finish this section by illustrating the necessity of transfinite recursive pruning in the algorithm in Figure \ref{diagram4}.

\begin{ex} Let $\boldsymbol \lambda, \boldsymbol d \in c_0^+$ be two sequences in $[0,1]$ such that the necessity conditions \eqref{kp1}--\eqref{kp4a} hold, but the sufficiency condition \eqref{kp4b} fails with $z=1$. That is, Theorem \ref{kp} can not be applied to deduce that there exists a compact operator $E$ with the eigenvalue list $\boldsymbol \lambda$ and diagonal $\boldsymbol d$. In other words, we encounter the kernel problem for compact positive operators with 1-dimensional kernel. Next, we shall add additional terms to $\boldsymbol \lambda$ and $\boldsymbol d$ to obfuscate the presence of the kernel problem.

By the von Neumann definition of ordinals, an ordinal $\alpha$ is a well ordered set of smaller ordinals $\beta<\alpha$. Let $\alpha$ be a countable limit ordinal. Let $\pi: \N \to \alpha$ be a bijection. Define a sequence $\boldsymbol \nu=(\nu_\beta)_{\beta<\alpha}$ by 
\[
\nu_\beta = 2- \sum_{n\in \N: \ \pi(n) < \beta}
2^{-n}.
\]
Hence, $\boldsymbol \nu$ is an embedding of ordinal $\alpha$ as a subset of the interval $[1,2]$ reversing the usual order $<$ of reals. 
Moreover, the limit points of the sequence $\boldsymbol \nu$ are in one-to-one correspondence with limit ordinals $\le \alpha$. We shall consider sequences obtained by concatenation of $\boldsymbol \lambda$ and $\boldsymbol  d$ with the sequence $\boldsymbol \nu$. 

Now suppose that our ordinal $\alpha$ is quite large, say $\alpha=\omega^\omega$. Let $E$ be a self-adjoint operator with eigenvalue list $\boldsymbol\lambda \oplus \boldsymbol \nu$. Hence, $\#|\sigma_{ess}(E)|=\infty$.
When we run the algorithm in Figure \ref{diagram4} for $E$ and the sequence $\boldsymbol d \oplus \boldsymbol \nu$, after normalization we see that all necessity conditions in Theorem \ref{mainthm2.} hold. While $\delta^L=0$, we have that $0\in \sigma^+_{ess}(E)$, and hence no pruning at $0$ is allowed. However, $\delta^U=0$, $1\not\in \sigma^-_{ess}(E)$,  and the upper excess of a pruning of $E$ at $1$ is zero, $\delta'^U=0$. Hence, we prune the operator $E$ at $1$, which corresponds to reducing the sequence $\boldsymbol \nu$ to $\boldsymbol \nu'= (\nu_\beta)_{\omega\le \beta<\alpha}$. Hence, up to a scaling factor the resulting pair is an operator $E$ with eigenvalue list $\boldsymbol\lambda \oplus \boldsymbol \nu'$ and the sequence $\boldsymbol d \oplus \boldsymbol \nu'$.

The algorithm takes this pair and then feeds it as an input for another run of the loop. The result of this run is an operator $E$ with eigenvalue list $\boldsymbol\lambda \oplus \boldsymbol \nu''$ and the sequence $\boldsymbol d \oplus \boldsymbol \nu''$, where $\boldsymbol \nu''= (\nu_\beta)_{2\omega\le \beta<\alpha}$. This is done infinitely many times until we are reduced to original pair $\boldsymbol \lambda$ and $\boldsymbol d$. If $\alpha=\omega^\omega$, only the algorithm in Figure \ref{diagram4} is emplyed. However, if $\alpha=\omega^\omega+3\omega$, then after $\omega$ runs through the loop, we are left out with the pair $\boldsymbol\lambda \oplus \boldsymbol \nu^\omega$ and $\boldsymbol d \oplus \boldsymbol \nu^\omega$, where $\boldsymbol \nu^\omega = (\nu_\beta)_{\omega^\omega\le \beta<\alpha}$. Since $\boldsymbol\lambda \oplus \boldsymbol \nu^\omega$ has 4 accumulation points we need to run through 4pt, then 3pt, and finally 2pt algorithms in Figures \ref{diagram3} and \ref{diagram2} at the final stages since pruning needs to be replaced by splitting. The compact kernel problem in Figure \ref{diagram3} does not arise at each splitting since sequences $\boldsymbol \lambda$ and $\boldsymbol d$ are being concatenated by the same sequence. Likewise, in 2pt algorithm in Figure \ref{diagram2} we observe that $(\mathfrak d_1)$ holds and we split $E$ into two compact operators for which 1pt Algorithm \cite[Figure 1]{mbjj7} and Theorem \ref{kp} are used, respectively. This is actually the special case of the splitting marked by ${}^\S$. However, the kernel problem in Theorem \ref{kp} is mute since sequences $\boldsymbol \lambda$ and $\boldsymbol d$ are  concatenated by the same sequence. The end result is the original pair $\boldsymbol \lambda$ and $\boldsymbol d$ for which 1pt algorithm is synonymous with the unsolved kernel problem. 


This illustrates how our algorithms process a pair of sequences $\boldsymbol \lambda$ and $\boldsymbol d$ into simpler subsequences until no more processing can be done and we either have a definite answer or we encounter the kernel problem. 
\end{ex}

\section{Algorithm for determining diagonals of self-adjoint operators \\
with \(2\)-point essential spectrum}

In this section we give an algorithm for determining whether a given sequence is a diagonal of a self-adjoint operator $E$ with exactly two points in its essential spectrum. 

\begin{thm}\label{pmtp3.}
Let $E$ be a self-adjoint operator on $\mathcal H$ such that its spectrum satisfies 
\[
\sigma(E) \subset (-\infty, 1]
\qquad\text{and}\qquad
\sigma_{ess}(E)=\{0,1\}.
\]
Suppose that $f_{\boldsymbol \lambda}(1/2) <\infty$, where
 $\boldsymbol\lambda=(\lambda_{j})_{j\in J}$ is the list of all eigenvalues of $E$ with multiplicity. 
Let $\boldsymbol{d} = (d_{i})_{i\in \N}$ be a sequence in $(-\infty,1]$. Define
\[\delta^{L} = \liminf_{\beta\nearrow 0}(f_{\boldsymbol\lambda}(\beta)-f_{\boldsymbol{d}}(\beta)).\]

\noindent {\rm (Necessity)} If $\boldsymbol{d}$  is a diagonal of $E$, then
\begin{align}
\label{pmtp3.0}
\#|\{i: d_i=1\}| & \le \#|\{j: \lambda_j=1\}|,\\
\label{pmtp3.1}
 f_{\boldsymbol{d}}(\alpha)  & \leq f_{\boldsymbol\lambda}(\alpha)  
\qquad\text{for all }\alpha<0,
\end{align}
and one of the following three conditions holds:
\begin{enumerate}
\item
$\delta^L=\infty$ and $\sum_{d_i<1} (1-d_i)=\infty$,
\item 
$\delta^L <\infty$ and $f_{\boldsymbol d}(1/2) =\infty$, or
\item
$\delta^L<\infty$,  $f_{\boldsymbol d}(1/2) <\infty$, $\boldsymbol d$ has accumulation points at $0$ and $1$, and there exists $\delta_0 \in [0,\delta^L]$ such that:
\begin{align}
\label{pmtp3.2}
f_{\boldsymbol\lambda}(\alpha)  &\leq f_{\boldsymbol{d}}(\alpha) + (1-\alpha)\delta_{0}  & \text{for all }\alpha\in(0,1),\\
\label{pmtp3.3}f_{\boldsymbol\lambda}'(\alpha)-f_{\boldsymbol{d}}'(\alpha)  & \equiv -\delta_0 \mod 1 & \text{for some }\alpha\in(0,1),
\\
\label{pmtp3.4}
\sum_{\lambda_i<0} |\lambda_i| &<\infty \implies
\delta^L=\delta_0.
\end{align}
\end{enumerate}

In addition, the operator $E$ decouples at a point $\alpha\in [0,1)$ if \((\mathfrak{d}_{\alpha}')\) holds, where
\begin{enumerate}
\item[$(\mathfrak d'_0)$] $\delta^L=0$ \quad or \quad $\delta_0=0$, 
\item[$(\mathfrak d'_\alpha)$] $\alpha \in (0,1)$ and $f_{\boldsymbol\lambda}(\alpha)  = f_{\boldsymbol{d}}(\alpha) + (1-\alpha)\delta_{0}$.
\end{enumerate}

\noindent {\rm (Sufficiency)}  Conversely, if \eqref{pmtp3.0} and \eqref{pmtp3.1} hold, $(\mathfrak d'_\alpha)$ does {\bf not} hold for any $\alpha \in [0,1)$, and either 
\begin{itemize}
\item (i) or (ii) holds, or
\item (iii) holds, $\lambda_j<1$ for all $j$,
\end{itemize}
then $\boldsymbol{d}$  is a diagonal of $E$.
\end{thm}

\begin{remark} 

Note that the number \(\delta_{0}\) is only defined when (iii) holds. Hence, in the case that (iii) does not hold, the statements \((\mathfrak{d}_{\alpha}')\) must be properly interpreted using the following convention. If (i) or (ii) holds, then  we interpret any equality involving \(\delta_{0}\) in the statements \((\mathfrak{d}_{\alpha}')\) as false. With this convention, we note that if (i) or (ii) holds, then \((\mathfrak{d}_{\alpha}')\) is false for all \(\alpha\in(0,1)\), whereas the statement \((\mathfrak{d}_{0}')\) becomes \(\delta^{L}=0\).

\end{remark}

\begin{proof}
The necessity part of Theorem \ref{pmtp3.} follows automatically from Theorem \ref{pmt}. The exception is \eqref{pmtp3.0} which follows from the fact that any basis vector corresponding to diagonal $d_i=1$ is an eigenvector of $E$.

For the sufficiency suppose that \eqref{pmtp3.0} and \eqref{pmtp3.1} hold, and \((\mathfrak{d}_{\alpha}')\) does not hold for any \(\alpha\in[0,1)\). By (i)--(iii) we have that $d_i<1$ for infinitely many $i$.  Let \(\boldsymbol{d}'=(d_{i})_{d_{i}<1}\). If \(\#|\{i : d_{i}=1\}|<\infty\), let \(\boldsymbol\lambda'\) be a sequence with all of the terms of \((\lambda_{i})_{\lambda_{i}<1}\) together with \(\#|\{i : \lambda_{i}=1\}| - \#|\{i : d_{i}=1\}|\) terms equal to \(1\). If \(\#|\{i : d_{i}=1\}|=\infty\), then set \(\boldsymbol{\lambda}' = \boldsymbol{\lambda}\). If \(\boldsymbol{d}'\) is a diagonal of a self-adjoint operator \(E\) with eigenvalue list \(\boldsymbol{\lambda}'\), then \(E\oplus\mathbf{I}\) has eigenvalue list \(\boldsymbol\lambda\) and diagonal \(\boldsymbol{d}\), where \(\mathbf{I}\) is the identity on a space of dimension \(\#|\{i : d_{i}=1\}|\). Note that \(1\) is an accumulation point of \(\boldsymbol{\lambda}'\), and hence, without loss of generality, we may assume \(d_{i}<1\) for all \(i\).

If (i) holds, then by Theorem \ref{nbd}(c') the sequence \(\boldsymbol{d}\) is a diagonal of \(E\).
If (ii) holds, then \(\delta^{L}>0\) since \((\mathfrak d_{0}')\) does not hold. By Lemma \ref{nbdL}  the sequence \(\boldsymbol{d}\) is a diagonal of \(E\). Finally, suppose that (iii) holds, and \(\lambda_{j}<1\) for all \(j\). That \((\mathfrak{d}_{0}')\) does not hold implies that \(\delta_{0}>0\). Hence we may apply Theorem \ref{csuff}(ii) to deduce that \(\boldsymbol{d}\) is a diagonal of \(E\).
\end{proof}

\begin{thm}\label{ppp}
Let $E$ be a self-adjoint operator on $\mathcal H$ such that its spectrum satisfies 
\[
\sigma(E) \subset [0, 1]
\qquad\text{and}\qquad
\sigma_{ess}(E)=\{0,1\}.
\]
Suppose that $f_{\boldsymbol \lambda}(1/2) <\infty$, where
 $\boldsymbol\lambda=(\lambda_{j})_{j\in J}$ is the list of all eigenvalues of $E$ with multiplicity. 
Let $\boldsymbol{d} = (d_{i})_{i\in \N}$ be a sequence in $[0,1]$. 

\noindent {\rm (Necessity)} If $\boldsymbol{d}$  is a diagonal of $E$, then
\begin{align}
\label{ppp0}
\#|\{i: d_i=0\}| & \le \#|\{j: \lambda_j=0\}|,\\
\#|\{i: d_i=1\}| & \le \#|\{j: \lambda_j=1\}|,
\label{ppp1}
\end{align}
and one of the following two conditions holds:
\begin{enumerate}
\item
$f_{\boldsymbol d}(1/2) =\infty$, or
\item $f_{\boldsymbol d}(1/2) <\infty$, $\boldsymbol d$ has accumulation points at $0$ and $1$, and 
\begin{align}
\label{ppp2}
f_{\boldsymbol\lambda}(\alpha)  &\leq f_{\boldsymbol{d}}(\alpha)  & \text{for all }\alpha\in(0,1),\\
\label{ppp3}f_{\boldsymbol\lambda}'(\alpha)-f_{\boldsymbol{d}}'(\alpha)  & \equiv 0 \mod 1 & \text{for some }\alpha\in(0,1).
\end{align}
\end{enumerate}

In addition, the operator $E$ decouples at a point $\alpha\in (0,1)$ if $f_{\boldsymbol\lambda}(\alpha)  = f_{\boldsymbol{d}}(\alpha)$.

\noindent {\rm (Sufficiency)}  Conversely, if \eqref{ppp0} and \eqref{ppp1} hold and either 
\begin{itemize}
\item (i) holds, or
\item (ii) holds, and $0<\lambda_j<1$ for all $j$,
\end{itemize}
then $\boldsymbol{d}$  is a diagonal of $E$.
\end{thm}

\begin{proof} (Necessity) Suppose that \(\boldsymbol{d}\) is a diagonal of \(E\). Both \eqref{ppp0} and \eqref{ppp1} follow from the fact that, since \(\sigma(E)\subset[0,1]\), any basis element corresponding to a diagonal term \(d_{i}\in\{0,1\}\) must be an eigenvector. If \(f_{\boldsymbol{d}}(1/2)=\infty\), then there is nothing to prove. In the case that \(f_{\boldsymbol{d}}(1/2)<\infty\), then \eqref{ppp2}, \eqref{ppp3}, and the additional statement about decoupling follow from Theorem \ref{nec}.

(Sufficiency) Suppose \eqref{ppp0} and \eqref{ppp1} hold. Suppose (i) holds. We can decompose \(E=\mathbf{0}\oplus E'\oplus \mathbf{I}\) where \(\mathbf{0}\) is the zero operator on a space of dimension \(\#|\{i : d_{i}=0\}|\), and \(\mathbf{I}\) is the identity operator on a space of dimension \(\#|\{i: d_{i} = 1\}|\). Moreover, we may assume \(\{0,1\}\subset\sigma_{ess}(E')\). By Theorem \ref{tomilov} the sequence \(\boldsymbol{d}':=(d_{i})_{d_{i}\in(0,1)}\) is a diagonal of \(E'\), and hence \(\boldsymbol{d}\) is a diagonal of \(E\). Next, suppose (ii) holds, and $0<\lambda_j<1$ for all $j$. Note that \eqref{ppp0} and \eqref{ppp1} imply that \(0<d_{i}<1\) for all \(i\in\N\), and hence \(\boldsymbol{d}\) is a diagonal of \(E\) by Theorem \ref{archaic}.
\end{proof}

\begin{figure}
\resizebox{6.5in}{!}{
\begin{tikzpicture}[node distance=2cm]

\node (input) [io] {Input: self-adjoint operator $E$ such that $\#|\sigma_{ess}(E)| =2 $ and a sequence $(d_{i})$};
\node (pro7) [process, below of=input] {Normalize $E$ so that $\sigma_{ess}(E) = \{0,1 \} $ };
\node (dec4) [decision, below of=pro7, yshift=-1cm] {$f_{\boldsymbol \lambda}(1/2)<\infty$};
\node (no2) [no, below of=dec4, yshift=-0.5cm] {$(d_{i})$ is not a \\ diagonal of $E$};
\node (pro8) [decision, right of=dec4, xshift=3cm] {(Necesssity) Theorem \ref{mainthm2.}};
\node (dec5) [decision, below of=pro8, yshift=-2.5cm] {$\delta_L=0$ and $0 \not\in \sigma^+_{ess}(E)$};
\node (pro9) [process, right of=dec5, xshift=3.5cm] {Prune lower part of $E$};
\node (pro9b) [kernel, below of=pro9] {Apply 1pt algorithm to upper part of $E$};
\draw [arrow] (pro9) -- (pro9b);
\node (dec6) [decision, below of=dec5, yshift=-2.5cm] {$\delta_U=0$ and $1 \not\in \sigma^-_{ess}(E)$};
\node (pro10) [process, right of=dec6, xshift=3.5cm] {Prune upper part of $E$};
\node (pro10b) [kernel, below of=pro10] {Apply 1pt algorithm to lower part of $E$};
\draw [arrow] (pro10) -- (pro10b);
\node (pro11) [stop, below of=dec6, yshift=-1.5cm] {$(d_{i})$ is a \\ diagonal of $E$};

\node (pro12) [decision, left of=dec4, xshift=-3cm] {(Necessity) Theorem \ref{pmt}};
\node (dec7) [decision, below of=pro12, yshift=-2.5cm] {$(\mathfrak d_\alpha)$ holds $\alpha \in (0,1)$};
\node (pro13) [process, right of=dec7, xshift=3cm] {\textbf{Split}  $E$ at $\alpha$ into two compact parts};
\node (pro13b) [kernel, below of=pro13] {Apply 1pt algorithm to each part of $E$};
\draw [arrow] (pro13) -- (pro13b);
\node (dec8) [decision, below of=dec7, yshift=-2.5cm] {$(\mathfrak d_0)$ or $(\mathfrak d_1)$ holds};
\node (pro14) [stop, below of=dec8, yshift=-1.5cm] {$(d_{i})$ is a \\ diagonal of $E$};
\node (pro15) [process, right of=dec8, xshift=3cm] {Split $E$ into compact and Blaschke parts$^\S$};
\node (pro16) [kernel, below of=pro15, yshift=-0.5cm] {Compact kernel problem};
\node (pro17) [process, below of=pro15, yshift=-1.5cm] {Theorem \ref{pmtp3.} applicable};
\node (dec10) [decision, below of=pro17, yshift=-1cm] {$(\mathfrak d_0)$ and $(\mathfrak d_1)$ hold};
\node (pro18) [process, right of=dec10, xshift=3cm] {Split $E$ into compact and Blaschke parts};
\node (pro19) [kernel, left of=dec10,  xshift=-3cm] {kernel problem for operator as in Theorem \ref{pmtp3.}};
\node (pro20) [kernel, right of=pro18, xshift=3.5cm, yshift=0.5cm] {Compact kernel problem};
\node (pro21) [kernel, right of=pro18, xshift=3.5cm, yshift=-0.5cm] {Kernel problem for operator as in Theorem \ref{ppp}};

\draw [arrow] (input) -- (pro7);
\draw [arrow] (pro7) -- (dec4);

\draw [arrow] (dec4) -- node [yshift=0.25cm] {yes} (pro12);
\draw [arrow] (dec4) -- node [yshift=0.25cm] {no} (pro8);
\draw [arrow] (pro8) -- node [xshift=0.45cm] {yes} (dec5);
\draw [arrow] (dec5) -- node [xshift=0.45cm] {no} (dec6);
\draw [arrow] (dec5) -- node [yshift=0.25cm] {yes} (pro9);
\draw [arrow] (dec6) -- node [xshift=0.45cm] {no} (pro11);
\draw [arrow] (dec6) -- node [yshift=0.25cm] {yes} (pro10);

\draw [arrow] (pro12) -- node [xshift=0.45cm] {yes} (dec7);
\draw [arrow] (dec7) -- node [yshift=0.25cm]{yes} (pro13);
\draw [arrow] (dec7) -- node [xshift=0.45cm]{no} (dec8);

\draw [arrow] (dec8) -- node [xshift=0.45cm] {no} (pro14);
\draw [arrow] (dec8) -- node [yshift=0.25cm] {yes} (pro15);
\draw [arrow] (pro15) -- (pro16);
\draw [arrow] (pro17) -- (dec10);

\draw [arrow] (pro8) -- node [yshift=0.25cm]{no} (no2);
\draw [arrow] (pro12) -- node [yshift=0.25cm]{no} (no2);

\draw [arrow] (dec10) -- node [yshift=0.25cm]{no} (pro19);
\draw [arrow] (dec10) -- node [yshift=0.25cm]{yes} (pro18);

\draw [arrow] (pro18) -- (pro20);
\draw [arrow] (pro18) -- (pro21);

\end{tikzpicture}
}
\caption{Algorithm for operators with 2-point essential spectrum. In the special case the splitting marked by $^\S$ might result into two compact operators for which 1pt Algorithm and Theorem \ref{kp} are used instead.}
\label{diagram2}
\end{figure}
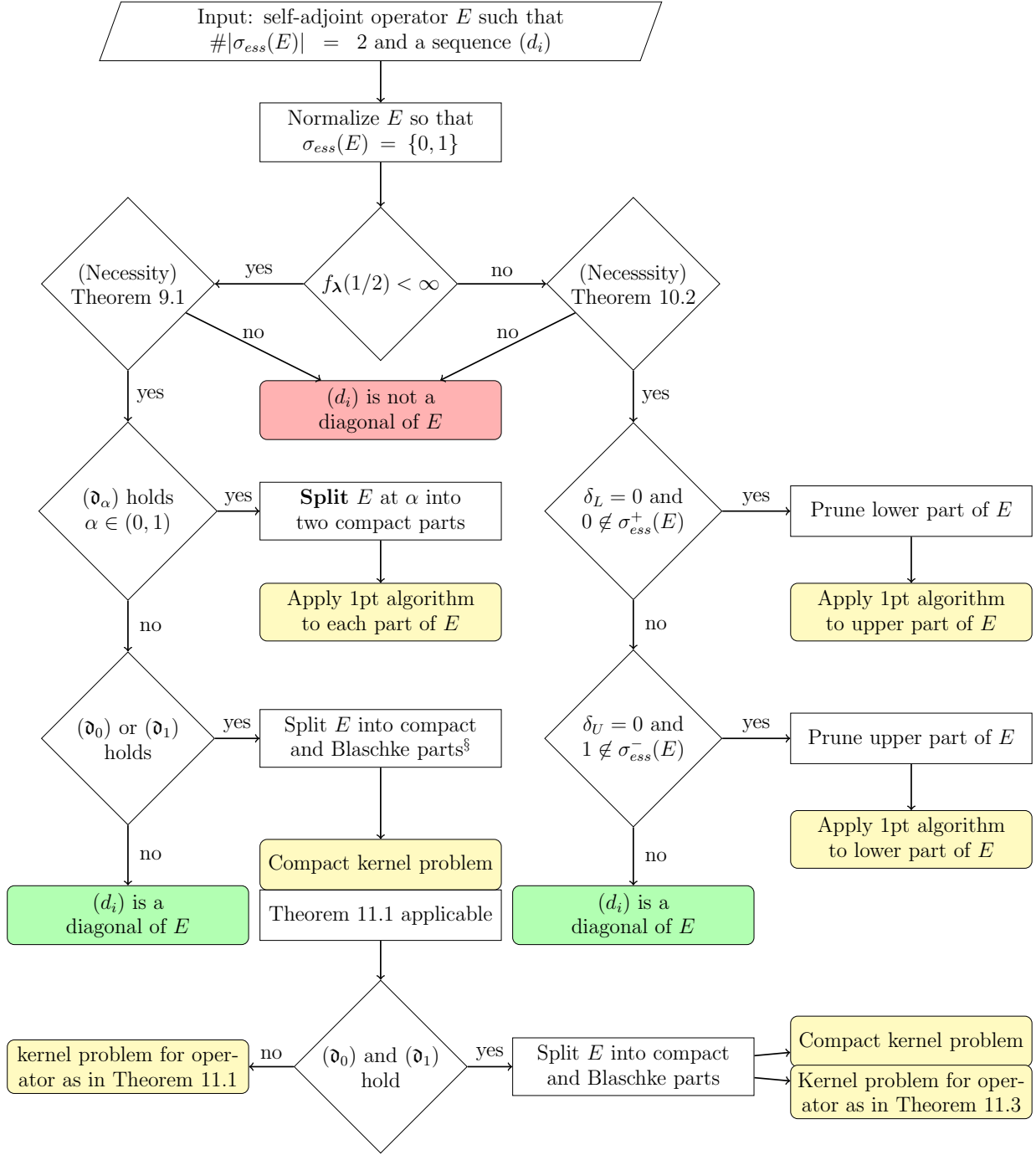

The algorithm for determining whether a sequence $\boldsymbol d$ is a diagonal of an operator $E$ with $\#|\sigma_{ess}(E)| = 2$ is represented by Figure \ref{diagram2}. This procedure starts by normalizing $E$ so that \eqref{pmt0} holds. Then the algorithm splits into two cases depending whether eigenvalue sequence $\boldsymbol \lambda$ satisfies the Blaschke condition $f_{\boldsymbol \lambda}(1/2)<\infty$, or not.

In the former case, we check whether the exterior majorization \eqref{pmt1} holds and whether one of the conditions (i)--(v) in Theorem \ref{pmt} is satisfied. If not, then $\boldsymbol d$ is not a diagonal of $E$. Once the necessity conditions in Theorem \ref{pmt} are satisfied, we check whether $(\mathfrak d_\alpha)$ holds for some $\alpha \in (0,1)$. By Remark \ref{pmt6} if one of the conditions (i)--(iv) hold, then $(\mathfrak d_\alpha)$ is false for all $\alpha \in (0,1)$. Hence, if $(\mathfrak d_\alpha)$ holds for some $\alpha \in (0,1)$, then necessarily condition (v) in Theorem \ref{pmt} holds. By Theorem \ref{nec}, if $\boldsymbol d$ is a diagonal of $E$, then the operator $E$ decouples at $\alpha$. Hence, Theorem \ref{isplit} yields a splitting of $E$ at $\alpha$ into operators $E_1$ and $E_2$ such that $(d_{i})_{d_i<\alpha}$ is a diagonal of $E_1$, $(d_{i})_{d_i>\alpha}$ is a diagonal of $E_2$, and \eqref{z102} holds. Conversely, the existence of such splitting implies that $\boldsymbol d$ is a diagonal of $E$. Hence, to determine whether $\boldsymbol d $ is a diagonal of $E$, we look for all possible splittings of $E$ into $E_1$ and $E_2$ at $\alpha$ as above. For each of operator $E_1$ and $E_2$ we use 1pt Algorithm in \cite[Figure 1]{mbjj7} to test whether    
$(d_{i})_{d_i<\alpha}$ and $(d_{i})_{d_i>\alpha}$ are diagonals of $E_1$ and $E_2$, respectively. 

If $(\mathfrak d_\alpha)$ is false for all $\alpha \in (0,1)$, then we check whether $(\mathfrak d_\alpha)$ holds for $\alpha=0$ or $\alpha=1$. If not, then Theorem \ref{pmt} yields that $\boldsymbol d $ is a diagonal of $E$. Hence, we can assume that $(\mathfrak d_\alpha)$ holds for $\alpha=0$ or $\alpha=1$. To fix attention, assume $(\mathfrak d_1)$ holds, but $(\mathfrak d_0)$ does not hold. By Theorem \ref{pmt}, if $\boldsymbol d$ is a diagonal of $E$, then the operator $E$ decouples at $1$. As before, Theorem \ref{isplit} yields a splitting of $E$ at $1$ into operators $E_1$ and $E_2$ such that $(d_{i})_{d_i<1}$ is a diagonal of $E_1$, $(d_{i})_{d_i>1}$ is a diagonal of $E_2$, and \eqref{z102} holds for $\alpha=1$. Conversely, the existence of such splitting implies that $\boldsymbol d$ is a diagonal of $E$.
While the operator $E_2$ is a shift of compact operator, the operator $E_1$ might or might not have two points in its essential spectrum depending whether $1$ belongs to the left essential spectrum $\sigma_{ess}^-(E)$. In the special case when $E_1$ has only one point in its essential spectrum, we use the 1pt Algorithm in \cite[Figure 1]{mbjj7} and Theorem \ref{kp} to test whether    
$(d_{i})_{d_i<1}$ and $(d_{i})_{d_i>1}$ are diagonals of $E_1$ and $E_2$, respectively. This corresponds to symbol ${}^\S$ in Figure \ref{diagram2}. In the generic case when $E_1$ has two points in its essential spectrum,  Theorem \ref{pmtp3.} is used to test whether  $(d_{i})_{d_i<1}$ is a diagonal of $E_1$. Since $(\mathfrak d_\alpha)$ does not hold for any $\alpha\in [0,1)$ and $\delta_1=0$, we can deduce that for the operator $E_1$ the condition $(\mathfrak d'_\alpha)$ does not hold for any $\alpha\in [0,1)$. Note that the sufficiency direction of Theorem \ref{pmtp3.} only applies when all the eigenvalues $\lambda$ of $E_1$ are $<1$. For all splittings where $E_1$ has at least one eigenvalue $\lambda=1$, the algorithm is inconclusive due to the kernel problem for operators as in Theorem \ref{pmtp3.}. Finally, observe that if any splitting of $E$ results in an operator $E_2$ that is finite dimensional, then we apply the classical Schur-Horn theorem instead of Theorem \ref{kp}.
This concludes the analysis of the algorithm when $(\mathfrak d_1)$ holds, but $(\mathfrak d_0)$ does not hold.

The case when $(\mathfrak d_0)$ holds, but $(\mathfrak d_1)$ does not hold follows by applying a symmetric variant of Theorem \ref{pmtp3.}. Finally, we consider the case when both $(\mathfrak d_0)$ and $(\mathfrak d_1)$ hold.
By Theorem \ref{pmt}, if $\boldsymbol d$ is a diagonal of $E$, then the operator $E$ decouples both at $0$ and at $1$. Two successive applications of Theorem \ref{isplit} yield a splitting of $E$ at $0$ and $1$ into operators $E_0$, $E_1$, and $E_2$ such that $(d_{i})_{d_i<0}$ is a diagonal of $E_0$, $(d_{i})_{0<d_i<1}$ is a diagonal of $E_1$, $(d_{i})_{d_i>1}$ is a diagonal of $E_2$, and 
\begin{align*}
\dim \ker ( E) & =  \dim \ker ( E_0) + \dim \ker ( E_1) + \#|\{i: d_i=0\}|,\\
\dim \ker ( \mathbf I - E) &=  \dim \ker ( \mathbf I -E_1) + \dim \ker ( \mathbf I - E_2) + \#|\{i: d_i=1\}|.
\end{align*}
Conversely, the existence of such splitting implies that $\boldsymbol d$ is a diagonal of $E$. Hence, to determine whether $\boldsymbol d $ is a diagonal of $E$, we look for all possible splittings of $E$ into $E_0$, $E_1$, and $E_2$ as above. For operators $E_0$ and $E_2$ we use Theorem \ref{kp} to test whether    
$(d_{i})_{d_i<0}$ and $(d_{i})_{d_i>1}$ are diagonals of $E_0$ and $E_2$, respectively, whereas we use Theorem \ref{ppp} to test whether $(d_{i})_{0<d_i<1}$ is a diagonal of $E_1$. For all splittings when when $E_1$ has at least one eigenvalue $\lambda=0$ or $\lambda=1$, the algorithm is inconclusive due to the kernel problem for operators as in Theorem \ref{ppp}. 

Finally, observe that if any splitting of $E$ results in one of the operators $E_0$, $E_1$, or $E_2$ being finite dimensional, then we apply the classical Schur-Horn theorem instead of Theorems \ref{kp} and \ref{ppp}. Observe that the extreme case when both $E_0$ and $E_2$ are finite dimensional is also possible. However, this is the only scenario when two of the operators $E_0$, $E_1$, or $E_2$ are finite dimensional. This concludes the analysis of the algorithm in Figure \ref{diagram2} when the eigenvalue sequence $\boldsymbol \lambda$ satisfies the Blaschke condition $f_{\boldsymbol \lambda}(1/2)<\infty$.

To analyze the algorithm when the eigenvalue sequence $\boldsymbol \lambda$ satisfies non-Blaschke condition $f_{\boldsymbol \lambda}(1/2)=\infty$, we employ the following lemma.

\begin{lem}\label{prune13}
Let $E$ be a self-adjoint operator on $\mathcal H$. Let $\boldsymbol\lambda$ be the list of all eigenvalues of $E$ with multiplicity. Suppose that 
\[
\sigma_{ess}(E)=\{0,1\} \qquad\text{and}\qquad
f_{\boldsymbol \lambda}(1/2)=\infty.
\]
Let $\boldsymbol{d} = (d_{i})_{i\in \N}$ be such that \eqref{mainthm2.1}, \eqref{mainthm2.2}, and one of the four conditions (i)--(iv) in Theorem \ref{mainthm2.} hold. Suppose that 
\[
\delta^U=0 \qquad\text{and} \qquad 1 \not\in \sigma^-_{ess}(E).
\]
Let $(E',\boldsymbol d')$ be a pruning of $(E,\boldsymbol d)$ at 1. Let $\boldsymbol\lambda'$ be the list of all eigenvalues of $E'$ with multiplicity. Then, $\boldsymbol{d}$ is a diagonal of $E$ $\iff$ $\boldsymbol{d}'$ is a diagonal of $E'$.
\end{lem}

\begin{proof} 
Suppose first that $\boldsymbol{d}$ is a diagonal of $E$. Since $\delta^U=0$, the operator $E$ decouples at $1$ with respect to $\boldsymbol d$ as in Definition \ref{decouple}. That is,
\[
\mathcal H_0=\overline{\lspan }\{ e_i: d_i \le 1 \}\qquad\text{and}\qquad \mathcal H_1=\overline{\lspan} \{ e_i: d_i >1 \},
\]
are invariant subspaces of $E$ and
\[
\sigma(E|_{\mathcal H_0}) \subset (-\infty,1]\qquad\text{and}\qquad\sigma(E|_{\mathcal H_1}) \subset [1,\infty).
\]
In particular, the operator $E|_{\mathcal H_0}$, which is a restriction of $E'$, has diagonal $\boldsymbol d'$. The operators $E'$ and $E|_{\mathcal H_0}$ are compact and may only differ by the multiplicity of the eigenvalue 1. That is, the eigenvalue list of $E|_{\mathcal H_0}$ is the same as that of $E'$ with the exception that it might contain fewer eigenvalues equal to 1. If $E'$ equals $E|_{\mathcal H_0}$, then $\boldsymbol{d}'$ is a diagonal of $E'$. Otherwise note that the value of $\sigma_+$ in Theorem \ref{intropv3} for the pair $(E',\boldsymbol d')$ is positive and greater than or equal to that of for the pair $(E|_{\mathcal H_0},\boldsymbol d')$. By the necessity direction of Theorem \ref{intropv3}, we see that \eqref{p5v3}--\eqref{p2v3} and \eqref{p4v3} also hold for the pair $(E',\boldsymbol d')$. Since $f_{\boldsymbol \lambda}(1/2)=\infty$, we have $(\boldsymbol \lambda')_+ \not\in \ell^1$. Hence,
if $\boldsymbol d_+ \in \ell^1$, then the values of $\sigma_+$ and $\sigma_-$ for the pair $(E|_{\mathcal H_0},\boldsymbol d')$ are infinite. Thus, \eqref{p3v3} also holds for the pair $(E',\boldsymbol d')$, and by the sufficiency direction of Theorem \ref{intropv3} the sequence $\boldsymbol d'$ is a diagonal of $E'$.

Conversely, suppose that $\boldsymbol d'$ is a diagonal of $E'$. By a special case of Theorem \ref{kp}, a characterization of diagonals of positive compact operators with trivial kernel due to Kaftal and Weiss \cite{kw}, $(d_{i})_{d_i>1}$ is a diagonal of the restriction of $E$ to the range of $\pi(1,\infty)$. Hence, $\boldsymbol d$ is a diagonal of $E$.
\end{proof}

In the case when $f_{\boldsymbol \lambda}(1/2)=\infty$, we start by checking whether all necessary conditions \eqref{mainthm2.1}--\eqref{mainthm2.3} and one of the conditions (i)--(iv) in Theorem \ref{mainthm2.} are satisfied. If not, then $\boldsymbol d$ is not a diagonal of $E$. In the case that both $\delta^U=0$ and $1 \not \in \sigma^-_{ess}(E)$ hold, then we define $(E',\boldsymbol d')$ to be a pruning of $(E,\boldsymbol d)$ at 1. Lemma \ref{prune13} yields that $\boldsymbol d$ is a diagonal of $E$ if and only if $\boldsymbol d'$ is a diagonal of $E'$. Since $E'$ is a compact operator, we apply the 1pt algorithm to test whether $\boldsymbol d'$ is a diagonal of $E'$. Likewise, if both $\delta^L=0$ and $0 \not \in \sigma^+_{ess}(E)$ hold, then we use a symmetric variant of Lemma \ref{prune13} to justify the pruning of $(E,\boldsymbol d)$ at 0. This leaves out the final case when \eqref{mainthm2.4} holds, where Theorem \ref{mainthm2.} yields that $\boldsymbol d$ is a diagonal of $E$. This concludes the analysis of the 2pt algorithm in Figure \ref{diagram2}.

\section{Kernel problem for operators with 2-point essential spectrum} \label{S23}

In this section we discuss the kernel problem for two classes of operators as Theorems \ref{pmtp3.} and \ref{ppp}. In particular, we show that an operator as in Theorem \ref{ppp} needs to satisfy additional necessary conditions in the analogy to compact positive operators in Theorem \ref{kp}.

\begin{thm}\label{trp} Let $E$ be a positive operator on a Hilbert space $\Hil$ with $\sigma(E)\subset [0,1]$. Let $(e_{i})_{i\in I}$ be an orthonormal basis of $\Hil$.
Define the sequence $\boldsymbol{d}=(d_{i})_{i\in I}$ by $d_{i} = \langle Ee_{i},e_{i}\rangle$. If $f_{\boldsymbol d}(\alpha)<\infty$ for some $\alpha \in (0,1)$,
 then $E$ is diagonalizable. Moreover, if $\boldsymbol\lambda=(\lambda_{i})_{i\in I}$ is the eigenvalue list (with multiplicity) of $E$, then 
\begin{align}\label{trp1}
f_{\boldsymbol{\lambda}}(\alpha)
\le 
 f_{\boldsymbol d}(\alpha)
&\qquad\text{for all }\alpha\in(0,1),\\
\label{trp2}f_{\boldsymbol{d}}'(\alpha)-f_{\boldsymbol\lambda}'(\alpha)\in\Z
&\qquad\text{for a.e. }\alpha\in(0,1).
\end{align}
In addition, we have
\begin{align}\label{trp3}
 \#|\{i: d_i=0\}| +\liminf_{\alpha \searrow 0} \frac{f_{\boldsymbol{d}}(\alpha)-
 f_{\boldsymbol \lambda}(\alpha)}{\alpha} \ge \#|\{j: \lambda_j=0\}| \ge  \#|\{i: d_i=0\}|,
 \\
 \label{trp4}
 \#|\{i: d_i=1\}| +\liminf_{\alpha \nearrow 1} \frac{f_{\boldsymbol{d}}(\alpha)-
 f_{\boldsymbol \lambda}(\alpha)}{1-\alpha} \ge \#|\{j: \lambda_j=1\}|  \ge \#|\{i: d_i=1\}|.
\end{align}
\end{thm}

\begin{remark}

Conditions \eqref{trp3} and \eqref{trp4} are analogues of the majorization inequality \eqref{kp4a} in Theorem \ref{kp}, albeit expressed in terms of majorization functions. Indeed, \eqref{kp4a} can be equivalently stated using the Lebesgue majorization as
\[
\liminf_{\alpha \searrow 0} \frac{\delta(\alpha,\boldsymbol \lambda,\boldsymbol d)}{\alpha} \ge z.
\]
Likewise, \eqref{kp4b} is equivalent to existence of $\alpha_0>0$ such that
\[
\delta(\alpha,\boldsymbol \lambda,\boldsymbol d) \ge \alpha z
\qquad\text{for } 0<\alpha<\alpha_0.
\]
Since we do not need these facts, the proof is left to the reader.
\end{remark}

\begin{proof} In light of Lemma \ref{trap}, it remains to show conclusions \eqref{trp3} and \eqref{trp4}. By symmetry consideration, it suffices to show only \eqref{trp3}.
If $d_i=0$ or $1$ for some $i\in I$, then $e_i$ is an eigenvector of $E$ with eigenvalue $0$ or $1$, resp. Hence, without loss of generality we can assume that $d_i\in (0,1)$ for all $i\in I$. We proceed as in the proof of Lemma \ref{trap}.

By the spectral theorem there is a projection valued measure $\pi$ such that 
\[
E = \int_{[0,1]}\lambda\,d\pi(\lambda).
\]
Fix $\alpha\in(0,1)$ and set
\[
K = \int_{[0,\alpha)}\lambda\,d\pi(\lambda)\quad\text{and}\quad 
T = \int_{[\alpha,1]}(1-\lambda)\,d\pi(\lambda).\]
Let $P$ be the projection given by $P =\pi([\alpha, 1])$. For each $i\in I$ set $k_{i} = \langle Ke_{i},e_{i}\rangle, p_{i} = \langle Pe_{i},e_{i}\rangle$, and $t_{i} = \langle Te_{i},e_{i}\rangle$. Since $E=K+P-T$ we also have 
\begin{equation}\label{trp5}d_{i} = k_{i}+p_{i}-t_{i}.\end{equation}

We claim that $K$ and $T$ are trace class. Since $K$ and $T$ are positive and $(e_{i})_{i\in I}$ is a Parseval frame, it is enough to show that $\sum_{i\in I}(k_{i}+t_{i})<\infty$. 

Next, we observe
\begin{equation}\label{trp6}T=\int_{[\alpha,1]}(1-\lambda)\,d\pi(\lambda)\leq \int_{[\alpha,1]}(1-\alpha)\,d\pi(\lambda)= (1-\alpha)P\end{equation}
and
\begin{equation}\label{trp7}K = \int_{(0,\alpha)}\lambda\,d\pi(\lambda)\leq \int_{(0,\alpha)}\alpha\,d\pi(\lambda)= \alpha(\mathbf{I}-P) - \alpha \pi(\{0\}).\end{equation}

Let $J=\{i\in I: d_i < \alpha\}$. Combing $E=K+P-T$  and \eqref{trp6} yields $E \ge K+ \frac{\alpha}{1-\alpha}T$. By \eqref{trp5} we deduce that  $d_{i}\geq k_{i}+\frac{\alpha}{1-\alpha}t_{i}$, and thus
\begin{equation}\label{trp8}\sum_{i\in J}\left(k_{i}+\frac{\alpha}{1-\alpha}t_{i}\right)\leq \sum_{i\in J}d_{i} <\infty.\end{equation}
From \eqref{trp7} we deduce that
 \[
\mathbf I - E = T-K +(\mathbf I - P)  \ge T +  \frac{1-\alpha}{\alpha} K + \pi(\{0\}).
\]
Thus, 
$1-d_{i} \geq \frac{1-\alpha}{\alpha}k_{i}+t_{i} + z_i$, where $z_i=\langle \pi(\{0\})e_i,e_i \rangle$. This yields
\begin{equation}\label{trp9}\sum_{i\in I \setminus J }\left(\frac{1-\alpha}{\alpha}k_{i}+t_{i} +z_i\right)\leq \sum_{i\in I \setminus J} (1-d_{i})
 <\infty.\end{equation}
Combining \eqref{trp8} and \eqref{trp9} yields that $K$ and $T$ are trace class.

By the spectral theorem for compact operators, there is an orthonormal basis $(f_{i})_{i\in J_{1}}$ for $\ran(\mathbf{I}-P)$ consisting of eigenvectors of $K$, with associated eigenvalues $(\lambda_{i})_{i\in J_{1}}$. There is also an orthonormal basis $(f_{i})_{i\in J_{2}}$ for $\ran(P)$ consisting or eigenvectors of $T$ with associated eigenvalues $(1-\lambda_{i})_{i\in J_{2}}$. Since $K$ and $T$ are trace class, both $(\lambda_{i})_{i\in J_{1}}$ and $(1-\lambda_{i})_{i\in J_{2}}$ are summable. For $i\in J_{1}$ we have $(P-T)f_{i} = 0$ and thus $Ef_{i} = \lambda_{i}f_{i}$. For $i\in J_{2}$ we have $Kf_{i} = 0$ so that $Ef_{i} = \lambda_{i}$. 
Therefore, $(f_{i})_{i\in J_{1}\cup J_{2}}$ is an orthonormal basis for $\Hil$ consisting of eigenvectors of $E$ with associated eigenvalues $(\lambda_{i})_{i\in J_{1}\cup J_{2}}$. This shows that $E$ is diagonalizable. 

From the definitions of $K$ and $T$ we see that $\lambda_{i}<\alpha$ for $i\in J_{1}$ and $\lambda_{i}\geq \alpha$ for $i\in J_{2}$. Thus,

\[C_{\boldsymbol\lambda}(\alpha)=\sum_{\lambda_{i}<\alpha}\lambda_{i} = \sum_{i\in J_{1}}\lambda_{i}<\infty\quad\text{and}\quad D_{\boldsymbol\lambda}(\alpha)=\sum_{\lambda_{i}\geq \alpha}(1-\lambda_{i})=\sum_{i\in J_{2}}(1-\lambda_{i})<\infty.\]
This implies $f_{\boldsymbol\lambda}(\alpha)<\infty$. Moreover, using \eqref{trp8} and \eqref{trp9} we have
\begin{align*}
f_{\boldsymbol\lambda}(\alpha) & = (1-\alpha)\sum_{i\in J_{1}}\lambda_{i} + \alpha\sum_{i\in J_{2}}(1-\lambda_{i}) = (1-\alpha)\tr(K)+\alpha\tr(T) = (1-\alpha)\sum_{i\in I}k_{i} + \alpha\sum_{i\in I}t_{i}\\
 & = (1-\alpha)\sum_{i\in J}\left(k_{i} + \frac{\alpha}{1-\alpha}t_{i}\right) + \alpha\sum_{i\in I \setminus J }\left(\frac{1-\alpha}{\alpha}k_{i} + t_{i}\right)\\
 & \leq 
  (1-\alpha)\sum_{i\in J}d_i + \alpha\sum_{i\in I \setminus J }( 1 - d_i) - \alpha\sum_{i\in I \setminus J } z_i 
  =  f_{\boldsymbol{d}}(\alpha) - \alpha\sum_{i\in I \setminus J } z_i,
\end{align*}
which shows \eqref{trp1}. Moreover,
\[
\frac{f_{\boldsymbol{d}}(\alpha) - f_{\boldsymbol\lambda}(\alpha)}{\alpha} \ge \sum_{d_i > \alpha } z_i \to \sum_{i\in I} z_i = \dim \ker E
\qquad\text{as }\alpha \searrow 0.
\]
This proves \eqref{trp3}.
\end{proof}

Theorem \ref{trp} could be extended to the class of operators as in Theorem \ref{pmtp3.} by replacing the conclusion \eqref{trp4} with
\[
 \#|\{i: d_i=1\}| +\liminf_{\alpha \nearrow 1} \frac{f_{\boldsymbol{d}}(\alpha)-
 f_{\boldsymbol \lambda}(\alpha)}{1-\alpha} +\delta_0 \ge \#|\{j: \lambda_j=1\}|  \ge \#|\{i: d_i=1\}|.
\]
Since we do not use this fact, we skip its proof which is a more complicated variant of Theorem \ref{trp}. Besides, it seems unlikely that \eqref{trp3} and \eqref{trp4} are sufficient conditions for the existence of operators with prescribed diagonals as in Theorem \ref{trp}. Indeed, we expect that the existence of  $\alpha_0>0$ such that
\[
 \#|\{i: d_i=0\}| + \frac{f_{\boldsymbol{d}}(\alpha)-
 f_{\boldsymbol \lambda}(\alpha)}{\alpha} \ge \#|\{j: \lambda_j=0\}| \ge  \#|\{i: d_i=0\}| \qquad\text{for } 0<\alpha<\alpha_0,
 \]
together with analogous statement for \eqref{trp4} are sufficient conditions.
This is in analogy with the kernel problem for positive compact operators in Theorem \ref{kp}, where the gap between the necessary \eqref{kp4a} and sufficient \eqref{kp4b} conditions persists.

\end{document}